\documentclass[a4paper, 11pt]{amsart}

\usepackage[square, numbers, comma]{natbib} % Use the natbib reference package - read up on this to edit the reference style; if you want text (e.g. Smith et al., 2012) for the in-text references (instead of numbers), remove 'numbers'. Add sort&compress to sort references the way they appear in the bibliography, and compress references which are close in numbers.

\usepackage[usenames,dvipsnames,svgnames,table]{xcolor}
\usepackage{amsmath,amsthm,amssymb,amssymb,esint,verbatim,tabularx,graphicx}
\usepackage{fancyhdr}

\usepackage{enumerate}

\usepackage{amsmath}
\usepackage{fancyhdr}
\usepackage{epic}
\usepackage{pgf,tikz}
\usetikzlibrary{positioning}
\usetikzlibrary{arrows.meta} %various arrows
\usetikzlibrary{calc} %positions of arrows
\usepackage[utf8]{inputenc}
\usepackage{color}
\usepackage{hyperref}
\usepackage{verbatim}
\usepackage{pdfpages}

\usepackage{mathrsfs} %extra curly letters

\usepackage{cancel} % to draw a diagonal to "cancel" or simplify formulae \cancel{} \bcancel{} \xcancel{}

\hypersetup{urlcolor=blue, colorlinks=true} % Colors hyperlinks in blue - change to black if annoying

\newtheorem{theorem}{Theorem}[section]

\newtheorem{lemma}[theorem]{Lemma}
\newtheorem{proposition}[theorem]{Proposition}
\newtheorem{corollary}[theorem]{Corollary}
\newtheorem{definition}{Definition}[section]

\theoremstyle{remark}
\newtheorem{remark}[theorem]{Remark}
\theoremstyle{definition}

\numberwithin{equation}{section}

\newcommand{\R}{\ensuremath{\mathbb{R}}}
\newcommand{\N}{\ensuremath{\mathbb{N}}}

\newcommand{\Levy}{\ensuremath{\mathcal{L}}}

\newcommand{\Operator}{\ensuremath{\mathfrak{L}}}

\newcommand{\dual}{\ensuremath{\gamma}}

\newcommand{\veps}{\varepsilon}

\newcommand{\dd}{\,\mathrm{d}}
\newcommand{\diver}{\mathrm{div}}
\newcommand{\dell}{\partial}

\newcommand{\e}{\textup{e}}

\DeclareMathOperator*{\esssup}{ess \, sup}

\DeclareMathOperator{\sgn}{\textup{sign}}

\definecolor{darkblue}{rgb}{0.05, .05, .65}
\definecolor{darkgreen}{rgb}{0.1, .65, .1}
\definecolor{darkred}{rgb}{0.8,0,0}

\begin{document}

\title[Smoothing effects, Green functions, and functional inequalities]{Nonlocal nonlinear diffusion equations.\\ Smoothing effects, Green functions, \\ and functional inequalities}

\author[M.~Bonforte]{Matteo Bonforte}
\address[M. Bonforte]{Departamento de Matem\'aticas\\
Universidad Aut\'onoma de Madrid (UAM)\\
Campus de Cantoblanco, 28049 Madrid, Spain}
\email[]{matteo.bonforte\@@{}uam.es}
\urladdr{http://verso.mat.uam.es/~matteo.bonforte/}

\author[J.~Endal]{J\o rgen Endal}
\address[J.~Endal]{Department of Mathematical Sciences\\
Norwegian University of Science and Technology (NTNU)\\
N-7491, Trondheim, Norway\\
and\\
Departamento de Matem\'aticas\\
Universidad Aut\'onoma de Madrid (UAM)\\
Campus de Cantoblanco, 28049 Madrid, Spain}
\email[]{jorgen.endal\@@{}ntnu.no}
\urladdr{https://folk.ntnu.no/jorgeen/}

\keywords{Nonlinear degenerate parabolic equations, boundedness estimates, $L^1$--$L^\infty$-smoothing effects, Green functions, heat kernels, Gagliardo-Nirenberg-Sobolev inequalities, Moser iteration,  porous medium, Laplacian, fractional Laplacian, nonlocal operators, existence}

\subjclass[2010]{35K55,   	%Nonlinear parabolic equations
35K65, %Degenerate parabolic equations
35A01, %Existence problems for PDEs: global existence, local existence, non-existence
%35A02, %Uniqueness problems
%35B30, %Dependence of solutions on initial and boundary data
35B45,   	%A priori estimates in context of PDEs
%35D30, %Weak solutions
%35K15, %Initial value problems for second-order parabolic equations
35R09, %Integro-partial differential equations
35R11. %Fractional partial differential equations
%65M06, %Finite difference methods
%65M12, %Stability and convergence of numerical methods
%76S05%Flows in porous media; filtration
}

\begin{abstract}
We establish boundedness estimates for solutions of generalized porous medium equations of the form
$$
\partial_t u+(-\mathfrak{L})[u^m]=0\quad\quad\text{in $\mathbb{R}^N\times(0,T)$},
$$
where $m\geq1$ and $-\mathfrak{L}$ is a linear, symmetric, and nonnegative operator. The wide class of operators we consider includes, but is not limited to, L\'evy operators. Our quantitative estimates take the form of precise $L^1$--$L^\infty$-smoothing effects and absolute bounds, and their proofs are based on the interplay between a dual formulation of the problem and estimates on the Green function of $-\mathfrak{L}$ and $I-\mathfrak{L}$.

In the linear case $m=1$, it is well-known that the $L^1$--$L^\infty$-smoothing effect, or ultracontractivity, is equivalent to Nash inequalities. This is also equivalent to heat kernel estimates, which imply the Green function estimates that represent a key ingredient in our techniques.

We establish a similar scenario in the nonlinear setting $m>1$. First, we can show that operators for which ultracontractivity holds, also provide $L^1$--$L^\infty$-smoothing effects in the nonlinear case. The converse implication is not true in general. A counterexample is given by $0$-order L\'evy operators like $-\mathfrak{L}=I-J\ast$. They do not regularize when $m=1$, but we show that surprisingly enough they do so when $m>1$, due to the convex nonlinearity. This reveals a striking property of nonlinear equations: the nonlinearity allows for better regularizing properties, almost independently of the linear operator.

Finally, we show that smoothing effects, both linear and nonlinear, imply families of inequalities of Gagliardo-Nirenberg-Sobolev type, and we explore equivalences both in the linear and nonlinear settings through the application of the Moser iteration.
\end{abstract}

\maketitle

%\newpage
%\small
\tableofcontents
%\newpage

\section{Introduction and main results}
In this paper, we consider solutions of generalized porous medium equations \cite{Vaz17}:
\begin{equation}\label{GPME}\tag{\textup{GPME}}
\begin{cases}
\dell_tu+(-\Operator)[u^m]=0 \qquad\qquad&\text{in $Q_T:=\R^N\times(0,T)$,}\\
u(\cdot,0)=u_0 \qquad\qquad&\text{on $\R^N$,}
\end{cases}
\end{equation}
where $m\geq1$, $T>0$, $0\leq u_0\in L^1(\R^N)$, and the operator $-\Operator$ is at least linear, symmetric, nonnegative\footnote{And moreover, densely defined, $\mathfrak{m}$-accretive, and Dirichlet in $L^1(\R^N)$. Basically, we need the comparison principle and $L^p$-decay to hold for solutions of \eqref{GPME}. We refer the reader to Appendix \ref{sec:mAccretiveDirichlet} for further information. Note that the terminology ``Dirichlet operator'' appears in the literature also as ``sub-Markovian operator''. This property is expressing the fact that the operator has to be order preserving.}, and includes L\'evy operators\footnote{That is, operators which are nonnegative at any global nonnegative maximum (usually called the positive maximum principle), see e.g. \cite{Cou64}. When $c>0$, there is (strong) absorption in \eqref{GPME}.} defined for $\psi\in C_\textup{c}^\infty(\R^N)$ and $c\geq 0$ as
\begin{equation}\label{def:LevyOperators}
c\psi(x)-\sum_{i,j=1}^{N}a_{ij}\dell_{x_ix_j}^2\psi(x)-\underbrace{P.V.\int_{\R^N\setminus\{0\} } \big(\psi(x+z)-\psi(x)\big) \dd\mu(z)}_{=:\Levy^\mu[\psi](x)}
\end{equation}
where the real matrix $[a_{ij}]_{i,j=1,\ldots,N}$ is nonnegative and symmetric,  $P.V.$ is the Cauchy principal value, and:
\begin{align}
&\label{muas}\tag{$\textup{H}_{\mu}$} \mu \text{ is a nonnegative symmetric Radon measure on
}\R^N\setminus\{0\}
\text{ satisfying}
\nonumber\\
&\int_{|z|\leq1}|z|^2\dd \mu(z)+\int_{|z|>1}1\dd
\mu(z)<\infty\nonumber.
\end{align}
Important examples are the Laplacian, the fractional Laplacian, sum of onedimensional fractional Laplacians, and so-called convolution type or $0$-order operators given as $-\Operator=I-J\ast$ where $J\geq 0$ satisfies $\|J\|_{L^1(\R^N)}=1$.

Boundedness estimates are the first step on the way to further regularity properties. This was exploited in e.g. \cite{AtCa10} (cf. Theorem 2.2 in \cite{DTEnVa21}), \cite[Section 7]{DPQuRoVa14}, \cite{BoVa14}, \cite[Theorem 1.2]{DPQuRo16}, \cite[Theorem 1.2]{DPQuRoVa17}, \cite[Theorem 1.1]{DPQuRo18}, and \cite[Theorem 1.2]{BrLiSt21}. It is also an important estimate in obtaining uniqueness in $L^1$ for very weak solutions of \eqref{GPME}, see e.g. \cite{BrCr79, DTEnJa17a, DTEnJa17b}. We will therefore focus on such estimates in this paper.

It is well-known since the works of B\'enilan \cite{Ben78} and V\'eron \cite{Ver79} that the parabolic equation $\dell_tu-\Delta[\varphi(u)]=0$ enjoys $L^1$--$L^\infty$-smoothing when $\varphi\in C^1(\R)$ and $\varphi'(r)\geq C|r|^{m-1}$ (see also \cite[Theorem 8.2]{DPQuRoVa17} in the case of the fractional Laplacian $-\Operator=(-\Delta)^{\frac{\alpha}{2}}$, and \cite{Vaz07} in the standard Laplacian case). Let us therefore fix $\varphi(r)=|r|^{m-1}r$. In the linear case ($m=1$), the standard heat equation and the fractional heat equation still enjoy $L^1$--$L^\infty$-smoothing \cite{BaPeSoVa14, BoSiVa17}, but there are cases in which the operator is too weak to ensure such estimates. This can e.g. be seen for the convolution type operators $-\Operator=I-J\ast$ (cf. \cite[Theorem 1.4 and Lemma 1.6]{A-VMaRoT-M10}), where the solutions are as smooth as the initial data. Hence, when the nonlinearity cannot help, the operator needs to be strong enough to provide bounded  solutions. One of our main concerns is therefore the following question:
$$
\textit{Which operators $\Operator$ produce bounded solutions of \eqref{GPME}?}
$$

To provide an answer to this intriguing question we will extend the so-called Green function method to a wide class of operators. Such a method  was developed in a series of papers \cite{BoVa15, BoVa16, BoFiR-O17, BoFiVa18a, BoFiVa18b, BeBoGaGr20, BeBoGrMu21, BoIbIs22} both for operators on bounded domains and on manifolds, including the Euclidean space $\R^N$. The key tool is having at disposal good estimates for the kernel of $(-\Operator)^{-1}$, i.e. the Green function $\mathbb{G}_{-\Operator}$. Now, applying the inverse operator on each side of the PDE in \eqref{GPME} yields the so-called dual equation
$$
0=(-\Operator)^{-1}[\dell_tu]+(-\Operator)^{-1}[(-\Operator)[u^m]]=\mathbb{G}_{-\Operator}\ast\dell_tu+u^m.
$$
Another essential ingredient is the so-called B\'enilan-Crandall (time-monotonicity) estimate
$$
\dell_tu(\cdot,t)\geq-\frac{u(\cdot,t)}{(m-1)t}\qquad\text{in $\mathcal{D}'(\R^N)$,}
$$
which is a weak version of the fact that the map $t\mapsto t^{\frac{1}{m-1}}u(\cdot,t)$ is nondecreasing. This is well-known to be a consequence of the time-scaling and comparison principle for \eqref{GPME}, cf. \cite{BeCr81b, Vaz07}.\footnote{The estimate is purely nonlinear since it degenerates when $m=1$. However, the stronger Aronson-B\'enilan estimate \cite{ArBe79} do hold for the linear case as well, but it relies on the operator itself having space-scaling. We refer the reader e.g. to \cite[Lemma 6.1]{BoSiVa17} and \cite[p. 1270]{DPQuRoVa12}. Thus, the Green function method can indeed hold for particular linear cases.} A combination of the above equations then gives the so-called fundamental upper bound or ``almost representation formula''
\begin{equation}\label{eq:IntroAlmostRepFormula}
u^m(x_0,t)\leq\frac{1}{(m-1)t}\mathbb{G}_{-\Operator}(\cdot-x_0)\ast_x u(\cdot,t).
\end{equation}
The latter name is justified in the sense that the bound is similar to the one given by the representation formula (convolution with the heat kernel) in the linear case $m=1$, where the Green function $\mathbb{G}_{-\Operator}(\cdot-x_0)$ is replaced by the heat kernel $\mathbb{H}_{-\Operator}(\cdot-x_0)$ corresponding to the operator. In both cases, the boundedness estimates follows directly by applying various properties of $\mathbb{G}_{-\Operator}(\cdot-x_0)$ and $\mathbb{H}_{-\Operator}(\cdot-x_0)$. Further details on the proofs can be found in Section \ref{sec:Proofs}.

Our method allows to recover the well-known $L^1$--$L^\infty$-smoothing result, cf. Theorem \ref{thm:L1ToLinfinitySmoothing2} and Figure \ref{fig:ImplicationOfOperatorsAndSobolev},
\begin{equation*}
\|u(\cdot,t)\|_{L^\infty(\R^N)}\lesssim t^{-N\theta_{\alpha}}\|u_0\|_{L^1(\R^N)}^{\alpha\theta_{\alpha}}\qquad\text{for all $t>0$}
\end{equation*}
where $\alpha\in(0,2]$ and $\theta_{\alpha}:=(\alpha+N(m-1))^{-1}$, is respectively valid for the Laplacian $(-\Operator)=(-\Delta)$ \cite{Vaz06, Vaz07, DBGiVe12} and the fractional Laplacian $-\Operator=(-\Delta)^{\frac{\alpha}{2}}$ \cite{DPQuRoVa12} (see also \cite{CoHa16, CoHa21} with $p=2$). An immediate consequence of our approach is that also L\'evy operators $\Operator=\Levy^\mu$ with $\mu$ comparable to the measure of the fractional Laplacian enjoys the same estimate, see Lemma \ref{lem:GreenComparableWithFractionalLaplacian}. It is also interesting to note that we are able to treat operators whose Green functions have different power behaviours. Solutions of \eqref{GPME} with such operators satisfy
\begin{equation*}
\|u(\cdot,t)\|_{L^\infty(\R^N)}\lesssim t^{-N\theta_{\alpha}}\|u_0\|_{L^1(\R^N)}^{\alpha\theta_{\alpha}}+t^{-N\theta_{2}}\|u_0\|_{L^1(\R^N)}^{2\theta_{2}}\qquad\text{for a.e. $t>0$,}
\end{equation*}
see Theorem \ref{thm:L1ToLinfinitySmoothing3}. The examples treated in Section \ref{sec:CombinationsOfAssumptionG_1} are $-\Operator=(-\Delta)+(-\Delta)^{\frac{\alpha}{2}}$, relativistic Schr\"odinger type operators $-\Operator=(\kappa^2I-\Delta)^\frac{\alpha}{2}-\kappa^{\alpha} I$ with $\kappa>0$, and $\Operator$ being the generator of a finite range isotropically symmetric $\alpha$-stable process in $\R^N$ with jumps of size larger than $1$ removed. Finally, if $\mathbb{G}_{-\Operator}\in L^1(\R^N)$ in \eqref{eq:IntroAlmostRepFormula}, then we immediately obtain the following absolute bound (cf. Theorem \ref{thm:AbsBounds} and Figure \ref{fig:ImplicationOfOperatorsAndWeaker}):
\begin{equation}\label{eq:IntroAbsoluteBound}
\|u(\cdot,t)\|_{L^\infty(\R^N)}\lesssim t^{-1/(m-1)}\qquad\text{for a.e. $t>0$.}
\end{equation}
All operators on the form $I-\Operator$ provide such an estimate (Lemma \ref{lem:GreenResolventIntegrable}), and also the operator $-\Operator=(I-\Delta)^{\frac{\alpha}{2}}$ corresponding to the Bessel potential (Lemma \ref{lem:OperatorWithBesselPotential})\footnote{This operator can be written as $I-\Levy^\mu$ for $\mu$ satisfying \eqref{muas}, i.e, on the form \eqref{def:LevyOperators}.}. They are furthermore examples of operators which have better boundedness properties in the nonlinear case than in the linear, see Remark \ref{rem:OperatorWithBesselPotential}.

The Green function method requires the existence of an inverse $(-\Operator)^{-1}$ with a kernel $\mathbb{G}_{-\Operator}$ satisfying suitable estimates. This of course puts a restriction on the class of operators we are able to treat. To remedy this fact, we also develop another approach which consists in considering $\mathbb{G}_{I-\Operator}$ instead, i.e., the Green function associated with the resolvent operator $I-\Operator$. In this case, the inverse always exists, and $\mathbb{G}_{I-\Operator}$ is at least as good as $\mathbb{G}_{-\Operator}$. By rewriting the PDE in \eqref{GPME} to $\dell_tu+(I-\Operator)[u^m]=u^m$, applying $(I-\Operator)^{-1}$, and using the time-monotonicity estimate (associated with $-\Operator$), we obtain the following fundamental upper bound:
\begin{equation}\label{eq:IntroAlmostRepFormula2}
u^m(x_0,t)\leq\bigg(\frac{1}{(m-1)t}+\|u(\cdot,t)\|_{L^\infty(\R^N)}^{m-1}\bigg)\mathbb{G}_{I-\Operator}(\cdot-x_0)\ast_x u(\cdot,t).
\end{equation}
Hence, we see that we have to pay the price of treating an equation with the reaction term $u^m$, which we then have to reabsorb to be able to obtain good estimates in this case. However, note that we can split the estimation of \eqref{eq:IntroAlmostRepFormula2} into two cases:
$$
\|u(\cdot,t)\|_{L^\infty(\R^N)}^{m-1}\leq \frac{1}{(m-1)t} \qquad\text{and}\qquad \|u(\cdot,t)\|_{L^\infty(\R^N)}^{m-1}> \frac{1}{(m-1)t}.
$$
In the first case, we already have the estimate $\|u(\cdot,t)\|_{L^\infty(\R^N)}\lesssim t^{-1/(m-1)}$, while in the other
\begin{equation*}%\label{eq:IntroAlmostRepFormula3}
u^m(x_0,t)\leq 2\|u(\cdot,t)\|_{L^\infty(\R^N)}^{m-1}\mathbb{G}_{I-\Operator}(\cdot-x_0)\ast_x u(\cdot,t),
\end{equation*}
from which we can deduce $\|u(\cdot,t)\|_{L^\infty(\R^N)}\lesssim \|u_0\|_{L^1(\R^N)}$ as long as $\mathbb{G}_{I-\Operator}\in L^p(\R^N)$ with $p\in(1,\infty)$. Hence, the fundamental upper bound \eqref{eq:IntroAlmostRepFormula2} yields
\begin{equation}\label{eq:IntroResolventBoundednessEstimate}
\|u(\cdot,t)\|_{L^\infty(\R^N)}\lesssim t^{-1/(m-1)}+\|u_0\|_{L^1(\R^N)}\qquad\text{for a.e. $t>0$,}
\end{equation}
see Theorem \ref{thm:L1ToLinfinitySmoothing} and Figure \ref{fig:ImplicationOfOperatorsAndNonSobolev}. The operator $-\Operator=\sum_{i=1}^N(-\dell_{x_ix_i}^2)^{\frac{\alpha}{2}}$ provides an important example in this case since $\mathbb{G}_{-\Operator}=\infty$ (for some values of $\alpha$), while $\mathbb{G}_{I-\Operator}\in L^p(\R^N)$. We refer to Lemma \ref{lem:GreenAnisotropicLaplacians} and Remark \ref{ref:GreenAnisotropicLaplacians} for further information. Indeed, this is the first time the Green function method is able to treat this operator. We end this part by also mentioning that L\'evy operators $\Operator=\Levy^\mu$ with $\mu$ such that, for $\alpha\in(0,2)$ and constants $C_1,C_2, C_3>0$,
\begin{equation}\label{eq:IntroBestSmoothingExample}
\frac{C_1}{|z|^{N+\alpha}}\mathbf{1}_{|z|\leq 1}\leq\frac{\dd \mu}{\dd z}(z)\leq \frac{C_2}{|z|^{N+\alpha}}\mathbf{1}_{|z|\leq 1}\qquad\text{and}\qquad \frac{\dd \mu}{\dd z}(z)\leq C_3\mathbf{1}_{|z|>1},
\end{equation}
fall into this case (Lemma \ref{lem:OperatorWithDerivativesAtZero}). The latter fits with the ``usual impression'' in the PDE community regarding the least assumptions expected on nonlocal operators which would produce bounded solutions of \eqref{GPME}. Nevertheless, we were not able to find such a result other places in the literature.

An alternative to the Green function method is the nowadays standard Moser iteration \cite{Mos64, Mos67}, which requires the quadratic form associated to the operator to satisfy Gagliardo-Nirenberg-Sobolev (GNS) and Stroock-Varopoulos inequalities. In the case $-\Operator=(-\Delta)^{\frac{\alpha}{2}}$, we refer to \cite{DPQuRoVa12}. We devote Section \ref{sec:SmoothingAndGNS} to a further discussion on the connections between Green function estimates, heat kernel estimates, and functional inequalities like GNS. In the linear case $m=1$, it is well-known that $L^1$--$L^\infty$-smoothing is equivalent with Nash inequalities (a subfamily of GNS) \cite{Nas58}, and moreover, equivalent with on-diagonal heat kernel $\mathbb{H}_{-\Operator}$ estimates. We present those connections in Theorem \ref{thm:LinearEquivalences}, where we also include---maybe the less-known---\emph{equivalence} with Sobolev inequalities. Since we are interested in Green function estimates, we finally prove that the bound $\mathbb{G}_{-\Operator}\lesssim |x|^{-(N-\alpha)}$ implies the Sobolev inequality. If the Green function exists, it is given by
$$
\mathbb{G}_{-\Operator}(x)=\int_0^\infty\mathbb{H}_{-\Operator}(x,t)\dd t.
$$
Hence, off-diagonal heat kernel bounds is needed to give estimates on the Green function. In other words, we need more information on $\mathbb{H}_{-\Operator}$ than what the previous equivalences give us. The linear panorama is more or less settled, and we move on to the nonlinear case $m>1$. Again, $L^1$--$L^\infty$-smoothing is equivalent with a family of GNS inequalities which is now subcritical since $m>1$. The latter is somehow interesting in the sense that we need a weaker inequality, compared to the linear case, in order to prove $L^1$--$L^\infty$-smoothing through the Moser iteration. However, this inequality is still equivalent with the Sobolev inequality by \cite{BaCoLeS-C95}. This is in contrast to the absolute bound which is equivalent to the Poincar\'e inequality! In general, the latter inequality can only give $L^q$--$L^p$-smoothing estimates through an iteration approach \cite{Gri10, GrMuPo13}, and somehow the Green function method then provides an improvement here (since we indeed reach $L^\infty$-estimates, cf. \eqref{eq:IntroAbsoluteBound}). See the Figures \ref{fig:ImplicationInLinearCase} and \ref{fig:ImplicationInNonlinearCase} for the various connections.

The resolvent approach also offers further interesting insight. Since
$$
\mathbb{G}_{I-\Operator}(x)=\int_0^\infty\mathbb{H}_{I-\Operator}(x,t)\dd t=\int_0^\infty\textup{e}^{-t}\mathbb{H}_{-\Operator}(x,t)\dd t,
$$
even poor on-diagonal heat kernel bounds for $\mathbb{H}_{-\Operator}$ will give $\mathbb{G}_{I-\Operator}\in L^p(\R^N)$. This has at least two consequences: (i) Such estimates for $\mathbb{H}_{-\Operator}$ imply both Green function bounds and also GNS inequalities, which furthermore imply that solutions of \eqref{GPME} (with $m>1$) are bounded whenever the Green function method and/or Moser iteration go through. (ii) If the operator is such that solutions of \eqref{GPME} with $m=1$ are bounded, then also solutions of \eqref{GPME} with $m>1$ are bounded (Theorem \ref{thm:OverviewBoundedness}). The last item corresponds to the ``usual impression'' in the PDE community, but again we were not able to find a good reference for such a statement. The first item provides a clear connection between the Green function method and the Moser iteration, but there are some rather simple on-diagonal bounds for which the algebra of the Moser iteration is hard to work out, while the Green function approach is more straightforward. Consider for example $\mathbb{H}_{-\Operator}(x,t)\lesssim t^{-N/\alpha}\textup{e}^t$ which corresponds to the L\'evy operator $\Levy^\mu$ with $\mu$ satisfying \eqref{eq:IntroBestSmoothingExample}. It is clear that the linear case has the estimate $
\|u(\cdot,t)\|_{L^\infty(\R^N)}\lesssim t^{-N/\alpha}\textup{e}^t\|u_0\|_{L^1(\R^N)}$, for a.e. $t>0$, but the unclear nonlinear case is in fact easily handled with the Green function method.

We have then reached our final task:
$$
\textit{Can the nonlinear case provide bounded solutions in cases when the linear cannot?}
$$
The question has parallels to other equations for which regularizing effects only happen when the nonlinearity is strong enough. Take e.g. the scalar conservation law $\dell_tu+\diver[f(u)]=0$. If $f(r)=r$, we are in the setting of the transport equation, and the solutions are as smooth as the initial data. Hence, the operator itself is not able to provide smoothing estimates. In the mentioned case, $f$ needs to be so-called genuinely nonlinear to provide regularizing effects. A sufficient condition is $f''(r)>0$ when $N=1$, and $f:\R^N\to\R^N$ defined as $f(r)=(u^2/2,u^3/3,\ldots,u^{N+1}/N+1)$ when $N>1$. $L^1$--$L^\infty$-smoothing can then be found in \cite{SeSi19}, while other regularizing properties in e.g. \cite{CrOtWe08}. In this context, we also mention \cite{AbBe98} which treats e.g. $\dell_tu+\diver[f(u)]-\Delta[u^m]=0$. Under some conditions on $f$, it is proven that properties like boundedness hold whenever it holds for $\dell_tu-\Delta[u^m]=0$.

We also found the answer to the above question by looking at operators which were too weak to provide boundedness estimates by themselves: the family $-\Operator=I-J\ast$, mentioned earlier. Basically, the porous medium nonlinearity is so strong that we were even able to prove that solutions of \eqref{GPME} with those operators are bounded as in \eqref{eq:IntroResolventBoundednessEstimate}. Theorem \ref{thm:L1ToLinfinitySmoothing0} provides the rigorous statement, and what is interesting to note is that the proof resembles the Green function of the resolvent operator method.

\subsubsection*{Notation}
Derivatives are denoted by $'$, $\frac{\dd}{\dd t}$, and $\dell_{x_i}$. We use standard notation for $L^p$, $W^{p,q}$, and $C_\textup{b}$. Moreover, $C_\textup{c}^\infty$ is the space of smooth functions with compact support, $C_\textup{b}^\infty$ the space of smooth functions with bounded derivatives of all orders, and $C([0,T];L_\textup{loc}^p(\R^N))$ the space of measurable functions $\psi:[0,T]\to L_\textup{loc}^p(\R^N)$ such that $\psi(t)\in L_\textup{loc}^p(\R^N)$ for all $t\in[0,T]$, $\sup_{t\in[0,T]}\|\psi(t)\|_{L^p(K)}<\infty$, and $\|\psi(t)-\psi(s)\|_{L^p(K)}\to0$ when $t\to s$ for all compact $K\subset \R^N$ and $t,s\in [0,T]$. In a similar way we also define $C([0,T];L^p(\R^N))$. Note that the notion of $\R^N\times(0,T)\ni(x,t)\mapsto \psi(x,t)\in C([0,T];L_\textup{loc}^p(\R^N))$ is a subtle one. In fact, we mean that $\psi$ has an a.e.-version which is continuous $[0,T]\to L^p(\R^N)$. Let $f,g$ be positive functions. The notation $f\lesssim g$ or $f\gtrsim g$ translates to $f\leq Cg$ or $f\geq Cg$ for some constant $C>0$. Hence, $f\eqsim g$ is exactly that $f\lesssim g$ and $f\gtrsim g$ hold simultaneously. For $\alpha\in(0,2]$ and $p\in[1,\infty)$, the quantity $(\alpha p+N(m-1))^{-1}$ will either be denoted by $\theta_p$ or $\theta_\alpha$, when there is no ambiguity. Finally, the following Young inequality is repeatedly used throughout the paper:
\begin{equation}\label{eq:Young}
ab\leq \frac{1}{\vartheta}a^\vartheta+\frac{\vartheta-1}{\vartheta}b^{\frac{\vartheta}{\vartheta-1}},\qquad\text{where $a,b>0$ and $\vartheta>1$.}
\end{equation}

%%%%%%%%%%%%%%%%%%%%%%%%%%%%%%%%%%%%%%%%%%%%%%%%%%%%
%%%%%%%%%%%%%%%%%%%%%NEW SECTION%%%%%%%%%%%%%%%%%%%%%%%
%%%%%%%%%%%%%%%%%%%%%%%%%%%%%%%%%%%%%%%%%%%%%%%%%%%%

\section{Assumptions and weak dual solutions}
\label{sec:AssGreenFunction}

The spatial dimension is fixed to be $N\geq 3$
\footnote{The cases $N=1$ and $N=2$ are different and could fall out of our general setting. An example is the fractional Laplacian, where the condition $N>\alpha$ plays an essential role in the form of the Green function. In this case we could consider $N\geq 1$, under the extra condition $\alpha\in(0,1)$. Also, the Green function of the standard Laplacian is sign-changing when $N=2$, and it falls out of our setting as it fails to satisfy assumption \eqref{Gas} below. Since we are dealing with Sobolev-type inequalities and their connection with smoothing effects, it is worth recalling that those inequalities tend to be different in dimension 1 and 2: For instance, functions of $H^1(\mathbb{R})$ are automatically bounded.}
, and the assumptions on the data $(u_0,m)$ are:
\begin{align}
&0\leq u_0\in L^1(\R^N).\footnotemark
\tag{$\textup{H}_{u_0}$}&
\label{u_0as}\\
&\text{The nonlinearity is $r\mapsto r^m$ for some fixed $m>1$}.
\tag{$\textup{H}_m$}&
\label{phias}
\end{align}

\footnotetext{For the purpose of boundedness results, there is no loss of generality in assuming nonnegative initial data. First, for sign-changing solutions, the nonlinearity $u^m$ has to be replaced by $|u|^{m-1}u$. As a consequence, $-u$ is a solution of (GPME) whenever $u$ is.
Second, consider the sign-changing solution $u$ with initial data $u_0$, and also the two other nonnegative solutions $u^+$ and $u^-$ corresponding respectively to the initial data $u_0^+=\max\{u_0,0\}$ and $u_0^-=-\min\{u_0,0\}$.
 By the comparison principle, $u_0\leq u_0^+$ implies $u\leq u^+$ and $-u_0\leq u_0^-$ implies $-u\leq u^-$. We can combine the inequalities to obtain $-u^-\leq u\leq u^+$, that is $|u|\leq \max\{u^+,u^-\} \le u^++u^-$. Also, being nonnegative, $u^+$ and $u^-$ satisfy (some form of) smoothing effect estimate, that we can sum up to obtain the same estimate for $|u|$,  since $u_0^+$ and $u_0^-$ have disjoint support.}

We will make repeated use of the Green functions (or fundamental solutions or potential kernels) $\mathbb{G}_{-\Operator}^{x_0}$ and $\mathbb{G}_{I-\Operator}^{x_0}$ of the nonnegative operator $-\Operator$ and the positive operator $I-\Operator$. A crucial assumption throughout the paper is therefore:
\begin{align}
\label{Gas}
\tag{$\textup{H}_\mathbb{G}$}
&\text{For the operator $A$, there exists a function $\mathbb{G}_{A}^{x_0}\in L_\textup{loc}^1(\R^N)$ such that:}\nonumber\\
&\text{$0\leq\mathbb{G}_{A}^{x_0}=\mathbb{G}_{A}^0(\cdot-x_0)=\mathbb{G}_{A}^0(x_0-\cdot)$ a.e. in $\R^N$ and $A[\mathbb{G}_{A}^{x_0}]=\delta_{x_0}$ in $\mathcal{D}'(\R^N)$.}\nonumber
\end{align}

\begin{remark}
The assumption $\mathbb{G}_{A}^{x_0}=\mathbb{G}_{A}^0(\cdot-x_0)$ (possibly) excludes $x$-dependent operators. To include $x$-dependent operators, one would instead need $\mathbb{G}_{A}^{x_0}=\mathbb{G}_{A}^0(\cdot,x_0)$ and $\mathbb{G}_{A}^0(\cdot,x_0)$ continuous in $\R^N\setminus\{x_0\}$. In this case, $A^{-1}[f]$ cannot be written as a convolution, but other than that, the proofs go  through as before.
\end{remark}

Appendix \ref{sec:InverseOfLinearmAccretiveDirichlet} provides a guide for checking \eqref{Gas} for specific operators. Let us just mention that when $A$ is of the form \eqref{def:LevyOperators} with a measure $\mu$ satisfying \eqref{muas}, then $\mathbb{G}_{A}^{x_0}$ satisfy the above under (possibly) some additional properties on the heat kernel associated with $A$. Moreover, for each such operator $A$, we have $A^{-1}$ defined as
$$
A^{-1}[f](y):=\int_{\R^N}\mathbb{G}_{A}^{y}f=\int_{\R^N}\mathbb{G}_{A}^{0}(\cdot-y)f=\mathbb{G}_{A}^{0}\ast f(y)=\mathbb{G}_{A}^{y}\ast f,
$$
whenever that integral is convergent. The Green functions that will be used in this paper satisfy (with $C_p,K_1,K_2,K_3, C_1>0$ all independent of $x_0$) one of the following additional assumptions:
\begin{align}
&\label{G_1}\tag{$\textup{G}_{1}$} \text{For all $R>0$, some $x_0\in\R^N$, and some $\alpha\in(0,2]$,}\nonumber\\
&\begin{cases}
\int_{B_R(x_0)}\mathbb{G}_{-\mathfrak{L}}^{x_0}(x)\dd x\leq K_1 R^\alpha, &\\
\text{and for a.e. $x\in\R^N\setminus B_R(x_0)$, }\mathbb{G}_{-\mathfrak{L}}^{x_0}(x)\leq K_2 R^{-(N-\alpha)}. &
\end{cases}\nonumber\\
&\label{G_1'}\tag{$\textup{G}_{1}'$} \text{For all $R>0$, some $x_0\in\R^N$, and some $\alpha\in(0,2]$,}\nonumber\\
&\begin{cases}
\int_{B_R(x_0)}\mathbb{G}_{-\mathfrak{L}}^{x_0}(x)\dd x\leq K_1 R^\alpha, &\\
\text{and for a.e. $x\in\R^N\setminus B_R(x_0)$, }\mathbb{G}_{-\mathfrak{L}}^{x_0}(x)\leq \max\{K_3,K_2R^{-(N-\alpha)}\}.&
\end{cases}\nonumber\\
&\label{G_2}\tag{$\textup{G}_{2}$} \text{For some $x_0\in \R^N$,}\nonumber\\
&\|\mathbb{G}_{-\mathfrak{L}}^{x_0}\|_{L^1(\R^N)}=\|\mathbb{G}_{-\mathfrak{L}}^{0}\|_{L^1(\R^N)}\leq C_{1}<\infty.\nonumber\\
&\label{G_3}\tag{$\textup{G}_{3}$} \text{For some $x_0\in \R^N$ and some $p\in(1,\infty)$,}\nonumber\\
&\|\mathbb{G}_{I-\mathfrak{L}}^{x_0}\|_{L^p(\R^N)}=\|\mathbb{G}_{I-\mathfrak{L}}^{0}\|_{L^p(\R^N)}\leq C_{p}<\infty.\nonumber
\end{align}

\begin{remark}\label{rem:G_3p=1}
\begin{enumerate}[{\rm (a)}]
\item Note that there is no ambiguity in assumption \eqref{G_1'}. Indeed, we cannot consider Green functions which are merely bounded around $x_0$ since this would contradict the integrability condition.
\item We can view assumption \eqref{G_3} in two ways: (i) We think of $-\Operator\mapsto I-\Operator$ in \eqref{GPME}, i.e., $c=1$ in \eqref{def:LevyOperators}. (ii) We think of $-\Operator$ in \eqref{GPME}, but we want to use the Green function of the resolvent of that operator, i.e., $c=0$ in \eqref{def:LevyOperators}. In both cases, if $-\Operator$ is such that the corresponding heat equation gives $L^1$-decay, then $\|\mathbb{G}_{I-\Operator}^{0}\|_{L^1(\R^N)}\leq 1$ (see Lemma \ref{lem:GreenResolventIntegrable}). Hence, if we consider item (i), we are actually in the case \eqref{G_2}.
\item For now, we just remark that the fractional Laplacian/Laplacian $-\Operator=(-\Delta)^{\frac{\alpha}{2}}$ with $\alpha\in(0,2]$ satisfy \eqref{Gas} and \eqref{G_1}--\eqref{G_3}, while $-\Operator$ of the full form \eqref{def:LevyOperators} (with $c>0$) satisfies \eqref{Gas} and \eqref{G_2}. Other important examples can be found in Section \ref{sec:GreenAndHeat}, where heat kernel bounds and Fourier methods are used to obtain Green function bounds.
\end{enumerate}
\end{remark}

If we apply the inverse $(-\Operator)^{-1}$ on each side of the PDE in \eqref{GPME}, we get
$$
0=(-\Operator)^{-1}[\dell_tu]+(-\Operator)^{-1}[(-\Operator)[u^m]]=\dell_t\big(\mathbb{G}_{-\Operator}^{x_0}\ast_xu)+u^m.
$$
We thus define a suitable class of solutions as the following:

\begin{definition}[Weak dual solution]\label{def:WeakDualSolution}
We say that a nonnegative measurable function $u$ is a \emph{weak dual solution} of \eqref{GPME}~if:
\begin{enumerate}[{\rm (i)}]
\item $u\in C([0,T]; L^1(\R^N))$ and $u^m\in L^1((0,T);L_\textup{loc}^1(\R^N))$.
\item For a.e. $0<\tau_1\leq\tau_2\leq T$, and all $\psi\in C_\textup{c}^1([\tau_1,\tau_2];L_\textup{c}^\infty(\R^N))$,
\begin{equation}\label{def:eqWeakDualSolution}
\begin{split}
&\int_{\tau_1}^{\tau_2}\int_{\R^N}\big((-\Operator)^{-1}[u]\dell_t\psi-u^m\psi\big)\dd x\dd t\\
&=\int_{\R^N}(-\Operator)^{-1}[u(\cdot,\tau_2)](x)\psi(x,\tau_2)\dd x-\int_{\R^N}(-\Operator)^{-1}[u(\cdot,\tau_1)](x)\psi(x,\tau_1)\dd x.
\end{split}
\end{equation}
\item $u(\cdot,0)=u_0$ a.e. in $\R^N$.
\end{enumerate}
\end{definition}

\begin{remark}\label{rem:WeakDualResolvent}
\begin{enumerate}[{\rm (a)}]
\item We need to argue that $(-\Operator)^{-1}[u]\in C([0,T];L_\textup{loc}^1(\R^N))$ in order to make sense of the above definition. By using \eqref{G_1} and \eqref{G_1'}, we have
$$
(-\Operator)^{-1}[\mathbf{1}_{B_r(x_0)}](x)=\int_{B_r(x_0)}\mathbb{G}_{-\Operator}^{x_0}(x)\dd x\leq C
$$
which implies that
$$
\int_{\R^N}(-\Operator)^{-1}[u(\cdot,t)](x)\mathbf{1}_{B_r(x_0)}(x)\dd x\leq C\|u(\cdot,t)\|_{L^1(\R^N)}
$$
for all $r>0$ and all $x_0\in\R^N$. This makes us able to complete the argument as in Remark 2.1 in \cite{BeBoGrMu21}. In the case of \eqref{G_2}, we get the stronger
$$
\int_{\R^N}(-\Operator)^{-1}[u(\cdot,t)](x)\dd x\leq \|\mathbb{G}_{-\Operator}^{x_0}\|_{L^1(\R^N)}\|u(\cdot,t)\|_{L^1(\R^N)},
$$
and hence, $(-\Operator)^{-1}[u]\in C([0,T];L^1(\R^N))$.
\item Later, we will also use the the weak dual formulation for
$$
\dell_tu+(I-\Operator)[u^m]=u^m \qquad\Longleftrightarrow\qquad \dell_tu-\Operator[u^m]=0.
$$
Part (ii) of the above definition then looks like
\begin{equation*}
\begin{split}
&\int_{\tau_1}^{\tau_2}\int_{\R^N}\big((I-\Operator)^{-1}[u]\dell_t\psi-u^m\psi+u^m(I-\Operator)^{-1}[\psi]\big)\dd x\dd t\\
&=\int_{\R^N}(I-\Operator)^{-1}[u(\cdot,\tau_2)](x)\psi(x,\tau_2)\dd x-\int_{\R^N}(I-\Operator)^{-1}[u(\cdot,\tau_1)](x)\psi(x,\tau_1)\dd x.
\end{split}
\end{equation*}
We again need $(I-\Operator)^{-1}[u]\in C([0,T];L_\textup{loc}^1(\R^N))$. Since
$$
\int_{\R^N}(I-\Operator)^{-1}[u(\cdot,t)](x)\dd x\leq \|\mathbb{G}_{I-\Operator}^0\|_{L^1(\R^N)}\|u(\cdot,t)\|_{L^1(\R^N)},
$$
which is finite by Remark \ref{rem:G_3p=1}, we get the stronger $(I-\Operator)^{-1}[u]\in C([0,T];L^1(\R^N))$.
\item Regarding uniqueness and very weak solutions. In many cases, weak dual solutions are very weak in the sense of \cite{DTEnJa17a, DTEnJa17b}. For instance this happens when $C_\textup{c}^\infty\subset \textup{dom}(-\Operator)$. A simple, and yet technical, proof follows by approximating $\Operator[\phi]$ by a sequence $\psi_n$ of admissible test functions in \eqref{def:eqWeakDualSolution}. As a consequence, we can use the results of \cite{DTEnJa17a, DTEnJa17b} to conclude existence and uniqueness of weak dual solutions in $L^1(\R^N)$ since we will show that they are a priori bounded. A general existence result for our purposes can be found in Proposition \ref{prop:APriori}.
\end{enumerate}
\end{remark}

%%%%%%%%%%%%%%%%%%%%%%%%%%%%%%%%%%%%%%%%%%%%%%%%%%%%
%%%%%%%%%%%%%%%%%%%%%NEW SECTION%%%%%%%%%%%%%%%%%%%%%%%
%%%%%%%%%%%%%%%%%%%%%%%%%%%%%%%%%%%%%%%%%%%%%%%%%%%%

\section{Statements of main boundedness results}

We present some explicit estimates regarding instantaneous boundedness which rely on the assumptions \eqref{G_1}--\eqref{G_3}. All of our results originate from what is often referred to as fundamental upper bounds, see Theorem \ref{thm:FundamentalOpRe} in Section \ref{sec:Proofs}. These bounds provides an ``almost representation formula'' similar to the one given by convolution in the linear case ($m=1$).

%%%%%%%%%%%%%%%%%%%%%%%%%%%%%%%%%%%%%%%%%%%%%%%%%%%%
%%%%%%%%%%%%%%%%%%%%%FIGURE%%%%%%%%%%%%%%%%%%%%%%%
%%%%%%%%%%%%%%%%%%%%%%%%%%%%%%%%%%%%%%%%%%%%%%%%%%%%

\begin{figure}[h!]
\centering
\begin{tikzpicture}
\color{black}

%Sum of onedimensional frac Lap
\node[draw,
    minimum width=2cm,
    minimum height=1.2cm,
    %fill=Rhodamine!50
] (l1) at (0,0){$\sum_{i=1}^N(-\partial_{x_ix_i}^2)^{\frac{\alpha}{2}}$};

%frac Lap
\node[draw,
    minimum width=2cm,
    minimum height=1.2cm,
    below=0.5cm of l1
] (l2){$(-\Delta)^{\frac{\alpha}{2}}$};

%Levy comparable to frac Lap
\node[draw,
    minimum width=2cm,
    minimum height=1.2cm,
    below=0.5cm of l2
] (l3){$\begin{array}{cc}\Levy^\mu\text{ such that}\\|z|^{-(N+\alpha)}\lesssim\frac{\dd \mu}{\dd z}\lesssim|z|^{-(N+\alpha)}\end{array}$};

%Smoothing
\node[draw,
    minimum width=2cm,
    minimum height=1.2cm,
    right=2cm of l2
] (c1){$\|u(\cdot,t)\|_{L^\infty}\lesssim t^{-N\theta_{\alpha}}\|u_0\|_{L^1}^{\alpha\theta_\alpha}$};

%Sobolev
\node[draw,
    minimum width=2cm,
    minimum height=1.2cm,
    right=1cm of c1
] (r1){$\text{Sobolev}$};

% Arrows with text label
\draw[-{Implies},double,line width=0.7pt] ( $ (c1.north east)!0.5!(c1.east) $ )  -- ( $ (r1.north west)!0.5!(r1.west) $ )
    node[midway,above]{$$};

\draw[-{Implies},double,line width=0.7pt] ( $ (r1.south west)!0.5!(r1.west) $ ) -- ( $ (c1.south east)!0.5!(c1.east) $ )
    node[midway,below]{$$};

\draw[-{Implies},double,line width=0.7pt] (l1.south east) |- (c1.west)
    node[midway,below]{$$};

\draw[-{Implies},double,line width=0.7pt] (l2.east) -- (c1.west)
    node[midway,below]{$$};

\draw[-{Implies},double,line width=0.7pt] (l3.north east) |- (c1.west)
    node[midway,below]{$$};

\normalcolor
\end{tikzpicture}
\caption{Operators that fall into the setting of Theorem \ref{thm:L1ToLinfinitySmoothing2}, see Section \ref{sec:GreenAndHeat}. Note that the operator $\sum_{i=1}^N(-\partial_{x_ix_i}^2)^{\frac{\alpha}{2}}$ actually enjoys Theorem \ref{thm:L1ToLinfinitySmoothing}, but after a scaling argument, we can deduce the better estimate above (Remark \ref{ref:GreenAnisotropicLaplacians}). According to Section \ref{sec:SmoothingAndGNS} they should furthermore enjoy a Sobolev inequality.}
\label{fig:ImplicationOfOperatorsAndSobolev}
\end{figure}
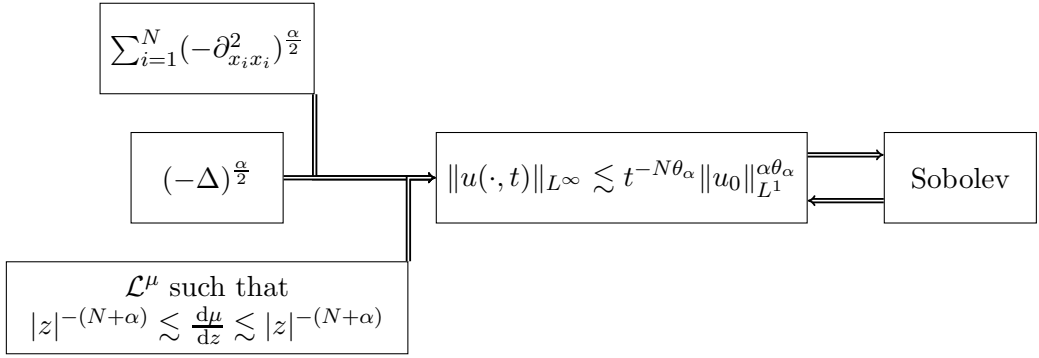

%%%%%%%%%%%%%%%%%%%%%%%%%%%%%%%%%%%%%%%%%%%%%%%%%%%%
%%%%%%%%%%%%%%%%%%%%%END FIGURE%%%%%%%%%%%%%%%%%%%%%%%
%%%%%%%%%%%%%%%%%%%%%%%%%%%%%%%%%%%%%%%%%%%%%%%%%%%%

\subsection{\texorpdfstring{$L^1$}{L1}--\texorpdfstring{$L^\infty$}{Linfty}-smoothing}

We start with the assumptions \eqref{G_1} and \eqref{G_1'} which impose the most structure. In effect, we deduce well-known results.

\begin{theorem}[$L^1$--$L^\infty$-smoothing]\label{thm:L1ToLinfinitySmoothing2}
Assume \eqref{u_0as}--\eqref{Gas}, and let $u$ be a weak dual solution of \eqref{GPME} with initial data $u_0$.
\begin{enumerate}[{\rm (a)}]
\item If \eqref{G_1} holds, then
$$
\|u(\cdot,t)\|_{L^\infty(\R^N)}\leq \frac{C(m,\alpha,N)}{t^{N\theta_{\alpha}}}\|u_0\|_{L^1(\R^N)}^{\alpha\theta_{\alpha}}\qquad\text{for a.e. $t>0$},
$$
where $\theta_{\alpha}:=(\alpha+N(m-1))^{-1}$, $C(m):=2^{\frac{m}{m-1}}$, and
$$
C(m,\alpha,N):=2^{\frac{1}{m}}C(m)^{N\theta_{\alpha}}\Big(\frac{m}{m-1}\Big)^{\alpha \theta_{\alpha}}K_1^{(N-\alpha)\theta_{\alpha}}K_2^{\alpha \theta_{\alpha}}.
$$
\item If \eqref{G_1'} holds, then
$$
\|u(\cdot,t)\|_{L^\infty(\R^N)}\leq \begin{cases}
\frac{C(m,\alpha,N)}{t^{N\theta_{\alpha}}}\|u_{0}\|_{L^1(\R^N)}^{\alpha\theta_{\alpha}} \qquad\qquad&\text{if $0<t\leq t_{0}$ a.e.,}\\
\frac{\tilde{C}(m)}{t^{\frac{1}{m}}}\|u_{0}\|_{L^1(\R^N)}^{\frac{1}{m}} \qquad\qquad&\text{if $t>t_{0}$ a.e.,}
\end{cases}
$$
where $\tilde{C}(m):=(2m(m-1)^{-1}C(m)K_3)^{1/m}$ and
$$
t_{0}:= 2^m\Big(\frac{m}{m-1}\Big)^{-(m-1)}C(m)K_1^mK_2^{\frac{\alpha m}{m-1}}K_3^{-(\frac{\alpha m}{m-1}+(m-1))}\|u_{0}\|_{L^1(\R^N)}^{-(m-1)}.
$$
\end{enumerate}
\end{theorem}

\begin{remark}\label{rem:L1ToLinfinitySmoothing2}
\begin{enumerate}[{\rm (a)}]
\item In this case, we can also get local smoothing estimates, see e.g. Proposition \ref{LocalpreCollectingResults2}.
\item Note that the estimate in Theorem \ref{thm:L1ToLinfinitySmoothing2}(a) is invariant under time- {\em and} space-scaling. Consider e.g. \eqref{GPME} with $-\Operator=(-\Delta)^{\frac{\alpha}{2}}$. If $u$ solves \eqref{GPME}, then
$$
u_{\kappa,\Xi,\Lambda}(x,t):=\kappa u(\Xi x,\Lambda t) \qquad\text{for all $\kappa,\Xi,\Lambda>0$}
$$
also solves \eqref{GPME} as long as $\kappa^{m-1}\Xi^\alpha=\Lambda$. By inserting $u_{\kappa,\Xi,\Lambda}$ into Theorem \ref{thm:L1ToLinfinitySmoothing2}(b), we see that the estimate remains the same since $\|u_{\kappa,\Xi,\Lambda}(\cdot,0)\|_{L^1(\R^N)}= \kappa \Xi^{-N}\|u(\cdot,0)\|_{L^1(\R^N)}$.
\item In a similar way, the second part of the estimate in Theorem \ref{thm:L1ToLinfinitySmoothing2}(b) is invariant under time-scaling (see Lemma \ref{lem:ScalingNonlinearity} below). Even if that estimate might seem a bit unusual, it has appeared in the literature before, see e.g. Theorem 2.7 in \cite{BeBoGrMu21}.
\item Observe also that the constant in front of both estimates blows up as $m\to1^+$.
\item As expected,
$$
\Big(\frac{1}{t}\Big)^{N\theta_{\alpha}}\leq \Big(\frac{1}{t}\Big)^{\frac{1}{m}} \qquad\text{for a.e. $t> t_0$}
$$
since the first estimate requires more assumptions at infinity.
\end{enumerate}
\end{remark}

We also include smoothing effects when \eqref{G_1} holds simultaneously for different $\alpha\in(0,2]$.

\begin{theorem}[$L^1$--$L^\infty$-smoothing]\label{thm:L1ToLinfinitySmoothing3}
Assume \eqref{u_0as}--\eqref{Gas}, and let $u$ be a weak dual solution of \eqref{GPME} with initial data $u_0$. If \eqref{G_1} holds with $\alpha\in(0,2)$ when $0<R\leq1$ and with $\alpha=2$ when $R>1$, then:
$$
\|u(\cdot,t)\|_{L^\infty(\R^N)}\leq \tilde{C}(m)\begin{cases}
t^{-N\theta_{\alpha}}\|u_0\|_{L^1(\R^N)}^{\alpha\theta_{\alpha}}&\qquad\text{if $0<t\leq \|u_0\|_{L^1(\R^N)}^{-(m-1)}$ a.e.,}\\
t^{-N\theta_{2}}\|u_0\|_{L^1(\R^N)}^{2\theta_{2}}&\qquad\text{if $t> \|u_0\|_{L^1(\R^N)}^{-(m-1)}$ a.e.,}
\end{cases}
$$
where $\theta_{\alpha}=(\alpha+N(m-1))^{-1}$ (defined for $\alpha\in (0,2]$) and
$$
\tilde{C}(m):=2\Big((C(m)K_1)^{\frac{m}{m-1}}+\frac{m}{m-1}C(m)K_2\Big)^{\frac{1}{m}}.
$$
\end{theorem}

\begin{remark}
\begin{enumerate}[{\rm (a)}]
\item Note that $t=\|u_0\|_{L^1(\R^N)}^{-(m-1)}$ gives the bound $\tilde{C}(m)\|u_0\|_{L^1(\R^N)}$ in both cases.
\item We can of course combine other behaviours in a similar way, and as a rule of thumb one can say that $0<R\leq 1$ gives small time behaviour while $R>1$ gives large time behaviour.
\end{enumerate}
\end{remark}

%%%%%%%%%%%%%%%%%%%%%%%%%%%%%%%%%%%%%%%%%%%%%%%%%%%%
%%%%%%%%%%%%%%%%%%%%%FIGURE%%%%%%%%%%%%%%%%%%%%%%%
%%%%%%%%%%%%%%%%%%%%%%%%%%%%%%%%%%%%%%%%%%%%%%%%%%%%

\begin{figure}[h!]
\centering
\begin{tikzpicture}
\color{black}

%0-order
\node[draw,
    minimum width=2cm,
    minimum height=1.2cm,
] (l1) at (0,0){$I-J\ast$};

%The most general
\node[draw,
    minimum width=2cm,
    minimum height=1.2cm,
    below=2.0cm of l1
] (l2){$\begin{array}{ccc}\Levy^\mu\text{ such that}\\|z|^{-(N+\alpha)}\mathbf{1}_{|z|\leq1}\lesssim\frac{\dd \mu}{\dd z}\lesssim|z|^{-(N+\alpha)}\mathbf{1}_{|z|\leq1}\\\quad\quad\quad\,\,\,\,\,\frac{\dd \mu}{\dd z}\lesssim \mathbf{1}_{|z|>1}\end{array}$};

%Smoothing
\node[draw,
    minimum width=2cm,
    minimum height=1.2cm,
] (c1) at (4,-1.6){$\begin{array}{ll}\|u(\cdot,t)\|_{L^\infty}\\\lesssim t^{-1/(m-1)}+\|u_0\|_{L^1}\end{array}$};

%Weaker than Sobolev
\node[draw,
    minimum width=2cm,
    minimum height=1.2cm,
    right=1cm of c1
] (r1){$\begin{array}{cc}\text{Weaker than}\\\text{Sobolev}\end{array}$};

% Arrows with text label
\draw[-{Implies},double,line width=0.7pt] (c1.east)  -- (r1.west)
    node[midway,above]{$$};

\draw[-{Implies},double,line width=0.7pt] (l1.south) |- (c1.west)
    node[midway,below]{$$};

\draw[-{Implies},double,line width=0.7pt] (l2.north) |- (c1.west)
    node[midway,below]{$$};

\normalcolor
\end{tikzpicture}
\caption{Operators that fall into the setting of Theorem \ref{thm:L1ToLinfinitySmoothing}, see Section \ref{sec:GreenAndHeat}. It is clear that not all of these operators enjoy a Sobolev inequality since e.g. $I-J\ast$ does not produce bounded solutions in the linear case. The general statement is therefore that they enjoy a functional inequality weaker than the Sobolev.}
\label{fig:ImplicationOfOperatorsAndNonSobolev}
\end{figure}
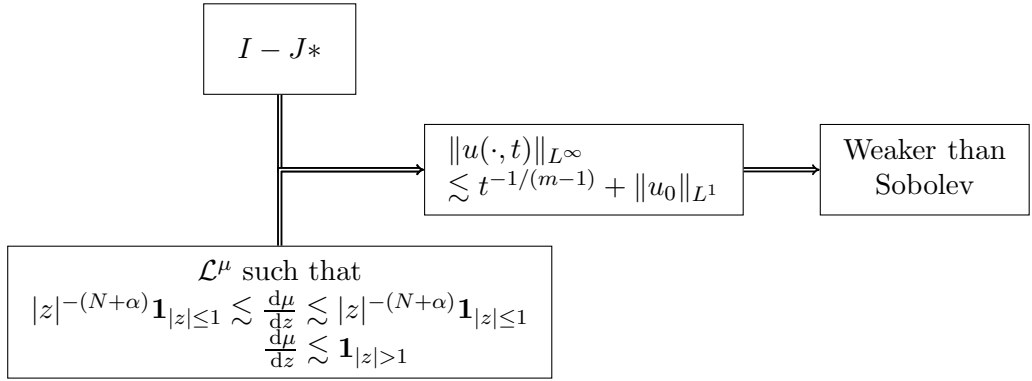

%%%%%%%%%%%%%%%%%%%%%%%%%%%%%%%%%%%%%%%%%%%%%%%%%%%%
%%%%%%%%%%%%%%%%%%%%%END FIGURE%%%%%%%%%%%%%%%%%%%%%%%
%%%%%%%%%%%%%%%%%%%%%%%%%%%%%%%%%%%%%%%%%%%%%%%%%%%%

When we use the test function $\mathbb{G}_{I-\Operator}^{x_0}$, we lack ``structure'' in the sense that we do not assume, a priori, any precise behaviour of the Green function at zero nor at infinity. Still, we arrive at:

\begin{theorem}[$L^1$--$L^\infty$-smoothing]\label{thm:L1ToLinfinitySmoothing}
Assume \eqref{u_0as}--\eqref{Gas}, and let $u$ be a weak dual solution of \eqref{GPME} with initial data $u_0$. If \eqref{G_3} holds, then
\begin{equation*}
\|u(\cdot,t)\|_{L^\infty(\R^N)}\leq\begin{cases}
2(m-1)^{-\frac{1}{m-1}}t^{-\frac{1}{m-1}} &\qquad\text{if $0<t\leq t_0$ a.e.,}\\
2C(m)^{1-\frac{1}{p}}C_p^{1-\frac{1}{p}}\|u_0\|_{L^1(\R^N)} &\qquad\text{if $t>t_0$ a.e.,}
\end{cases}
\end{equation*}
where
$$
t_0:=\frac{1}{m-1}\Big(C(m)^{1-\frac{1}{p}}C_p^{1-\frac{1}{p}}\Big)^{-(m-1)}\|u_0\|_{L^1(\R^N)}^{-(m-1)}
$$
and $C(m):=2(1+m)^{\frac{m}{m-1}}$.
\end{theorem}

\begin{remark}
\begin{enumerate}[{\rm (a)}]
\item The time-scaling (see Lemma \ref{lem:ScalingNonlinearity} below) ensures that the above estimate is of an invariant form.
\item Due to the ``linear structure'' of the fundamental upper bound in Theorem \ref{thm:FundamentalOpRe}(b), we cannot improve Theorem \ref{thm:L1ToLinfinitySmoothing} even if we strengthen assumption \eqref{G_3} in the spirit of \eqref{G_1} or \eqref{G_1'}.
\item We would also like to refer the reader to \cite{Sac85, GrMePu21}. The settings are respectively bounded domains or Riemannian manifolds, but (some of) the results have a flavour of the above estimate.
\end{enumerate}
\end{remark}

\subsection{Absolute bounds}

We also include the especially well-behaved case when $\mathbb{G}_{-\Operator}^{x_0}$ is integrable.

\begin{theorem}[Absolute bounds]\label{thm:AbsBounds}
Assume \eqref{u_0as}--\eqref{Gas}, and let $u$ be a weak dual solution of \eqref{GPME} with initial data $u_0$. If \eqref{G_2} holds, then
$$
\|u(\cdot,t)\|_{L^\infty(\R^N)}\leq \tilde{C}(m)t^{-1/(m-1)}\qquad\text{for a.e. $t>0$},
$$
where $\tilde{C}(m):=(2^{\frac{m}{m-1}}C_1)^{1/(m-1)}$.
\end{theorem}

\begin{remark}
The above estimate immediately enjoys time-scaling (see Lemma \ref{lem:ScalingNonlinearity} below).
\end{remark}

%%%%%%%%%%%%%%%%%%%%%%%%%%%%%%%%%%%%%%%%%%%%%%%%%%%%
%%%%%%%%%%%%%%%%%%%%%FIGURE%%%%%%%%%%%%%%%%%%%%%%%
%%%%%%%%%%%%%%%%%%%%%%%%%%%%%%%%%%%%%%%%%%%%%%%%%%%%

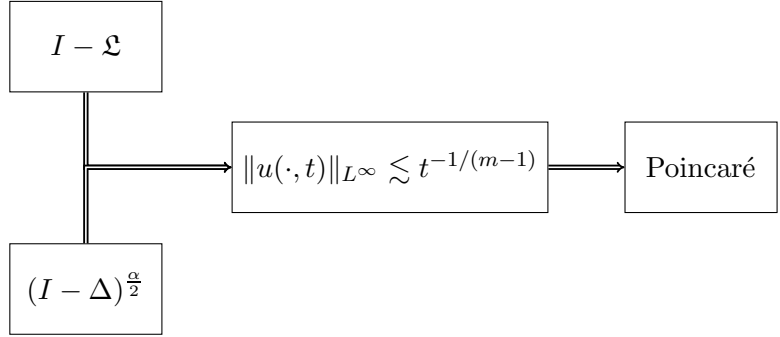
\begin{figure}[h!]
\centering
\begin{tikzpicture}
\color{black}

%Resolvent
\node[draw,
    minimum width=2cm,
    minimum height=1.2cm,
] (l1) at (0,0){$I-\Operator$};

%Bessel
\node[draw,
    minimum width=2cm,
    minimum height=1.2cm,
    below=2.0cm of l1
] (l2){$(I-\Delta)^{\frac{\alpha}{2}}$};

%Absolute bound
\node[draw,
    minimum width=2cm,
    minimum height=1.2cm,
] (c1) at (4,-1.6){$\|u(\cdot,t)\|_{L^\infty}\lesssim t^{-1/(m-1)}$};

%Poincare
\node[draw,
    minimum width=2cm,
    minimum height=1.2cm,
    right=1cm of c1
] (r1){$\text{Poincar\'e}$};

% Arrows with text label
\draw[-{Implies},double,line width=0.7pt] (c1.east)  -- (r1.west)
    node[midway,above]{$$};

\draw[-{Implies},double,line width=0.7pt] (l1.south) |- (c1.west)
    node[midway,below]{$$};

\draw[-{Implies},double,line width=0.7pt] (l2.north) |- (c1.west)
    node[midway,below]{$$};

\normalcolor
\end{tikzpicture}
\caption{Operators that fall into the setting of Theorem \ref{thm:AbsBounds}, see Section \ref{sec:GreenAndHeat}. According to Section \ref{sec:SmoothingAndGNS} they should furthermore enjoy a Poincar\'e inequality. Note that the Poincar\'e inequality is not strong enough to imply absolute bounds (only $L^q$--$L^p$-smoothing).}
\label{fig:ImplicationOfOperatorsAndWeaker}
\end{figure}

%%%%%%%%%%%%%%%%%%%%%%%%%%%%%%%%%%%%%%%%%%%%%%%%%%%%
%%%%%%%%%%%%%%%%%%%%%END FIGURE%%%%%%%%%%%%%%%%%%%%%%%
%%%%%%%%%%%%%%%%%%%%%%%%%%%%%%%%%%%%%%%%%%%%%%%%%%%%

\subsection{Linear implies nonlinear}

Since on-diagonal heat kernel bounds give $L^1$--$L^\infty$-smoothing in the linear case, we are able to transfer such estimates to the nonlinear setting \eqref{GPME} by using $\mathbb{G}_{I-\Operator}$ as a test function, see the proof in Section \ref{sec:ProofsLinearImpliesNonlinear}.

\begin{theorem}[Linear implies nonlinear]\label{thm:OverviewBoundedness}
Assume $p\in(1,\infty)$ and \eqref{u_0as}--\eqref{Gas}, and let $u$ be a weak dual solution of \eqref{GPME} with $m\geq1$ and initial data $u_0$. If the operator $-\Operator$ is such that $\mathbb{H}_{-\Operator}^{x_0}$ satisfy
$$
0\leq\mathbb{H}_{-\Operator}^{x_0}(x,t)\leq C(t)\qquad\text{with}\qquad \int_0^\infty\e^{-t}C(t)^{\frac{p-1}{p}}\dd t <\infty,
$$
then $u$ is bounded on $\R^N\times[\tau,\infty)$, for a.e. $\tau>0$.
\end{theorem}

\begin{remark}\label{rem.thm:OverviewBoundedness}
\begin{enumerate}[{\rm (a)}]
\item It is not possible to obtain that nonlinear implies linear since we construct a counterexample in Section \ref{sec:Boundedness0Order}. There we find an operator for which a linear boundedness estimate does not hold, but a nonlinear do.
\item In the case of e.g. $-\Operator=-\Delta$, the on-diagonal heat kernel bound gives
$$
|u(x,t)|\leq \int_{\R^N}|u_0(x-y)|\mathbb{H}_{-\Delta}^{x_0}(y,t)\dd y\lesssim t^{-N/2}\|u_0\|_{L^1(\R^N)}.
$$
\item By e.g. \cite[Theorem 8.16]{LiLo01}, linear operators satisfying the $L^1$--$L^\infty$-smoothing is characterized by the Nash inequality when $C(t)\eqsim t^{-\gamma/(1-\gamma)}$ for some $\gamma\in(0,1)$. Hence, the $p$ needs to be restricted to $(1,\gamma/(2\gamma-1))$. That is, given $-\Operator$ for which $C(t)$ is power-like, we can find $p$ in that interval, and then, the result transfers to the nonlinear setting ($m>1$) in the sense that $u$ solving \eqref{GPME} is bounded. See also the discussion in Section \ref{sec:TheWell-KnownLinearCase}.
\end{enumerate}
\end{remark}

%%%%%%%%%%%%%%%%%%%%%%%%%%%%%%%%%%%%%%%%%%%%%%%%%%%%
%%%%%%%%%%%%%%%%%%%%%NEW SECTION%%%%%%%%%%%%%%%%%%%%%%%
%%%%%%%%%%%%%%%%%%%%%%%%%%%%%%%%%%%%%%%%%%%%%%%%%%%%

\section{Proofs in the Green function setting}
\label{sec:Proofs}

Starting from the fundamental upper bound, already mentioned in the introduction,
\begin{equation*}%\label{eq:SecProofsAlmostRepFormula}
u^m(x_0,t)\leq\frac{1}{(m-1)t}\mathbb{G}_{-\Operator}(\cdot-x_0)\ast_x u(\cdot,t)=\frac{1}{(m-1)t}\int_{\R^N}\mathbb{G}_{-\Operator}(x-x_0)u(x,t)\dd x.
\end{equation*}
one can indeed deduce $L^1$--$L^\infty$-smoothing estimates. To motivate the proofs, let us provide the formal computations assuming \eqref{G_1}. Split the integral over $\R^N$ into $B_R(x_0)$ and $\R^N\setminus B_R(x_0)$, then estimate each part:
$$
u^m(x_0,t)\leq \|u(\cdot,t)\|_{L^\infty(\R^N)}\frac{1}{(m-1)t}K_1R^\alpha+\|u(\cdot,t)\|_{L^1(\R^N)}\frac{1}{(m-1)t}K_2R^{-(N-\alpha)}
$$
The Young inequality \eqref{eq:Young} with $\vartheta=m$ applied to the first term yields
\begin{equation*}
\begin{split}
\frac{1}{m}\|u(\cdot,t)\|_{L^\infty(\R^N)}^m+\frac{m-1}{m}\Big(\frac{1}{(m-1)t}K_1R^\alpha\Big)^{\frac{m}{m-1}}.
\end{split}
\end{equation*}
By taking the supremum, with respect to $x_0\in \R^N$, on each side of the above inequality and using the $L^1$-decay of solutions, we get, for some constant $C>0$,
\begin{equation*}
\begin{split}
\|u(\cdot,t)\|_{L^\infty(\R^N)}^m&\leq \frac{1}{2}C^m\frac{R^{\frac{\alpha m}{m-1}}}{t^{\frac{m}{m-1}}}\Big(1+\frac{t^{\frac{1}{m-1}}\|u_0\|_{L^1(\R^N)}}{R^{\frac{1}{(m-1)\theta_\alpha}}}\Big).
\end{split}
\end{equation*}
We then have
$$
R=\big(t^{\frac{1}{m-1}}\|u_0\|_{L^1(\R^N)}\big)^{(m-1)\theta_\alpha} \quad\Longrightarrow\quad
\|u(\cdot,t)\|_{L^\infty(\R^N)}\leq \frac{C}{t^{N\theta_{\alpha}}}\|u_{0}\|_{L^1(\R^N)}^{\alpha\theta_{\alpha}}.
$$

\subsection{Properties of weak dual solutions}

Our rigorous justification of the above computations starts with collecting some a priori results for weak dual solutions of \eqref{GPME}, which will be a consequence of the existence theory---its proof is postponed to Appendix \ref{sec:ExistenceAPriori}. Note that the proof requires the operator $-\Operator$ to be in a specific class, and that this particular class is at least one of the requirements to ensure \eqref{Gas}. A thorough discussion can be found in Appendix \ref{sec:InverseOfLinearmAccretiveDirichlet}.

\begin{proposition}[Existence and a priori results]\label{prop:APriori}
Assume $0\leq u_0\in (L^1\cap L^\infty)(\R^N)$, $-\Operator$ is densely defined, $\mathfrak{m}$-accretive, and Dirichlet in $L^1(\R^N)$, \eqref{phias}--\eqref{Gas}, and \eqref{G_1}--\eqref{G_3}.
\begin{enumerate}[{\rm (a)}]
\item There exists a weak dual solution $u$ of \eqref{GPME} such that
$$
0\leq u\in (L^1\cap L^\infty)(Q_T)\cap C([0,T]; L^1(\R^N)).
$$
\item Let $u,v$ be weak dual solutions of \eqref{GPME} with initial data $u_0,v_0\in (L^1\cap L^\infty)(\R^N)$. Then:
	\begin{enumerate}[{\rm (i)}]
	\item \textup{(Comparison)} If $u_0\leq v_0$ a.e. in $\R^N$, then $u\leq v$ a.e. in $Q_T$.
	\item \textup{($L^p$-decay)} $\|u(\cdot,\tau_2)\|_{L^p(\R^N)}\leq \|u(\cdot,\tau_1)\|_{L^p(\R^N)}$ for all $p\in[1,\infty]$ and a.e. $0\leq \tau_1\leq \tau_2\leq T$.
	\end{enumerate}
\end{enumerate}
\end{proposition}

\begin{remark}\label{rem:APrioriCasem1}
\begin{enumerate}[{\rm (a)}]
\item If $u_0\in L^1(\R^N)$, then Proposition \ref{prop:APriori}(b)(i)--(ii) hold also when $m=1$ by approximation, and then also for $u_0\in TV(\R^N)$.
\item We provide no general uniqueness proof. However, the constructed solutions are unique by definition since they satisfy the comparison principle.
\end{enumerate}
\end{remark}

The following scaling property holds independently of the operator $\Operator$:

\begin{lemma}[Time-scaling]\label{lem:ScalingNonlinearity}
Assume \eqref{phias} and $\Lambda>0$. If $(x,t)\mapsto u(x,t)$ solves \eqref{GPME} on $\R^N\times(0,T)$, then
$$
(x,t)\mapsto u_{\Lambda}(x,t):=\Lambda^{\frac{1}{m-1}}u(x,\Lambda t)
$$
solves \eqref{GPME} on $\R^N\times(0,\Lambda T)$ for all $\Lambda>0$.
\end{lemma}

\begin{proof}
Note that
$$
\dell_t u_\Lambda(x,t)=\Lambda^{\frac{m}{m-1}}\dell_t u(x,\Lambda t)\quad\text{and}\quad \Operator[u_\Lambda^m(\cdot,t)](x)=\Lambda^{\frac{m}{m-1}}\Operator[u^m(\cdot,\Lambda t)](x),
$$
and the proof is finished.
\end{proof}

Our proofs heavily relies on:

\begin{proposition}[Time-monotonicity, Theorem 4 in \cite{CrPi82}]\label{prop:Monotonicity}
If $u$ is a solution of \eqref{GPME} with initial data $u_0$, then the function $0<t\mapsto t^{\frac{m}{m-1}}u^m(\cdot,t)$ is nondecreasing for a.e. $x\in \R^N$.
\end{proposition}

\begin{remark}
Let us address the issue of defining exactly what we mean by the mapping ``$0<t\mapsto t^{\frac{m}{m-1}}u^m(\cdot,t)$ is nondecreasing for a.e. $x\in \R^N$''. Of course, if functions are continuous (in space and time) this is not an issue, but a priori $u$ is merely integrable, hence some remarks are in order.

We shall make a systematical use of the following properties throughout the paper. Indeed, for functions $f\in L^\infty(Q_T)\subset L_\textup{loc}^1(Q_T)$, we have that:
\begin{enumerate}[{\rm (i)}]
\item  If $(x_0,t_0)$ is a Lebesgue point for $f$, then it is a Lebesgue point for $f^m$. This follows from the fact that $f$ is essentially bounded.
\item If $(x_0,t_0)$ is a Lebesgue point for $f^m$, then it is also a Lebesgue point for $f$. This easily follows from Jensen inequality and from the fact that $m>1$, so that the power nonlinearity $f^m$ is convex.
\item We can always choose a version of $f$ which is bounded at every point, and such that all points are Lebesgue points. Indeed, we know that the set of ``non-Lebesgue points'' has measure zero, as well as the set where the function is not bounded, hence we can just redefine the function at those points as: letting $Q_R(x,t)=B_R(x)\times [t-R,t+R]\subset Q_T$ we define
    \[
    f(x,t):=\lim_{R\to 0} \frac{1}{|Q_R(x,t)|}\int_{Q_R(x,t)}f(y,\tau)\dd y \dd \tau ,
    \]
    noticing that the latter integral is always finite, since $f\in L^\infty(Q_T)$. Hence in this case we have a version of $f$ for which all points are Lebesgue points.
\item Moreover, if $f\in C([0,T];L^1(\R^N))$, we can choose a version such that $f:[0,T]\to L^1(\R^N)$ is a continuous mapping, so that for all $t\in[0,T]$, $t\mapsto f(\cdot,t)\in L^1(\R^N)$.
\end{enumerate}
Throughout the paper we will therefore fix a version of a solution $u \in L^\infty(Q_T)\cap C([0,T];L^1(\R^N))$ to the \eqref{GPME} such that all the above properties hold. These properties provide a precise meaning to the statements that we will often use in the proofs, in particular when we use a solution $u$ evaluated at a Lebesgue point $(t,x)$, or at a Lebesgue point with respect to one variable, for instance a point $(\cdot,t)$ for a.e. $t>0$ or a point $(x,\cdot)$ for a.e. $x\in\R^N$. This is clear in view of the fact that we can redefine $u \in L^\infty(Q_T)\cap C([0,T];L^1(\R^N))$ so that \emph{all} space-time points are Lebesgue points, as explained above. This will happen for instance in the proof of Theorem \ref{thm:L1ToLinfinitySmoothing}, and also in the fundamental upper bounds Theorem \ref{thm:FundamentalOpRe} and its consequences. Sometimes, one can also extend the validity of statements that hold ``for a.e. $t>0$'' to statements that hold ``for all $t>0$''.
\end{remark}

Since \eqref{GPME} enjoys time-scaling, the framework provided in \cite{CrPi82} simplifies. In our setting, the proof only relies on the comparison principle. We thus include the argument, which we originally learned from Prof. Juan Luis V\'azquez (see also \cite{BeCr81b, Vaz06}).

\begin{proof}[Proof of Proposition \ref{prop:Monotonicity}]
For all $\Lambda\geq 1$, we have $\Lambda^{\frac{1}{m-1}}u_0\geq u_0$ a.e. in $\R^N$. Lemma \ref{lem:ScalingNonlinearity} gives that $u_\Lambda$ solves \eqref{GPME} with initial data $\Lambda^{\frac{1}{m-1}}u_0$, and then the comparison principle (Theorem \ref{prop:APriori}(b)(i)) implies $u_{\Lambda}\geq u$ for $(x,t)\in Q_T$ and all $\Lambda\geq1$. For any fixed $t>0$, choose
$$
\Lambda:=\frac{t+h}{t}=1+\frac{h}{t}\qquad\text{for all $h\geq0$.}
$$
Then
$$
u(x,t)\leq u_{\Lambda}(x,t)=\Lambda^{\frac{1}{m-1}}u(x,\Lambda t)=\Big(\frac{t+h}{t}\Big)^{\frac{1}{m-1}}u(x,t+h).
$$
We conclude by noting that $r\mapsto r^m$ is increasing.
\end{proof}

\subsection{Reduction argument}

Throughout, we fix $\tau_*, T>0$ such that $0<\tau_*<T$, and let $\tau\in(\tau_*,T]$. We also consider the following sequence of approximations $\{u_{0,n}\}_{n\in\N}$ satisfying
\begin{equation}\label{eq:PropApproxu_0}
\begin{cases}
0\leq u_{0,n}\in (L^1\cap L^\infty)(\R^N)\text{ such that}\\
u_{0,n}\to u_0\text{ in $L^1(\R^N)$ and}\\
u_{0,n}\to u_0\text{ a.e. in $\R^N$ monotonically from below as $n\to \infty$.}
\end{cases}
\end{equation}
When we take $u_{0,n}$ as initial data in \eqref{GPME}, we denote the corresponding solutions by $u_n$, and they satisfy Proposition \ref{prop:APriori}, Lemma \ref{lem:ScalingNonlinearity} and Proposition \ref{prop:Monotonicity}.

\subsection{Fundamental upper bounds}

This section is devoted to prove:

\begin{theorem}[Fundamental upper bounds]\label{thm:FundamentalOpRe}
Assume $0\leq u_{0,n}\in (L^1\cap L^\infty)(\R^N)$, \eqref{phias}, and \eqref{Gas}. If $u_n$ is a weak dual solution of \eqref{GPME} with initial data $u_{0,n}$, then:
\begin{enumerate}[{\rm (a)}]
\item Under assumptions \eqref{G_1}--\eqref{G_2}, for a.e. $\tau_*>0$ and all Lebesgue points $x_0\in\R^N$,
\begin{equation*}%\label{eq:FundamentalUpperBound1}
u_n^m(x_0,\tau_*)\leq C(m)\frac{1}{\tau_*}\int_{\R^N}u_n(x,\tau_*)\mathbb{G}_{-\Operator}^{x_0}(x)\dd x,
\end{equation*}
where $C(m):=2^{\frac{m}{m-1}}$.
\item Under assumption \eqref{G_3}, for a.e. $\tau_*>0$ and all Lebesgue points $x_0\in \R^N$,
\begin{equation*}%\label{eq:FundamentalUpperBound}
u_n^m(x_0,\tau_*)\leq C(m)\lambda
\int_{\R^N}u_n(x,\tau_*)\mathbb{G}_{I-\Operator}^{x_0}(x)\dd x
\end{equation*}
where
$$
C(m):=2(1+m)^{\frac{m}{m-1}} \qquad\text{and}\qquad \lambda:=\|u_n(\cdot,\tau_*)\|_{L^\infty(\R^N)}^{m-1}>\frac{1}{(m-1)\tau_*}.
$$
\end{enumerate}
\end{theorem}

\begin{remark}
\begin{enumerate}[{\rm (a)}]
\item Theorem \ref{thm:FundamentalOpRe}(a) corresponds to equation (5.3) in \cite{BoVa15}.
\item Theorem \ref{thm:FundamentalOpRe}(b) somehow corresponds to the mentioned equation (5.3) as well, however, the inequality has a ``linear structure'' due to the presence of $\lambda$.
\end{enumerate}
\end{remark}

We begin by choosing the test function in the weak dual formulation.

\begin{lemma}\label{lem:InverseBounded}
Assume \eqref{Gas}. Then there is a constant $C$ depending on the Green function such that:
\begin{enumerate}[{\rm (a)}]
\item If \eqref{G_1} or \eqref{G_1'} holds, then
$$
\|(-\Operator)^{-1}[\psi]\|_{L^\infty(\R^N)}\leq C(\|\psi\|_{L^\infty(\R^N)}+\|\psi\|_{L^1(\R^N)})\qquad\text{for all $\psi\in (L^1\cap L^\infty)(\R^N)$.}
$$
\item If \eqref{G_2} or \eqref{G_3} holds, then, for $A=-\Operator$ or $A=I-\Operator$ respectively,
$$
\|A^{-1}[\psi]\|_{L^1(\R^N)}\leq C\|\psi\|_{L^1(\R^N)}\qquad\text{for all $\psi\in L^1(\R^N)$.}
$$
\end{enumerate}
\end{lemma}

\begin{proof}
\noindent(a) To incorporate the assumptions \eqref{G_1}--\eqref{G_1'}, we split the integral over the sets $B_R(x)$ and $\R^N\setminus B_R(x)$ and change the variable $x-y\mapsto y$ to obtain
\begin{equation*}
\begin{split}
|(-\Operator)^{-1}[\psi](x)|&\leq \int_{\R^N}\mathbb{G}_{-\Operator}^{0}(x-y)|\psi(y)|\dd y\\
&=\int_{B_R(0)}\mathbb{G}_{-\Operator}^{0}(y)|\psi(x-y)|\dd y+\int_{\R^N\setminus B_R(0)}\mathbb{G}_{-\Operator}^{0}(y)|\psi(x-y)|\dd y\\
&\leq \|\psi\|_{L^\infty(\R^N)}\int_{B_R(0)}\mathbb{G}_{-\Operator}^{0}(y)\dd y+C\|\psi\|_{L^1(\R^N)}.
\end{split}
\end{equation*}
The bound in $L^\infty$ then follows.

\medskip
\noindent(b) We simply use that $\mathbb{G}_{A}^{0}\in L^1(\R^N)$, see Remark \ref{rem:G_3p=1}.
\end{proof}

\begin{proposition}\label{prop:ChoosingTestFunction}
Assume $0\leq u_{0,n}\in (L^1\cap L^\infty)(\R^N)$, \eqref{phias}, \eqref{Gas}, and \eqref{G_1}--\eqref{G_2}.
Let $0\leq \psi\in L_\textup{c}^\infty(\R^N)$ and $0\leq \Theta\in C_\textup{b}^1([\tau_*,\tau])$. If $u_n$ is a weak dual solution of \eqref{GPME} with initial data $u_{0,n}$, then, for a.e. $\tau\in(\tau_*,T]$,
\begin{equation}\label{eq:TestFunctionInSpaceAndTime}
\begin{split}
&\int_{\tau_*}^\tau\Theta(t)\int_{\R^N}u_n^m(x,t)\psi(x)\dd x\dd t= \int_{\tau_*}^\tau\Theta'(t)\int_{\R^N}(-\Operator)^{-1}[u_n(\cdot,t)](x)\psi(x)\dd x\dd t\\
&\quad+\Theta(\tau_*)\int_{\R^N}(-\Operator)^{-1}[u_n(\cdot,\tau_*)](x)\psi(x)\dd x-\Theta(\tau)\int_{\R^N}(-\Operator)^{-1}[u_n(\cdot,\tau)](x)\psi(x)\dd x.
\end{split}
\end{equation}
\end{proposition}

\begin{proof}
Define $\dual(x,t):=\Theta(t)\psi(x)$.
Consider a sequence $\{\psi_k\}_{k\in\N}\subset C_\text{c}^\infty(\R^N\times[\tau_*,\tau])$ (i.e., a mollifying sequence) such that $\psi_k\to\dual$ and $\dell_t\psi_k\to\dell_t\dual$ a.e. as $k\to\infty$.
By Definition \ref{def:WeakDualSolution} (with $\tau_1:=\tau_*$, $\tau_2:=\tau$, and $\psi=\psi_k$),
\begin{equation*}
\begin{split}
&\int_{\R^N}(-\Operator)^{-1}[u_n(\cdot,\tau)](x)\psi_k(x,\tau)\dd x\\
&=\int_{\R^N}(-\Operator)^{-1}[u_n(\cdot,\tau_*)]\psi_k(x,\tau_*)\dd x+\int_{\tau_*}^{\tau}\int_{\R^N}\big((-\Operator)^{-1}[u_n]\dell_t\psi_k-u_n^m\psi_k\big)\dd x\dd t\\
\end{split}
\end{equation*}
holds for a.e. $\tau\in(\tau_*,T]$. Since $u_n\in L^\infty(\R^N\times[\tau_*,\tau])$, $(-\Operator)^{-1}[u_n]\in C([\tau_*,\tau];L_\textup{loc}^1(\R^N))$, and $\psi_k$ is compactly supported, we can take the limit in the above equality to get the result.
\end{proof}

\begin{corollary}[Limit estimate 1]\label{cor:LimitEstimate1}
Under the assumptions of Proposition \ref{prop:ChoosingTestFunction},
let $\psi$ approximate $\delta_{x_0}$ and choose $\Theta\equiv1$. Then
\begin{equation*}
\int_{\tau_*}^\tau u_n^m(x_0,t)\dd t= \int_{\R^N}u_n(x,\tau_*)\mathbb{G}_{-\Operator}^{x_0}(x)\dd x-\int_{\R^N}u_n(x,\tau)\mathbb{G}_{-\Operator}^{x_0}(x)\dd x.
\end{equation*}
for all Lebesgue points $x_0\in \R^N$.
\end{corollary}

\begin{proof}
Since we choose $\Theta\equiv1$, equation \eqref{eq:TestFunctionInSpaceAndTime} becomes
\begin{equation*}
\begin{split}
&\int_{\tau_*}^\tau\int_{\R^N}u_n^m(x,t)\psi(x)\dd x\dd t\\
&=\int_{\R^N}(-\Operator)^{-1}[u_n(\cdot,\tau_*)](x)\psi(x)\dd x-\int_{\R^N}(-\Operator)^{-1}[u_n(\cdot,\tau)](x)\psi(x)\dd x\\
&=\int_{\R^N}u_n(x,\tau_*)(-\Operator)^{-1}[\psi](x)\dd x-\int_{\R^N}u_n(x,\tau)(-\Operator)^{-1}[\psi](x)\dd x.
\end{split}
\end{equation*}
Now, fix $x_0\in\R^N$ and choose
$$
\psi_k^{(x_0)}=\frac{\mathbf{1}_{B_{1/k}(x_0)}}{|B_{1/k}(x_0)|}\in L_\textup{c}^\infty(\R^N).
$$
as a test function in the above equality. Since $u_n^m(\cdot,t)\in L_\textup{loc}^1(\R^N)$, and by the definition of a Lebesgue point,
$$
\int_{\tau_*}^\tau\int_{\R^N}u_n^m(x,t)\psi_k^{(x_0)}(x)\dd x\dd t\to\int_{\tau_*}^\tau u_n^m(x_0,t)\dd t\qquad\text{as $k\to\infty$.}
$$
For the remaining two terms, the argument is bit more involved, but let us start with the simplest case, in which $\mathbb{G}_{-\Operator}^{x_0}$ satisfies \eqref{G_2}. Since $\mathbb{G}_{-\Operator}^{x_0}$ is symmetric and integrable, we get
\begin{equation*}
\begin{split}
&\bigg|\int_{\R^N}u_n(x,\tau)(-\Operator)^{-1}[\psi_k^{(x_0)}](x)\dd x-\int_{\R^N}u_n(x,\tau)\mathbb{G}_{-\Operator}^{x_0}(x)\dd x\bigg|\\
& =\bigg|\int_{\R^N}u_n(x,\tau)\bigg(\fint_{B_{1/k}(x_0)}\mathbb{G}_{-\Operator}^{x}(y)\dd y-\mathbb{G}_{-\Operator}^{x_0}(x)\bigg)\dd x\bigg|\\
&\leq \|u_n(\cdot,\tau)\|_{L^\infty(\R^N)}\int_{\R^N}\fint_{B_{1/k}(x_0)}|\mathbb{G}_{-\Operator}^{0}(y-x)-\mathbb{G}_{-\Operator}^{0}(x_0-x)|\dd y\dd x\\
&= \|u_n(\cdot,\tau)\|_{L^\infty(\R^N)}\fint_{B_{1/k}(x_0)}\int_{\R^N}|\mathbb{G}_{-\Operator}^{0}(z)-\mathbb{G}_{-\Operator}^{0}(z+(x_0-y))|\dd z\dd y\\
&\leq \|u_n(\cdot,\tau)\|_{L^\infty(\R^N)}\sup_{|x_0-y|\leq 1/k}\|\mathbb{G}_{-\Operator}^{0}-\mathbb{G}_{-\Operator}^{0}(\cdot+(x_0-y))\|_{L^1(\R^N)},
\end{split}
\end{equation*}
which goes to zero as $k\to\infty$ by the continuity of the $L^1$-translation.

In the case of \eqref{G_1} and \eqref{G_1'}, we still have that $\mathbb{G}_{-\Operator}^0\in L_\textup{loc}^1(\R^N)$, and hence,
$$
(-\Operator)^{-1}[\psi_k^{(x_0)}](x)=\frac{1}{|B_{1/k}(x_0)|}\int_{B_{1/k}(x_0)}\mathbb{G}_{-\Operator}^{x}(y)\dd y\to \mathbb{G}_{-\Operator}^{x}(x_0)=\mathbb{G}_{-\Operator}^{x_0}(x)
$$
for a.e. $x\in \R^N$ as $k\to\infty$. However, we cannot simply apply the Lebesgue dominated convergence theorem since the $L^\infty$-bound of $(-\Operator)^{-1}[\psi_k^{(x_0)}]$ depends on $\|\psi_k^{(x_0)}\|_{L^\infty}\lesssim k^N$ coming from the estimate in $B_R(x_0)$ by Lemma \ref{lem:InverseBounded}. We therefore split the integral over the sets $B_{R}(x_0)$ and $\R^N\setminus B_{R}(x_0)$:
\begin{equation*}
\begin{split}
&\bigg|\int_{\R^N}u_n(x,\tau)(-\Operator)^{-1}[\psi_k^{(x_0)}](x)\dd x-\int_{\R^N}u_n(x,\tau)\mathbb{G}_{-\Operator}^{x_0}(x)\dd x\bigg|\\
&=\bigg|\int_{B_{R}(x_0)}u_n(x,\tau)(-\Operator)^{-1}[\psi_k^{(x_0)}](x)\dd x-\int_{B_{R}(x_0)}u_n(x,\tau)\mathbb{G}_{-\Operator}^{x_0}(x)\dd x\bigg|\\
&\quad+\bigg|\int_{\R^N\setminus B_{R}(x_0)}u_n(x,\tau)(-\Operator)^{-1}[\psi_k^{(x_0)}](x)\dd x-\int_{\R^N\setminus B_{R}(x_0)}u_n(x,\tau)\mathbb{G}_{-\Operator}^{x_0}(x)\dd x\bigg|\\
&=:I_1+I_2.
\end{split}
\end{equation*}
The integral $I_1$ can be handled more or less as for \eqref{G_2}. Indeed, since $|x-x_0|\leq R$ and $|x-y|\leq (3/2)R$ (in the latter we assume that $k\geq 2/R$), we estimate $I_1$ as
\begin{equation*}
\begin{split}
I_1&\leq \|u_n(\cdot,\tau)\|_{L^\infty(\R^N)}\fint_{B_{1/k}(x_0)}\int_{B_{(3/2)R}(0)}|\mathbb{G}_{-\Operator}^{0}(z)-\mathbb{G}_{-\Operator}^{0}(z+(x_0-y))|\dd z\dd y\\
&= \|u_n(\cdot,\tau)\|_{L^\infty(\R^N)}\times\\
&\quad\times\sup_{|x_0-y|\leq 1/k}\big\|\big(\mathbb{G}_{-\Operator}^{0}\mathbf{1}_{B_{(5/2)R}(0)}\big)-\big(\mathbb{G}_{-\Operator}^{0}\mathbf{1}_{B_{(5/2)R}(0)}\big)(\cdot+(x_0-y))\big\|_{L^1(\R^N)},
\end{split}
\end{equation*}
which goes to zero as $k\to\infty$ by the continuity of the $L^1$-translation. To estimate $I_2$, we consider
\begin{equation*}
\begin{split}
I_2&\leq \int_{\R^N\setminus B_{R}(x_0)}|u_n(x,\tau)|\big|(-\Operator)^{-1}[\psi_k^{(x_0)}](x)-\mathbb{G}_{-\Operator}^{x_0}(x)\big|\dd x.
\end{split}
\end{equation*}
Now, since $|x-x_0|\geq R$ and $|y-x_0|\leq 1/k\leq (1/2)R$, we use the triangle inequality to get $|x-y|\geq |x-x_0|-1/k\geq (1/2)R$. Hence, both $(-\Operator)^{-1}[\psi_k^{(x_0)}](x)=\fint_{B_{1/k}(x_0)}\mathbb{G}_{-\Operator}^{x}(y)\dd y$ and $\mathbb{G}_{-\Operator}^{x_0}(x)$ are uniformly bounded in $k$ by \eqref{G_1} and \eqref{G_1'}. The conclusion then follows by the Lebesgue dominated convergence theorem.
\end{proof}

In the case of $\mathbb{G}_{I-\Operator}^{x_0}$, we note that we can obtain a similar result as in Proposition \ref{prop:ChoosingTestFunction} (see also Remark \ref{rem:WeakDualResolvent}(b)):
\begin{equation}\label{eq:TestFunctionInSpaceAndTimeResolvent}
\begin{split}
&\int_{\tau_*}^\tau\Theta(t)\int_{\R^N}u_n^m(x,t)\psi(x)\dd x\dd t\\
&= \int_{\tau_*}^\tau\int_{\R^N}\big(\Theta'(t)+u_n^{m-1}(x,t)\Theta(t)\big)u_n(x,t)(I-\Operator)^{-1}[\psi](x)\dd x\dd t\\
&\quad+\Theta(\tau_*)\int_{\R^N}(I-\Operator)^{-1}[u_n(\cdot,\tau_*)](x)\psi(x)\dd x\\
&\quad-\Theta(\tau)\int_{\R^N}(I-\Operator)^{-1}[u_n(\cdot,\tau)](x)\psi(x)\dd x.
\end{split}
\end{equation}

To find a suitable $\Theta$, we need to fix $\tau_*>0$ and $T(\lambda):=\tau_*+\frac{m}{\lambda(m-1)}>\tau_*$. The latter is denoted by $T$ from now on.

\begin{lemma}\label{lem:TimeTest}
Assume $0\leq u_{0,n}\in (L^1\cap L^\infty)(\R^N)$ and \eqref{phias}. Let $u_n$ be a weak dual solution of \eqref{GPME} with initial data $u_{0,n}$, $t\in[\tau_*,T]$, and define
$$
\Theta(t):=(T-t)^{\frac{m}{m-1}}.
$$
Then $0\leq \Theta\in C_\textup{b}^1([\tau_*,T])$ and solves
$$
\Theta'(t)+u^{m-1}(x,t)\Theta(t)\leq 0\qquad\text{for a.e. $t\in [\tau_*,T]$ and a.e. $x\in\R^N$.}
$$
\end{lemma}

\begin{remark}
\begin{enumerate}[{\rm (a)}]
\item In particular, the choice $\tau=T=\tau_*+\frac{m}{\lambda(m-1)}$ will be used throughout the rest of the paper.
\item The exponent $\frac{m}{m-1}$ is chosen to match the one of the time-monotonicity (Proposition \ref{prop:Monotonicity}).
\end{enumerate}
\end{remark}

\begin{proof}[Proof of Lemma \ref{lem:TimeTest}]
A direct computation gives
$$
\Theta'(t)=\frac{m}{m-1}(T-t)^{\frac{m}{m-1}-1}(-1)=-\frac{ m}{(m-1)(T-t)}\Theta(t).
$$
By Proposition \ref{prop:APriori}(b)(ii) with $p=\infty$,
$$
0\leq u_n^{m-1}(x,t)\leq \|u_n\|_{L^\infty(\R^N\times(\tau_*,T))}^{m-1}\leq \lambda
$$
and then
$$
\Theta'(t)+u_n^{m-1}(x,t)\Theta(t)\leq \Big(\lambda-\frac{ m}{(m-1)(T-t)}\Big)\Theta(t)\leq 0,
$$
where, in the last inequality, we used that $\lambda$ is such that
$$
\lambda<\frac{ m}{(m-1)(T-t)}\qquad\text{for all $t\in [\tau_*,T]$.}
$$
This finished the proof.
\end{proof}

By following the proof of Corollary \ref{cor:LimitEstimate1} as for the assumption \eqref{G_2}, we get:

\begin{corollary}[Limit estimate 2]\label{cor:LimitEstimate2}
Assume $0\leq u_{0,n}\in (L^1\cap L^\infty)(\R^N)$, \eqref{phias}, \eqref{Gas}, and \eqref{G_3}. Let $u_n$ be a weak dual solution of \eqref{GPME} with initial data $u_{0,n}$, let $\psi$ approximate $\delta_{x_0}$, and choose $\Theta$ as in Lemma \ref{lem:TimeTest}. Then
\begin{equation*}
\int_{\tau_*}^{T}\Theta(t) u_n^{m}(x_0,t)\dd t\leq \Theta(\tau_*)\int_{\R^N}u_n(x,\tau_*)\mathbb{G}_{I-\Operator}^{x_0}(x)\dd x
\end{equation*}
for all Lebesgue points $x_0\in \R^N$.
\end{corollary}

\begin{remark}
We note that Corollaries \ref{cor:LimitEstimate1} and \ref{cor:LimitEstimate2} reveal that another proper functional setting is the one where
$$
\int_{\R^N}u_n(x,t)\mathbb{G}^{x_0}(x)\dd x<\infty\qquad\text{for a.e. $t>0$.}
$$
\end{remark}

We are ready to prove the fundamental upper bounds.

\begin{proof}[Proof of Theorem \ref{thm:FundamentalOpRe}]
We begin the proof by noting the following consequence of Proposition \ref{prop:Monotonicity}:
For a.e. $t\in[\tau_*,\tau]$ and all Lebesgue points $x_0\in\R^N$,
\begin{equation}\label{eq:Monotonicity}
\tau_*^{\frac{m}{m-1}}u_n^m(x_0,\tau_*)\leq t^{\frac{m}{m-1}}u_n^m(x_0,t)\leq \tau^{\frac{m}{m-1}}u_n^m(x_0,\tau).
\end{equation}

\smallskip
\noindent\text{(a)} For a.e. $\tau\geq\tau_*>0$, we combine Corollary \ref{cor:LimitEstimate1} and \eqref{eq:Monotonicity} to get
\begin{equation*}
\begin{split}
\tau_*^{\frac{m}{m-1}}u_n^m(x_0,\tau_*)\int_{\tau_*}^{\tau}\frac{1}{t^{\frac{m}{m-1}}}\dd t&\leq \int_{\tau_*}^{\tau}\frac{1}{t^{\frac{m}{m-1}}} t^{\frac{m}{m-1}}u_n^{m}(x_0,t)\dd t\\
&\leq \int_{\R^N}u_n(x,\tau_*)\mathbb{G}_{-\Operator}^{x_0}(x)\dd x,
\end{split}
\end{equation*}
or
\begin{equation*}
\begin{split}
u_n^m(x_0,\tau_*)&\leq\frac{1}{\tau_*^{\frac{m}{m-1}}}\Bigg(\int_{\tau_*}^{\tau}\frac{1}{t^{\frac{m}{m-1}}}\dd t\Bigg)^{-1} \int_{\R^N}u_n(x,\tau_*)\mathbb{G}_{-\Operator}^{x_0}(x)\dd x\\
&=\frac{1}{(m-1)\tau_*^{\frac{m}{m-1}}}\Bigg(\bigg(\frac{1}{\tau_*}\bigg)^{\frac{1}{m-1}}-\bigg(\frac{1}{\tau}\bigg)^{\frac{1}{m-1}}\Bigg)^{-1} \int_{\R^N}u_n(x,\tau_*)\mathbb{G}_{-\Operator}^{x_0}(x)\dd x.
\end{split}
\end{equation*}

Note that $t\mapsto t^{-\frac{1}{m-1}}$ is convex when $m>1$, and hence,
$$
\bigg(\frac{1}{\tau_*}\bigg)^{\frac{1}{m-1}}-\bigg(\frac{1}{\tau}\bigg)^{\frac{1}{m-1}}\geq \frac{1}{m-1}\bigg(\frac{1}{\tau}\bigg)^{\frac{m}{m-1}}\big(\tau-\tau_*\big).
$$
Moreover,
\begin{equation*}
u_n^m(x_0,\tau_*)\leq\bigg(\frac{\tau}{\tau_*}\bigg)^{\frac{m}{m-1}}\frac{1}{\tau-\tau_*}\int_{\R^N}u_n(x,\tau_*)\mathbb{G}_{-\Operator}^{x_0}(x)\dd x.
\end{equation*}
We conclude by choosing $\tau=2\tau_*$.

\medskip
\noindent\text{(b)} For some fixed $T>\tau_*$, we combine Corollary \ref{cor:LimitEstimate2} and \eqref{eq:Monotonicity} to get
\begin{equation*}
\begin{split}
\tau_*^{\frac{m}{m-1}}u_n^m(x_0,\tau_*)\int_{\tau_*}^{T}\frac{\Theta(t)}{t^{\frac{m}{m-1}}}\dd t&\leq \int_{\tau_*}^{T}\frac{\Theta(t)}{t^{\frac{m}{m-1}}} t^{\frac{m}{m-1}}u_n^{m}(x_0,t)\dd t\\
&\leq \int_{\R^N}u_n(x,\tau_*)\Theta(\tau_*)\mathbb{G}_{I-\Operator}^{x_0}(x)\dd x,
\end{split}
\end{equation*}
or
\begin{equation*}
\begin{split}
u_n^m(x_0,\tau_*)&\leq\frac{(T-\tau_*)^{\frac{m}{m-1}}}{\tau_*^{\frac{m}{m-1}}}\Bigg(\int_{\tau_*}^{T}\frac{(\tau-t)^{\frac{m}{m-1}}}{t^{\frac{m}{m-1}}}\dd t\Bigg)^{-1} \int_{\R^N}u_n(x,\tau_*)\mathbb{G}_{I-\Operator}^{x_0}(x)\dd x\\
&=\Big(\frac{T-\tau_*}{\tau_*}\Big)^{\frac{m}{m-1}}\Bigg(\int_{\tau_*}^{T}\Big(\frac{\tau-t}{t}\Big)^{\frac{m}{m-1}}\dd t\Bigg)^{-1} \int_{\R^N}u_n(x,\tau_*)\mathbb{G}_{I-\Operator}^{x_0}(x)\dd x\\
&\leq \Big(1+\frac{m}{m-1}\Big)\left(\frac{T}{\tau_*}\right)^{\frac{m}{m-1}}\frac{1}{T-\tau_*}
\int_{\R^N}u_n(x,\tau_*)\mathbb{G}_{I-\Operator}^{x_0}(x)\dd x.
\end{split}
\end{equation*}
The last step follows from the estimate
\begin{equation*}
\begin{split}
\Big(\frac{T-\tau_*}{\tau_*}\Big)^{\frac{m}{m-1}}\Bigg(\int_{\tau_*}^{T}\Big(\frac{\tau-t}{t}\Big)^{\frac{m}{m-1}}\dd t\Bigg)^{-1} \le \Big(1+\frac{m}{m-1}\Big)\left(\frac{T}{\tau_*}\right)^{\frac{m}{m-1}}\frac{1}{T-\tau_*}
\end{split}
\end{equation*}
which can be proven as follows:
\begin{equation*}
\begin{split}
&\Big(\frac{\tau_*}{T-\tau_*}\Big)^{\frac{m}{m-1}} \int_{\tau_*}^{T}\Big(\frac{T-t}{t}\Big)^{\frac{m}{m-1}}\dd t
\ge \Big(\frac{\tau_*}{T-\tau_*}\Big)^{\frac{m}{m-1}} \int_{\tau_*}^{T}\Big(\frac{T-t}{T}\Big)^{\frac{m}{m-1}}\dd t\\
&= \Big(\frac{\tau_*}{T}\Big)^{\frac{m}{m-1}}  \frac{1}{(T-\tau_*)^{\frac{m}{m-1}}} \int_{\tau_*}^{T}\Big( T-t \Big)^{\frac{m}{m-1}}\dd t= \Big(1+\frac{m}{m-1}\Big)^{-1}\Big(\frac{\tau_*}{T}\Big)^{\frac{m}{m-1}}  \frac{(T-\tau_*)^{\frac{m}{m-1}+1}}{(T-\tau_*)^{\frac{m}{m-1}}} .
\end{split}
\end{equation*}

We thus have
\begin{equation*}
\begin{split}
u_n^m(x_0,\tau_*)
&\leq \Big(1+\frac{m}{m-1}\Big)\left(\frac{T}{\tau_*}\right)^{\frac{m}{m-1}}\frac{1}{T-\tau_*}
\int_{\R^N}u_n(x,\tau_*)\mathbb{G}_{I-\Operator}^{x_0}(x)\dd x.
\end{split}
\end{equation*}
The choice $T=\tau_*+\frac{m}{\lambda(m-1)}$ and the assumption $\lambda^{-1}<(m-1)\tau_*$ readily give
\begin{equation*}
\begin{split}
u_n^m(x_0,\tau_*)
&\leq\Big(1+\frac{m-1}{m}\Big)\left(1+\frac{m}{\lambda(m-1)\tau_*}\right)^{\frac{m}{m-1}}\lambda
\int_{\R^N}u_n(x,\tau_*)\mathbb{G}_{I-\Operator}^{x_0}(x)\dd x\\
&\leq \Big(2-\frac{1}{m}\Big)(1+m)^{\frac{m}{m-1}}\lambda
\int_{\R^N}u_n(x,\tau_*)\mathbb{G}_{I-\Operator}^{x_0}(x)\dd x\\
&\leq 2(1+m)^{\frac{m}{m-1}}\lambda
\int_{\R^N}u_n(x,\tau_*)\mathbb{G}_{I-\Operator}^{x_0}(x)\dd x,
\end{split}
\end{equation*}
which is the desired result.
\end{proof}

\subsection{Boundedness under \eqref{G_3}}
Recall that the fundamental upper bound (Theorem \ref{thm:FundamentalOpRe}(b)) was only valid when $\lambda>((m-1)\tau_*)^{-1}$. Hence, we need to combine that case with $\lambda\leq((m-1)\tau_*)^{-1}$ to reach a finial conclusion. Under the latter assumption, however, we immediately have
$$
\|u_n(\cdot,\tau_*)\|_{L^\infty(\R^N)}\leq \Big(\frac{1}{(m-1)\tau_*}\Big)^{\frac{1}{m-1}}\qquad \text{for a.e.  $\tau_*>0$.}
$$
Let us therefore continue with $\lambda>((m-1)\tau_*)^{-1}$.

\begin{lemma}[$L^q$--$L^\infty$-smoothing]\label{lem:LqLinftySmoothing}
Let $p,q\in(1,\infty)$ be such that $\frac{1}{p}+\frac{1}{q}=1$. Under the assumptions of Theorem \ref{thm:FundamentalOpRe} and \eqref{G_3}, we have that
$$
\|u_n(\cdot,\tau_*)\|_{L^\infty(\R^N)}\leq C(m)C_p\|u_n(\cdot,\tau_*)\|_{L^q(\R^N)}\qquad\text{for a.e.  $\tau_*>0$.}
$$
\end{lemma}

\begin{proof}
By Theorem \ref{thm:FundamentalOpRe}(b), we get
$$
u_n^m(x_0,\tau_*)\leq C(m)\lambda\int_{\R^N} u_n(x,\tau_*)\mathbb{G}_{I-\Operator}^{x_0}(x)\dd x.
$$
Now, take the essential supremum over $x_0\in\R^N$ on both sides and use the Young inequality \eqref{eq:Young} with $\vartheta=m/(m-1)>1$ to get
$$
\|u_n(\cdot,\tau_*)\|_{L^\infty(\R^N)}^m\leq \frac{m-1}{m}\lambda^{\frac{m}{m-1}}+\frac{1}{m}\bigg(C(m)\esssup_{x_0\in\R^N}\int_{\R^N} u_n(x,\tau_*)\mathbb{G}_{I-\Operator}^{x_0}(x)\dd x\bigg)^{m},
$$
or, since $\lambda=\|u_n(\cdot,\tau_*)\|_{L^\infty(\R^N)}^{m-1}$,
$$
\|u_n(\cdot,\tau_*)\|_{L^\infty(\R^N)}\leq C(m)\esssup_{x_0\in\R^N}\int_{\R^N} u_n(x,\tau_*)\mathbb{G}_{I-\Operator}^{x_0}(x)\dd x.
$$
By assumption,
\begin{equation*}
\int_{\R^N} u_n(x,\tau_*)\mathbb{G}_{I-\Operator}^{x_0}(x)\dd x\leq \|u_n(\cdot,\tau_*)\|_{L^q(\R^N)}\|\mathbb{G}_{I-\Operator}^{0}\|_{L^p(\R^N)},
\end{equation*}
and the result follows.
\end{proof}

So, $u_n(\cdot,t)$ is in fact bounded whenever $u_n(\cdot,t)\in L^q$ for some $q\in(1,\infty)$. We exploit this in the next result.

\begin{lemma}[$L^1$--$L^\infty$-smoothing]
Let $p,q\in(1,\infty)$ be such that $\frac{1}{p}+\frac{1}{q}=1$. Under the assumptions of Theorem \ref{thm:FundamentalOpRe} and \eqref{G_3}, we have that
$$
\|u_n(\cdot,\tau_*)\|_{L^\infty(\R^N)}\leq C(m)^qC_p^q\|u_n(\cdot,\tau_*)\|_{L^1(\R^N)}\qquad\text{for a.e.  $\tau_*>0$.}
$$
\end{lemma}

\begin{proof}
We use the H\"older inequality in the proof of the Lemma \ref{lem:LqLinftySmoothing} to get
\begin{equation*}
\begin{split}
\lambda\int_{\R^N} u_n(x,\tau_*)\mathbb{G}_{I-\Operator}^{x_0}(x)\dd x
&\leq \lambda\|u_n(\cdot,\tau_*)\|_{L^q(\R^N)}\|\mathbb{G}_{I-\Operator}^{0}\|_{L^p(\R^N)}\\
&\leq \|u_n(\cdot,\tau_*)\|_{L^\infty(\R^N)}^{\frac{(m-1)q+q-1}{q}}\|u_n(\cdot,\tau_*)\|_{L^1(\R^N)}^{\frac{1}{q}}\|\mathbb{G}_{I-\Operator}^{0}\|_{L^p(\R^N)}.
\end{split}
\end{equation*}
Now, the Young inequality \eqref{eq:Young} with $\vartheta=mq>1$ gives
\begin{equation*}
\begin{split}
&\|u_n(\cdot,\tau_*)\|_{L^\infty}^{\frac{mq-1}{q}}C(m)\|u_n(\cdot,\tau_*)\|_{L^1}^{\frac{1}{q}}\|\mathbb{G}_{I-\Operator}^{x_0}\|_{L^p}\\
&\leq \frac{mq-1}{mq}\|u_n(\cdot,\tau_*)\|_{L^\infty}^m+\frac{1}{mq}C(m)^{mq}\normalcolor\|u_n(\cdot,\tau_*)\|_{L^1}^m\|\mathbb{G}_{I-\Operator}^{x_0}\|_{L^p}^{mq}.
\end{split}
\end{equation*}
Combining the above yields
\[
\|u_n(\cdot,\tau_*)\|_{L^\infty}^m\leq C(m)^{mq}C_p^{mq}\|u_n(\cdot,\tau_*)\|_{L^1}^m. \qedhere
\]
\end{proof}

We sum up the results in the following theorem:

\begin{proposition}[Smoothing effects]\label{prop:preCollectingResults}
Let $p,q\in(1,\infty)$ be such that $\frac{1}{p}+\frac{1}{q}=1$. Under the assumptions of Theorem \ref{thm:FundamentalOpRe} and \eqref{G_3}, we have that:
\begin{enumerate}[{\rm (a)}]
\item \textup{($L^q$--$L^\infty$-smoothing)}
$$
\|u_n(\cdot,\tau_*)\|_{L^\infty(\R^N)}\leq\Big(\frac{1}{(m-1)\tau_*}\Big)^{\frac{1}{m-1}}+C(m)C_p\|u_n(\cdot,\tau_*)\|_{L^q(\R^N)}\qquad\text{for a.e. $\tau_*>0$.}
$$
\item \textup{($L^1$--$L^\infty$-smoothing)}
$$
\|u_n(\cdot,\tau_*)\|_{L^\infty(\R^N)}\leq \Big(\frac{1}{(m-1)\tau_*}\Big)^{\frac{1}{m-1}}+C(m)^qC_p^q\|u_n(\cdot,\tau_*)\|_{L^1(\R^N)}\qquad\text{for a.e. $\tau_*>0$.}
$$
\end{enumerate}
\end{proposition}

The above results are not invariant under time-scaling (Lemma \ref{lem:ScalingNonlinearity}). We thus rewrite them in a proper form:

\begin{proposition}[Scaling-invariant smoothing effects]\label{prop:preCollectingResultsScaled}
Let $p,q\in(1,\infty)$ be such that $\frac{1}{p}+\frac{1}{q}=1$. Under the assumptions of Theorem \ref{thm:FundamentalOpRe} and \eqref{G_3}, we have that:
\begin{enumerate}[{\rm (a)}]
\item \textup{($L^q$--$L^\infty$-smoothing)}
\begin{equation*}
\|u_n(\cdot,t)\|_{L^\infty(\R^N)}\leq\begin{cases}
2((m-1)t)^{-\frac{1}{m-1}} \qquad\qquad&\text{if $0<t\leq t_{0,n}$ a.e.,}\\
2C(m)C_p\|u_{0,n}\|_{L^q(\R^N)} \qquad\qquad&\text{if $t>t_{0,n}$ a.e.,}
\end{cases}
\end{equation*}
where
$$
t_{0,n}:=\frac{1}{m-1}\Big(C(m)C_p\|u_{0,n}\|_{L^q(\R^N)}\Big)^{-(m-1)}.
$$
\item \textup{($L^1$--$L^\infty$-smoothing)}
\begin{equation*}
\|u_n(\cdot,t)\|_{L^\infty(\R^N)}\leq\begin{cases}
2((m-1)t)^{-\frac{1}{m-1}} \qquad\qquad&\text{if $0<t\leq t_{0,n}$ a.e.,}\\
2C(m)^{q}C_p^{q}\|u_{0,n}\|_{L^1(\R^N)} \qquad\qquad&\text{if $t>t_{0,n}$ a.e.,}
\end{cases}
\end{equation*}
where
$$
t_{0,n}:=\frac{1}{m-1}\Big(C(m)^{q}C_p^{q}\|u_{0,n}\|_{L^1(\R^N)}\Big)^{-(m-1)}.
$$
\end{enumerate}
\end{proposition}

\begin{proof}
We only provide a proof for part (a) since part (b) is similar.

Proposition \ref{prop:preCollectingResults}(a) gives
$$
\|u_n(\cdot,\tau_*)\|_{L^\infty}\leq \Big(\frac{1}{(m-1)\tau_*}\Big)^{\frac{1}{m-1}}+C(m)C_p\|u_n(\cdot,\tau_*)\|_{L^q}\qquad \text{for a.e. $\tau_*>0$,}
$$
but this result is not respecting the time-scaling (Lemma \ref{lem:ScalingNonlinearity}):
$$
\Lambda^{\frac{1}{m-1}}\|u_n(\cdot,\Lambda\tau_*)\|_{L^\infty}\leq \Lambda^{\frac{1}{m-1}}\Big(\frac{1}{(m-1)\Lambda\tau_*}\Big)^{\frac{1}{m-1}}+\Lambda^{\frac{1}{m-1}}C(m)C_p\|u_n(\cdot,\Lambda\tau_*)\|_{L^q}.
$$
By Proposition \ref{prop:APriori}(b)(ii) with $p=q$, we can optimize by requiring that
$$
\Big(\frac{1}{(m-1)\Lambda\tau_*}\Big)^{\frac{1}{m-1}}=C(m)C_p\|u_{0,n}\|_{L^q},\quad\text{or}\quad\Lambda\tau_*=\frac{1}{m-1}\Big(\frac{1}{C(m)C_p\|u_{0,n}\|_{L^q}}\Big)^{m-1}=:t_{0,n}.
$$
We obtain that
$$
\|u_n(\cdot,t_{0,n})\|_{L^\infty}\leq 2C(m)C_p\|u_{0,n}\|_{L^q}.
$$

Now, if $0<t\leq t_{0,n}$, we use time-monotonicity (Lemma \ref{prop:Monotonicity})
$$
u_n(\cdot,t)\leq \Big(\frac{t_{0,n}}{t}\Big)^{\frac{1}{m-1}}u_n(\cdot,t_{0,n})
$$
to get
\begin{equation*}
\begin{split}
\|u_n(\cdot,t)\|_{L^\infty}
&\leq \Big(\frac{t_{0,n}}{t}\Big)^{\frac{1}{m-1}}\|u_n(\cdot,t_{0,n})\|_{L^\infty}\leq \Big(\frac{t_{0,n}}{t}\Big)^{\frac{1}{m-1}}2C(m)C_p\|u_{0,n}\|_{L^q}=2\Big(\frac{1}{(m-1)t}\Big)^{\frac{1}{m-1}}.
\end{split}
\end{equation*}
And, if $t>t_{0,n}$, we use Proposition \ref{prop:APriori}(b)(ii) with $p=\infty$
$$
\|u_n(\cdot,t)\|_{L^\infty}\leq \|u_n(\cdot,t_{0,n})\|_{L^\infty}
$$
to get
\[
\|u_n(\cdot,t)\|_{L^\infty}\leq \|u_n(\cdot,t_{0,n})\|_{L^\infty}\leq 2C(m)C_p\|u_{0,n}\|_{L^q}. \qedhere
\]
\end{proof}

\begin{proof}[Proof of Theorem \ref{thm:L1ToLinfinitySmoothing}]
By the proof of Proposition \ref{prop:APriori}, we know that for every $u_{0,n}$, there is a unique mild solution $u_n$ enjoying comparison and $L^p$-decay, and this solution is moreover a weak dual solution in the sense of Definition \ref{def:WeakDualSolution}. As a consequence, the $L^1$--$L^\infty$-smoothing of Proposition \ref{prop:preCollectingResultsScaled}(b) holds for $u_n$.
By construction, we have that $0\leq u_{0,n}\leq u_{0,n+1}\leq u_0$ a.e. in $\R^N$ for all $n\in\N$ (see \eqref{eq:PropApproxu_0}), so that Proposition \ref{prop:APriori}(b)(i) yields
$$
0\leq u_{n}\leq u_{n+1}\qquad\text{a.e. in $Q_T$ for all $n\in\N$.}
$$
By monotonicity, the pointwise limit of $\{u_n\}_{n\in\N}$ always exists (possibly being $+\infty$ on a set of measure zero), and we then define our candidate limit solution as
$$
u(x,t):=\liminf_{n\to\infty} u_n(x,t).
$$
Moreover, by the Fatou lemma and Proposition \ref{prop:APriori}(b)(ii), we immediately have that
$$
\|u(\cdot,t)\|_{L^1(\R^N)}\leq \liminf_{n\to\infty}\|u_n(\cdot,t)\|_{L^1(\R^N)}\leq \|u_{0,n}\|_{L^1(\R^N)}\leq \|u_{0}\|_{L^1(\R^N)}.
$$
As a consequence, the set $\{(x,t)\in Q_T \,:\, u(x,t)=\infty\}$ has measure zero, so that the convergence above holds a.e. in $Q_T$, and $0\leq u_n\leq u$  a.e. in $Q_T$. Note that, for a.e. $x\in \R^N$ and a.e. $t>0$, we have
\[
u(x,t)=\liminf_{n\to\infty} u_n(x,t)\le \liminf_{n\to\infty} \|u_n( \cdot,t )\|_{L^\infty(\R^N)}\,.
\]
As a consequence, $u$  inherits from $u_n$ the mentioned $L^1$--$L^\infty$-smoothing effect of Proposition \ref{prop:preCollectingResultsScaled}(b) since \eqref{eq:PropApproxu_0} gives
$$
\|u_{0,n}\|_{L^1(\R^N)}\to\|u_0\|_{L^1(\R^N)}\quad\text{and}\quad t_{0,n}\to t_0\qquad\text{as $n\to\infty$}.
$$

It remains to check that the constructed limit $u$ is indeed a weak dual solution. To that end, note that the regularity assumptions $u^m\in L^1((0,T);L_\textup{loc}^1(\R^N))$ and $u\in L^\infty((0,T);L^1(\R^N))$ are straightforward consequences of the fact that $u\in L^1(Q_T)\cap L^\infty(\R^N\times[\tau,T))$ for a.e. $\tau>0$. Moreover, $0\leq u_n\leq u_{n+1}\leq u$ a.e. and $u_n\to u$ a.e. as $n\to\infty$ yield that a simple use of the monotone convergence theorem ensures that parts (ii) and (iii) of Definition \ref{def:WeakDualSolution} are true for $u$ (the limit integrals are all finite). It only remains to prove that $u\in C([0,T]; L^1(\R^N))$. Let us begin with $t\in (0,T)$. We shall use that, for all $n\in\N$, $u_n\in C([0,T]; L^1(\R^N))$ satisfies the following time-monotonicity estimate (cf. Proposition \ref{prop:Monotonicity}):
$$
t_0^{\frac{1}{m-1}}u_n(x,t_0)\leq t_1^{\frac{1}{m-1}}u_n(x,t_1), \qquad\text{for a.e. $t_1\geq t_0>0$ and a.e. $x\in\R^N$,}
$$
which can be rearranged to
$$
u_n(x,t_1)-u_n(x,t_0)\geq\Big(\frac{t_0}{t_1}\Big)^{\frac{1}{m-1}}u_n(x,t_0)-u_n(x,t_0)= -\Big(1-\Big(\frac{t_0}{t_1}\Big)^{\frac{1}{m-1}}\Big)u_n(x,t_0).
$$
Now, recall that $|f|=2f^-+f=2(-\min\{f,0\})+f$, so that
\begin{equation*}
\begin{split}
\|u_n(\cdot,t_1)-u_n(\cdot,t_0)&\|_{L^1(\R^N)}=2\int_{\R^N}\big(u_n(x,t_1)-u_n(x,t_0)\big)^-\dd x+\int_{\R^N}\big(u_n(x,t_1)-u_n(x,t_0)\big)\dd x\\
&\leq 2\Big(1-\Big(\frac{t_0}{t_1}\Big)^{\frac{1}{m-1}}\Big)\|u_n(\cdot,t_0)\|_{L^1(\R^N)}+\|u_n(\cdot,t_1)\|_{L^1(\R^N)}-\|u_n(\cdot,t_0)\|_{L^1(\R^N)}\\
&\leq \frac{2\|u_0\|_{L^1(\R^N)}}{t_1^{\frac{1}{m-1}}}\big(t_1^{\frac{1}{m-1}}-t_0^{\frac{1}{m-1}}\big),
\end{split}
\end{equation*}
where we used that $u_n\geq 0$, the $L^1$-decay of $u_n$, and $u_{0,n}\leq u_0$. Changing the roles of $t_0$ and $t_1$ reveals that the estimate also holds when $0<t_1\leq t_0$, i.e.,
$$
\|u_n(\cdot,t_1)-u_n(\cdot,t_0)\|_{L^1(\R^N)}\leq \frac{2\|u_0\|_{L^1(\R^N)}}{t_1^{\frac{1}{m-1}}}\big|t_1^{\frac{1}{m-1}}-t_0^{\frac{1}{m-1}}\big|.
$$
Moreover, since $u_n(\cdot,t)\to u(\cdot,t)$ a.e. in $\R^N$ for a.e. $t>0$, we get that $|u_n(\cdot,t_1)-u_n(\cdot,t_0)|\to |u(\cdot,t_1)-u(\cdot,t_0)|$. Then a simple application of the Fatou lemma yields
$$
\|u(\cdot,t_1)-u(\cdot,t_0)\|_{L^1(\R^N)}\leq \liminf_{n\to\infty}\|u_n(\cdot,t_1)-u_n(\cdot,t_0)\|_{L^1(\R^N)}\leq \frac{2\|u_0\|_{L^1(\R^N)}}{t_1^{\frac{1}{m-1}}}\big|t_1^{\frac{1}{m-1}}-t_0^{\frac{1}{m-1}}\big|,
$$
so that $u\in C((0,T];L^1(\R^N))$.

The continuity at $t=0$ is a consequence of the triangle inequality. Indeed, for a.e. $t\in(0,T]$,
\begin{equation*}
\begin{split}
\|u(\cdot,t)-u_0\|_{L^1(\R^N)}&\leq \|u(\cdot,t)-u_n(\cdot,t)\|_{L^1(\R^N)}+\|u_n(\cdot,t)-u_{0,n}\|_{L^1(\R^N)}\\
&\quad+\|u_{0,n}-u_0\|_{L^1(\R^N)}.
\end{split}
\end{equation*}
The last two terms go to zero as $n\to\infty$ since $u_n\in C([0,T]; L^1(\R^N))$ and by the assumption \eqref{eq:PropApproxu_0} on $u_{0,n}$. Finally, the first term goes to zero by the Lebesgue dominated convergence theorem, noting that $|u(\cdot,t)-u_n(\cdot,t)|\leq 2|u(\cdot,t)|\in L^1(\R^N)$.
\end{proof}

\begin{remark}
Note that we never used any of the particular assumptions on the Green function \eqref{G_1}--\eqref{G_3} here. This will be important later when we want repeat the above argument in slightly different settings.
\end{remark}

\subsection{Boundedness under \eqref{G_1} and \eqref{G_1'}}

\begin{proposition}[Smoothing effects]\label{preCollectingResults2}
Under the assumptions of Theorem \ref{thm:FundamentalOpRe}, we have that:
\begin{enumerate}[{\rm (a)}]
\item If \eqref{G_1} holds, then
$$
\|u_n(\cdot,\tau_*)\|_{L^\infty(\R^N)}\leq \frac{C(m,\alpha,N)}{\tau_*^{N\theta}}\|u_{0,n}\|_{L^1(\R^N)}^{\alpha\theta}\qquad\text{for a.e. $\tau_*>0$},
$$
where $\theta:=(\alpha+N(m-1))^{-1}$ and
$$
C(m,\alpha,N):=2^{\frac{1}{m}}C(m)^{N\theta}\Big(\frac{m}{m-1}\Big)^{\alpha \theta}K_1^{(N-\alpha)\theta}K_2^{\alpha \theta}.
$$
\item If \eqref{G_1'} holds, then
\begin{equation*}
\|u_n(\cdot,\tau_*)\|_{L^\infty(\R^N)}\leq\begin{cases}
\frac{C(m,\alpha,N)}{\tau_*^{N\theta}}\|u_{0,n}\|_{L^1(\R^N)}^{\alpha\theta} \qquad\qquad&\text{if $0<\tau_*\leq t_{0,n}$ a.e.,}\\
\Big(\frac{\tilde{C}(m)}{(m-1)\tau_*}\Big)^{\frac{1}{m}}\|u_{0,n}\|_{L^1(\R^N)}^{\frac{1}{m}} \qquad\qquad&\text{if $\tau_*>t_{0,n}$  a.e.,}
\end{cases}
\end{equation*}
where $\tilde{C}(m):=2mC(m)K_3$ and
$$
t_{0,n}:= 2^m\Big(\frac{m}{m-1}\Big)^{-(m-1)\theta}C(m)K_1^mK_2^{\frac{\alpha m}{m-1}}K_3^{-(\frac{\alpha m}{m-1}+(m-1))}\|u_{0,n}\|_{L^1(\R^N)}^{-(m-1)}.
$$
\end{enumerate}
\end{proposition}

The proof is based on the following intermediate results:

\begin{proposition}[Local smoothing effects]\label{LocalpreCollectingResults1}
Under the assumptions of Theorem \ref{thm:FundamentalOpRe}, and that, for all $\rho>0$ and all $\alpha\in(0,2]$,
$$
\int_{B_\rho(x_0)}\mathbb{G}_{-\Operator}^{x_0}(x)\dd x\leq K_1 \rho^\alpha \qquad\text{and}\qquad
\int_{\R^N\setminus B_{\rho}(x_0)}u_n(x,\tau_*)\mathbb{G}_{-\Operator}^{x_0}(x)\dd x<\infty ,
$$
we get, for  a.e. $\tau_*>0$,  a.e. $z\in\R^N$, and all $0<\bar{R}<R<2\bar{R}$,
\begin{equation*}
\begin{split}
&\|u_n(\cdot,\tau_*)\|_{L^\infty(B_{\bar{R}}(z))}^m\\
&\leq AR^{\frac{\alpha m}{m-1}}+\esssup_{x_0\in B_{R}(z)}\frac{m}{m-1}\frac{C(m)}{\tau_*}\int_{\R^N\setminus B_{R}(x_0)}u_n(x,\tau_*)\mathbb{G}_{-\Operator}^{x_0}(x)\dd x,
\end{split}
\end{equation*}
where
$$
A:=\Big(\frac{2C(m)}{\tau_*}K_1\Big)^{\frac{m}{m-1}}.
$$
\end{proposition}

\begin{remark}
\begin{enumerate}[{\rm (a)}]
\item By Corollary \ref{cor:LimitEstimate1},
$$
\int_{\R^N}u_n(x,\tau_*)\mathbb{G}_{-\Operator}^{x_0}(x)\dd x\leq \int_{\R^N}u_{0,n}(x)\mathbb{G}_{-\Operator}^{x_0}(x)\dd x.
$$
\item The above assumptions are analogous to the space $L_{\alpha}(\R^N)$ discussed in \cite{KuMiSi18}.
\end{enumerate}
\end{remark}

\begin{proof}[Proof of Proposition \ref{LocalpreCollectingResults1}]
Fix $0<\bar{R}< R<2\bar{R}$. We split the integral in Theorem \ref{thm:FundamentalOpRe}(a) and use assumption \eqref{G_1} to obtain
\begin{equation*}
\begin{split}
&u_n^m(x_0,\tau_*)\\
&\leq \frac{C(m)}{\tau_*}\int_{B_{\bar{R}}(x_0)}u_n(x,\tau_*)\mathbb{G}_{-\Operator}^{x_0}(x)\dd x+\frac{C(m)}{\tau_*}\int_{B_{R}(x_0)\setminus B_{\bar{R}}(x_0)}u_n(x,\tau_*)\mathbb{G}_{-\Operator}^{x_0}(x)\dd x\\
&\quad+\frac{C(m)}{\tau_*}\int_{\R^N\setminus B_{R}(x_0)}u_n(x,\tau_*)\mathbb{G}_{-\Operator}^{x_0}(x)\dd x\\
&\leq \Big(\|u_n(\cdot,\tau_*)\|_{L^\infty(B_{\bar{R}}(x_0))}+\|u_n(\cdot,\tau_*)\|_{L^\infty(B_{R}(x_0)\setminus B_{\bar{R}}(x_0))}\Big)\frac{C(m)}{\tau_*}K_1R^\alpha\\
&\quad+\frac{C(m)}{\tau_*}\int_{\R^N\setminus B_{R}(x_0)}u_n(x,\tau_*)\mathbb{G}_{-\Operator}^{x_0}(x)\dd x.
\end{split}
\end{equation*}
The Young inequality \eqref{eq:Young} with $\vartheta=m$ applied to the first term yields
\begin{equation*}
\begin{split}
\frac{1}{2m}\|u_n(\cdot,\tau_*)\|_{L^\infty(B_{\bar{R}}(x_0))}^m&+\frac{1}{2m}\|u_n(\cdot,\tau_*)\|_{L^\infty(B_{R}(x_0)\setminus B_{\bar{R}}(x_0))}^m\\
&+\frac{2^{1+\frac{1}{m-1}}(m-1)}{m}\Big(\frac{C(m)}{\tau_*}K_1R^\alpha\Big)^{\frac{m}{m-1}}.
\end{split}
\end{equation*}
By taking the supremum on each side with respect to $x_0\in B_{\bar{R}}(z)$ and using that
$$
\esssup_{x_0\in B_{\bar{R}}(z)}\|u_n(\cdot,\tau_*)\|_{L^\infty(B_{\bar{R}}(x_0))}\leq \|u_n(\cdot,\tau_*)\|_{L^\infty(B_{2\bar{R}}(z))}\leq \|u_n(\cdot,\tau_*)\|_{L^\infty(B_{3\bar{R}}(z))}
$$
and
$$
\esssup_{x_0\in B_{\bar{R}}(z)}\|u_n(\cdot,\tau_*)\|_{L^\infty(B_{R}(x_0)\setminus B_{\bar{R}}(x_0))}\leq \|u_n(\cdot,\tau_*)\|_{L^\infty(B_{R+\bar{R}}(z))}\leq \|u_n(\cdot,\tau_*)\|_{L^\infty(B_{3\bar{R}}(z))},
$$
we get
\begin{equation*}
\begin{split}
\|u_n(\cdot,\tau_*)\|_{L^\infty(B_{\bar{R}}(z))}^m&\leq \frac{1}{m}\|u_n(\cdot,\tau_*)\|_{L^\infty(B_{3\bar{R}}(z))}^m+\frac{m-1}{m}AR^{\frac{\alpha m}{m-1}}\\
&\quad+\esssup_{x_0\in B_{R}(z)}\frac{C(m)}{\tau_*}\int_{\R^N\setminus B_{R}(x_0)}u_n(x,\tau_*)\mathbb{G}_{-\Operator}^{x_0}(x)\dd x.
\end{split}
\end{equation*}
To conclude, we absorb the term $\|u_n(\cdot,\tau_*)\|_{L^\infty(B_{3\bar{R}}(z))}^m$ due to a classical lemma  (cf. Lemma \ref{lem:DeGiorgi}).
\end{proof}

\begin{proposition}[Local smoothing effects 2]\label{LocalpreCollectingResults2}
Under the assumptions of Theorem \ref{thm:FundamentalOpRe}, and for all $z\in \R^N$ and all $\bar{R}>0$ small enough, we have that:
\begin{enumerate}[{\rm (a)}]
\item If \eqref{G_1} holds, then
$$
\|u_n(\cdot,\tau_*)\|_{L^\infty(B_{\bar{R}}(z))}\leq \frac{C(m,\alpha,N)}{\tau_*^{N\theta}}\|u_{0,n}\|_{L^1(\R^N)}^{\alpha\theta}\qquad\text{for  a.e.  $\tau_*>0$},
$$
where $\theta:=(\alpha+N(m-1))^{-1}$ and
$$
C(m,\alpha,N):=2^{\frac{1}{m}+(N-\alpha)\theta}C(m)^{N\theta}\Big(\frac{m}{m-1}\Big)^{\alpha \theta}K_1^{(N-\alpha)\theta}K_2^{\alpha \theta}.
$$
\item If \eqref{G_1'} holds, then
$$
\|u_n(\cdot,\tau_*)\|_{L^\infty(B_{\bar{R}}(z))}\leq \begin{cases}
\frac{C(m,\alpha,N)}{\tau_*^{N\theta}}\|u_{0,n}\|_{L^1(\R^N)}^{\alpha\theta} \qquad\qquad&\text{if $0<\tau_*\leq t_{0,n}$  a.e.,}\\
\Big(\frac{\tilde{C}(m)}{(m-1)\tau_*}\Big)^{\frac{1}{m}}\|u_{0,n}\|_{L^1(\R^N)}^{\frac{1}{m}} \qquad\qquad&\text{if $\tau_*>t_{0,n}$  a.e.,}
\end{cases}
$$
where $\tilde{C}(m):=2mC(m)K_3$ and
$$
t_{0,n}:= 2^m\Big(\frac{m}{m-1}\Big)^{-(m-1)}C(m)K_1^mK_2^{\frac{\alpha m}{m-1}}K_3^{-(\frac{\alpha m}{m-1}+(m-1))}\|u_{0,n}\|_{L^1(\R^N)}^{-(m-1)}.
$$
\end{enumerate}
\end{proposition}

\begin{proof}
\noindent(a) Recall that we fixed $0<\bar{R}< R<2\bar{R}$. Now, estimate
$$
\esssup_{x_0\in B_{R}(z)}\frac{m}{m-1}\frac{C(m)}{\tau_*}\int_{\R^N\setminus B_{R}(x_0)}u_n(x,\tau_*)\mathbb{G}_{-\Operator}^{x_0}(x)\dd x
$$
in Proposition \ref{LocalpreCollectingResults1} by using \eqref{G_1} to get
$$
\frac{m}{m-1}\frac{C(m)}{\tau_*}K_2R^{-(N-\alpha)}\|u_n(\cdot,\tau_*)\|_{L^1(\R^N)}=:BK_2R^{-(N-\alpha)}.
$$
Optimizing in $R$ gives
$$
R=\Big(\frac{BK_2}{A}\Big)^{(m-1)\theta},
$$
and
\begin{equation*}
\begin{split}
&\|u_n(\cdot,\tau_*)\|_{L^\infty(B_{\bar{R}}(z))}^m\\
&\leq 2A^{1-\alpha m\theta}(BK_2)^{\alpha m\theta}\\
&=2^{1+(N-\alpha)m\theta}C(m)^{Nm\theta}\Big(\frac{m}{m-1}\Big)^{\alpha m\theta}K_1^{(N-\alpha)m\theta}K_2^{\alpha m \theta}\frac{1}{\tau_*^{Nm\theta}}\|u_n(\cdot,\tau_*)\|_{L^1(\R^N)}^{\alpha m\theta}.
\end{split}
\end{equation*}

\smallskip
\noindent(b) According to assumption \eqref{G_1'}, we have power-like behaviour of the Green function when
$$
0<R\leq \Big(\frac{K_2}{K_3}\Big)^{\frac{1}{N-\alpha}},
$$
and power-like behaviour around $x=x_0$ and constant around $x\to\infty$ when
$$
R> \Big(\frac{K_2}{K_3}\Big)^{\frac{1}{N-\alpha}}.
$$

Let us then first consider the case of small $R$. Following part (a), we have
$$
\|u_n(\cdot,\tau_*)\|_{L^\infty(B_{\bar{R}}(z))}^m\leq AR^{\frac{\alpha m}{m-1}}+BK_2R^{-(N-\alpha)}.
$$
Optimizing in $R$ gives
$$
R=\Big(\frac{BK_2}{A}\Big)^{(m-1)\theta},
$$
and also the $L^1$--$L^\infty$-smoothing of part (a). However, this can only hold when
$$
\Big(\frac{BK_2}{A}\Big)^{(m-1)\theta}\leq \Big(\frac{K_2}{K_3}\Big)^{\frac{1}{N-\alpha}} \qquad\Longleftrightarrow\qquad \tau_*\leq t_{0,n}.
$$

Now, we turn our attention to the case of big $R$. Following part (a), we instead have
$$
\|u_n(\cdot,\tau_*)\|_{L^\infty(B_{\bar{R}}(z))}^m\leq AR^{\frac{\alpha m}{m-1}}+BK_3.
$$
Optimizing in $R$ gives
$$
R=\Big(\frac{BK_3}{A}\Big)^{\frac{m-1}{\alpha m}},
$$
and
\begin{equation*}
\begin{split}
\|u_n(\cdot,\tau_*)\|_{L^\infty(B_{\bar{R}}(z))}^m&\leq 2BK_3=\frac{2m}{m-1}C(m)K_3\frac{1}{\tau_*}\|u_n(\cdot,\tau_*)\|_{L^1(\R^N)}.
\end{split}
\end{equation*}
Similarly as in the case of small $R$, the above estimate can now only hold when $\tau_*>t_{0,n}$.
\end{proof}

\begin{proof}[Proof of Proposition \ref{preCollectingResults2}]
We simply take the supremum over $z\in \R^N$.
\end{proof}

The proof of Theorem \ref{thm:L1ToLinfinitySmoothing2} follows as for Theorem \ref{thm:L1ToLinfinitySmoothing}.

\subsection{Boundedness under combinations of \eqref{G_1}}

\begin{proposition}[Combined smoothing effects]\label{preCollectingResults3}
Under the assumptions of Theorem \ref{thm:FundamentalOpRe}, we have that:
If \eqref{G_1} holds with $\alpha\in(0,2)$ when $0<R\leq1$ and with $\alpha=2$ when $R>1$, then:
$$
\|u_n(\cdot,\tau_*)\|_{L^\infty(\R^N)}\leq \tilde{C}(m)\begin{cases}
\tau_*^{-N\theta_{\alpha}}\|u_{0,n}\|_{L^1(\R^N)}^{\alpha\theta_{\alpha}}&\qquad\text{if $0<\tau_*\leq \|u_{0,n}\|_{L^1(\R^N)}^{-(m-1)}$  a.e.,}\\
\tau_*^{-N\theta_{2}}\|u_{0,n}\|_{L^1(\R^N)}^{2\theta_{2}}&\qquad\text{if $\tau_*> \|u_{0,n}\|_{L^1(\R^N)}^{-(m-1)}$ a.e.,}
\end{cases}
$$
where $\theta_{\alpha}=(\alpha+N(m-1))^{-1}$ (defined for $\alpha\in (0,2]$) and
$$
\tilde{C}(m):=2\Big((C(m)K_1)^{\frac{m}{m-1}}+\frac{m}{m-1}C(m)K_2\Big)^{\frac{1}{m}}.
$$
\end{proposition}

\begin{proof}
Fix $0<R\leq 1$ (to be determined). We split the integral in Theorem \ref{thm:FundamentalOpRe}(a) and use assumption \eqref{G_1} to obtain
\begin{equation*}
\begin{split}
&u_n^m(x_0,\tau_*)\\
&\leq \frac{C(m)}{\tau_*}\int_{B_{R}(x_0)}u_n(x,\tau_*)\mathbb{G}_{-\Operator}^{x_0}(x)\dd x+\frac{C(m)}{\tau_*}\int_{\R^N\setminus B_{R}(x_0)}u_n(x,\tau_*)\mathbb{G}_{-\Operator}^{x_0}(x)\dd x\\
&\leq \|u_n(\cdot,\tau_*)\|_{L^\infty(\R^N)}\frac{C(m)}{\tau_*}K_1R^\alpha+\|u_n(\cdot,\tau_*)\|_{L^1(\R^N)}\frac{C(m)}{\tau_*}K_2R^{-(N-\alpha)}.
\end{split}
\end{equation*}
We then proceed as in the beginning of Section \ref{sec:Proofs} to obtain
$$
\|u_n(\cdot,\tau_*)\|_{L^\infty(\R^N)}\leq \frac{\tilde{C}(m)}{\tau_*^{N\theta_{\alpha}}}\|u_{0,n}\|_{L^1(\R^N)}^{\alpha\theta_{\alpha}}
$$
as long as
$$
\big(\tau_*^{\frac{1}{m-1}}\|u_0\|_{L^1(\R^N)}\big)^{(m-1)\theta_\alpha}\leq 1\qquad\Longleftrightarrow\qquad \tau_*\leq\|u_0\|_{L^1(\R^N)}^{-(m-1)}.
$$

Now, fix $R>1$ (to be determined). By simply repeating the above calculations (replacing $\alpha$ by $2$), the choice
$$
R=\big(\tau_*^{\frac{1}{m-1}}\|u_0\|_{L^1(\R^N)}\big)^{(m-1)\theta_2}
$$
gives
$$
\|u_n(\cdot,\tau_*)\|_{L^\infty(\R^N)}\leq \frac{\tilde{C}(m)}{\tau_*^{N\theta_{2}}}\|u_{0,n}\|_{L^1(\R^N)}^{2\theta_{2}}
$$
as long as
\[
\big(\tau_*^{\frac{1}{m-1}}\|u_0\|_{L^1(\R^N)}\big)^{(m-1)\theta_2}> 1\qquad\Longleftrightarrow\qquad \tau_*>\|u_0\|_{L^1(\R^N)}^{-(m-1)}. \qedhere
\]
\end{proof}

The proof of Theorem \ref{thm:L1ToLinfinitySmoothing3} follows as for Theorem \ref{thm:L1ToLinfinitySmoothing}.

\subsection{Boundedness under \eqref{G_2}}

\begin{proposition}[Absolute bounds]\label{prop:AbsBounds}
Under the assumptions of Theorem \ref{thm:FundamentalOpRe} and \eqref{G_2}, we have that
$$
\|u_n(\cdot,\tau_*)\|_{L^\infty(\R^N)}\leq \Big(\frac{C(m)C_1}{\tau_*}\Big)^{\frac{1}{m-1}}\qquad\text{for  a.e.  $\tau_*>0$.}
$$
\end{proposition}

\begin{proof}
By Theorem \ref{thm:FundamentalOpRe}(a), we get
$$
u_n^m(x_0,\tau_*)\leq C(m)\frac{1}{\tau_*}\int u_n(x,\tau_*)\mathbb{G}_{-\Operator}^{x_0}(x)\dd x\leq \frac{C(m)}{\tau_*}\|u_n(\cdot,\tau_*)\|_{L^\infty}\|\mathbb{G}_{-\Operator}^{0}\|_{L^1}.
$$
Now, the Young inequality \eqref{eq:Young} with $\vartheta=m$ gives
$$
\Big(1-\frac{1}{m}\Big)\|u_n(\cdot,\tau_*)\|_{L^\infty}^{m}\leq \frac{m-1}{m}\Big(\frac{C(m)C_1}{\tau_*}\Big)^{\frac{m}{m-1}},
$$
and hence, the result follows.
\end{proof}

The proof of Theorem \ref{thm:AbsBounds} follows as for Theorem \ref{thm:L1ToLinfinitySmoothing}.

\subsection{Linear implies nonlinear}
\label{sec:ProofsLinearImpliesNonlinear}

\begin{proof}[Proof of Theorem \ref{thm:OverviewBoundedness}]
Linear smoothing effects hold due to Theorem \ref{thm:LinearEquivalences} below. As for the nonlinear case, to be in the setting of Theorem \ref{thm:L1ToLinfinitySmoothing}, we only need to check that \eqref{G_3} holds. By Proposition \ref{prop:TheInverseOperatorA-1} and the Minkowski inequality for integrals (cf. Theorem 2.4 in \cite{LiLo01}),
\begin{equation}\label{eq:LpNormOfGreenOfResolvent}
\begin{split}
\|\mathbb{G}_{I-\Operator}^{x_0}\|_{L^p(\R^N)}=\|\mathbb{G}_{I-\Operator}^{0}\|_{L^p(\R^N)}
&=\bigg(\int_{\R^N}\bigg(\int_{0}^\infty\textup{e}^{-t} \mathbb{H}_{-\Operator}^{0}(x,t)\dd t\bigg)^p\dd x\bigg)^{\frac{1}{p}}\\
&\leq \int_{0}^\infty\bigg( \int_{\R^N}\Big(\textup{e}^{-t}\mathbb{H}_{-\Operator}^{0}(x,t)\Big)^p\dd x\bigg)^{\frac{1}{p}}\dd t\\
&=\int_{0}^\infty\textup{e}^{-t}\bigg( \int_{\R^N}\Big(\mathbb{H}_{-\Operator}^{0}(x,t)\Big)^p\dd x\bigg)^{\frac{1}{p}}\dd t\\
&=\int_{0}^\infty\textup{e}^{-t}\|\mathbb{H}_{-\Operator}^{0}(\cdot,t)\|_{L^p(\R^N)}\dd t.
\end{split}
\end{equation}
Finally,
$$
\|\mathbb{H}_{-\Operator}^{0}(\cdot,t)\|_{L^p(\R^N)}\leq \|\mathbb{H}_{-\Operator}^{0}(\cdot,t)\|_{L^\infty(\R^N)}^{\frac{p-1}{p}}\|\mathbb{H}_{-\Operator}^{0}(\cdot,t)\|_{L^1(\R^N)}\leq \|\mathbb{H}_{-\Operator}^{0}(\cdot,t)\|_{L^\infty(\R^N)}^{\frac{p-1}{p}}\leq C(t)^{\frac{p-1}{p}}
$$
completes the proof.
\end{proof}

%%%%%%%%%%%%%%%%%%%%%%%%%%%%%%%%%%%%%%%%%%%%%%%%%%%%
%%%%%%%%%%%%%%%%%%%%%NEW SECTION%%%%%%%%%%%%%%%%%%%%%%%
%%%%%%%%%%%%%%%%%%%%%%%%%%%%%%%%%%%%%%%%%%%%%%%%%%%%

\section{Boundedness results for \texorpdfstring{$0$}{0}-order operators}
\label{sec:Boundedness0Order}

We need some assumptions regarding $0$-order or nonsingular operators, i.e., operators of the form $-\Operator=-\Levy^{\mu}$ with $\dd \mu=J\dd z$ where:
\begin{align}
&\textup{$J\geq0$ a.e. on $\R^N$, symmetric, and $\|J\|_{L^1(\R^N)}=1$.}
\tag{$\textup{J}_{1}$}&
\label{J_1}\\
&\|J\|_{L^p(\R^N)}\leq C_{J,p}<\infty \textup{ for some $p\in(1,\infty]$.}\nonumber
\tag{$\textup{J}_{2}$}&
\label{J_2}
\end{align}
I.e., we consider convolution type operators $-\Operator=I-J\ast$ (which we will denote by $-\Levy^{J}$). The nonlinear equation \eqref{GPME} with such operators has been studied in e.g. \cite{DFFeLi08}. Assumption \eqref{J_2} ensure that $J$ is far away from being concentrated. Therefore, we cannot consider discrete measures $\mu$, and thus, operators like the discrete Laplacian. The smoothing takes the following form:

\begin{theorem}[$L^1$--$L^\infty$-smoothing]\label{thm:L1ToLinfinitySmoothing0}
Assume \eqref{u_0as}, \eqref{phias}, and $q=p/(p-1)\in[1,\infty)$, and let $u$ be a very weak solution of \eqref{GPME} with initial data $u_0$. If \eqref{J_1} and \eqref{J_2} hold, then
\begin{equation*}
\|u(\cdot,t)\|_{L^\infty(\R^N)}\leq\begin{cases}
2(mqC(m)^{\frac{m}{m-1}})^{\frac{1}{m}}t^{-\frac{1}{m-1}} \qquad\qquad&\text{if $0<t\leq t_{0}$  a.e.,}\\
2\Big(\frac{mC(m)}{m-1}C_{J,p}\Big)^q\|u_{0}\|_{L^1(\R^N)} \qquad\qquad&\text{if $t>t_{0}$ a.e.,}
\end{cases}
\end{equation*}
where $C(m):=2^{\frac{m}{m-1}}$ and
$$
t_{0}:=(mq)^{\frac{m-1}{m}}\Big(\frac{m}{m-1}C_{J,p}\Big)^{-q(m-1)}C(m)^{1-q(m-1)}\|u_{0}\|_{L^1(\R^N)}^{-(m-1)}.
$$
\end{theorem}

\begin{remark}
The time-scaling (Lemma \ref{lem:ScalingNonlinearity}) ensures that the above estimate is of the proper form.
\end{remark}

The $0$-order or nonsingular operators have a particularly simple approach. In contrast to general singular L\'evy operators in this paper which maps $W^{2,p}(\R^N)$ to $L^p(\R^N)$, $0$-order operators $-\Operator=-\Levy^{J}$ are well-defined for merely $L^p(\R^N)$-functions:
$$
-\Levy^{J}:L^p(\R^N)\rightarrow L^p(\R^N) \qquad\text{for all $p\in[1,\infty]$.}
$$
We then define very weak solutions:

\begin{definition}[Very weak solution]\label{def:VeryWeak}
Assume $-\Operator=-\Levy^{J}$. We say that a nonnegative measurable function $u$ is a \emph{very weak solution} of \eqref{GPME}~if:
\begin{enumerate}[{\rm (i)}]
\item $u\in L^1(Q_T)\cap C([0,T]; L_{\textup{loc}}^1(\R^N))$ and $u^m\in L^1(Q_T)$.
\item For a.e. $0<\tau_1\leq\tau_2\leq T$, and all $\psi\in C_\textup{c}^\infty(\R^N\times[\tau_1,\tau_2])$,
\begin{equation*}
\begin{split}
&\int_{\tau_1}^{\tau_2}\int_{\R^N}\big(u\dell_t\psi-(-\Levy^{J})[u^m]\psi\big)\dd x\dd t=\int_{\R^N}u(x,\tau_2)\psi(x,\tau_2)\dd x-\int_{\R^N}u(x,\tau_1)\psi(x,\tau_1)\dd x.
\end{split}
\end{equation*}
\item $u(\cdot,0)=u_0$ a.e. in $\R^N$.
\end{enumerate}
\end{definition}

\begin{remark}
For general L\'evy operators, see \eqref{def:LevyOperators}, we need to put the operator on the test function instead.
\end{remark}

We collect some known a priori results for \eqref{GPME} which will be useful in the proofs, see e.g. Theorem 2.3 in \cite{DTEnJa17b}.

\begin{lemma}[Known a priori results]\label{lem:APriori0}
Assume $0\leq u_0\in (L^1\cap L^\infty)(\R^N)$, \eqref{phias}, and $-\Operator=-\Levy^{J}$.
\begin{enumerate}[{\rm (a)}]
\item There exists a unique very weak solution $u$ of \eqref{GPME} with initial data $u_0$ such that
$$
0\leq u\in (L^1\cap L^\infty)(Q_T)\cap C([0,T]; L_\textup{loc}^1(\R^N)).
$$
\item Let $u,v$ be very weak solutions of \eqref{GPME} with initial data $u_0,v_0\in (L^1\cap L^\infty)(\R^N)$. Then:
	\begin{enumerate}[{\rm (i)}]
	\item \textup{(Comparison)} If $u_0\leq v_0$ a.e. in $\R^N$, then $u\leq v$ a.e. in $Q_T$.
	\item \textup{($L^p$-decay)} $\|u(\cdot,\tau_2)\|_{L^p(\R^N)}\leq \|u(\cdot,\tau_1)\|_{L^p(\R^N)}$ for all $p\in[1,\infty]$ and a.e. $0\leq \tau_1\leq \tau_2\leq T$.
	\end{enumerate}
\end{enumerate}
\end{lemma}

\begin{remark}\label{rem:APrioriCasem10}
If $u_0\in L^1(\R^N)$, then Lemma \ref{lem:APriori0}(b)(i)--(ii) hold also when $m=1$ by approximation, and then also for $u_0\in TV(\R^N)$.
\end{remark}

Again, we fix $\tau_*, T>0$ such that $0<\tau_*<T$, and let $\tau\in(\tau_*,T]$. We also consider the following sequence of approximations $\{u_{0,n}\}_{n\in\N}$ satisfying
\begin{equation*}%\label{eq:PropApproxu_02}
\begin{cases}
0\leq u_{0,n}\in (L^1\cap L^\infty)(\R^N)\text{ such that}\\
u_{0,n}\to u_0\text{ in $L^1(\R^N)$, and}\\
u_{0,n}\to u_0\text{ a.e. in $\R^N$ monotonically from below.}
\end{cases}
\end{equation*}
When we take $u_{0,n}$ as initial data in \eqref{GPME}, we denote the corresponding solutions by $u_n$, and they satisfy Lemmas \ref{lem:APriori0}, \ref{lem:ScalingNonlinearity} and Proposition \ref{prop:Monotonicity}.

\begin{proposition}\label{prop:ChoosingTestFunction2}
Assume $0\leq u_{0,n}\in (L^1\cap L^\infty)(\R^N)$, \eqref{phias}, $-\Operator=-\Levy^{J}$, and \eqref{J_1}. Let $0\leq v\in L^1(\R^N)$ and $0\leq \Theta\in C_\textup{b}^1([\tau_*,\tau])$. If $u_n$ is a very weak solution of \eqref{GPME} with initial data $u_{0,n}$, then, for a.e. $\tau\in(\tau_*,T]$,
\begin{equation}\label{eq:TestFunctionInSpaceAndTime2}
\begin{split}
&\int_{\tau_*}^\tau\Theta(t)\int_{\R^N}(-\Levy^{J})[u_n^m(\cdot,t)](x)v(x)\dd x\dd t= \int_{\tau_*}^\tau\Theta'(t)\int_{\R^N}u_n(x,t)v(x)\dd x\dd t\\
&\quad+\Theta(\tau_*)\int_{\R^N}u_n(x,\tau_*)v(x)\dd x-\Theta(\tau)\int_{\R^N}u_n(x,\tau)v(x)\dd x.
\end{split}
\end{equation}
\end{proposition}

\begin{corollary}[Limit estimate 3]\label{cor:LimitEstimate3}
Under the assumptions of Proposition \ref{prop:ChoosingTestFunction2},
let
$$
v_R(x):=\frac{\mathbf{1}_{B(x_0,R)}(x)}{|B(x_0,R)|}\qquad \text{with $R>0$}
$$
approximate $\delta_{x_0}$``$=\mathbb{G}_{I}^{x_0}$'' and choose $\Theta\equiv1$. Then
\begin{equation*}
\int_{\tau_*}^\tau(-\Levy^{J})[u_n^m(\cdot,t)](x_0)\dd t=u_n(x_0,\tau_*)-u_n(x_0,\tau)
\end{equation*}
for all Lebesgue points $x_0\in \R^N$.
\end{corollary}

\begin{proof}
Simply apply the Lebesgue differentiation theorem as $R\to0^+$ in \eqref{eq:TestFunctionInSpaceAndTime2}.
\end{proof}

\begin{theorem}[Fundamental upper bound]\label{thm:FundamentalOpRe0}
Assume $0\leq u_{0,n}\in (L^1\cap L^\infty)(\R^N)$, \eqref{phias}, $-\Operator=-\Levy^{J}$, and \eqref{J_1}. If $u_n$ is a very weak solution of \eqref{GPME} with initial data $u_{0,n}$, then, for a.e. $\tau_*>0$ and all Lebesgue points $x_0\in \R^N$,
\begin{equation*}
\begin{split}
u_n^m(x_0,\tau_*)\leq C(m)\bigg( \frac{u_n(x_0,\tau_*)}{\tau_*}+\frac{1}{\tau_*}\int_{\tau_*}^{2\tau_*}\int_{\R^N}u_n^m(x,t)J^{x_0}(x)\dd x\dd t\bigg).
\end{split}
\end{equation*}
where $J^{x_0}(x)=J(x-x_0)$ and $C(m)=2^{\frac{m}{m-1}}$.
\end{theorem}

\begin{remark}
This is a completely new result, but we see that $J^{x_0}$ somehow takes the role of $\mathbb{G}^{x_0}$.
\end{remark}

\begin{proof}[Proof of Theorem \ref{thm:FundamentalOpRe0}]
We begin the proof by noting the following consequence of Proposition \ref{prop:Monotonicity}:
For a.e. $t\in[\tau_*,\tau]$ and all Lebesgue points $x_0\in\R^N$,
\begin{equation}\label{eq:Monotonicity0}
\tau_*^{\frac{m}{m-1}}u_n^m(x_0,\tau_*)\leq t^{\frac{m}{m-1}}u_n^m(x_0,t)\leq \tau^{\frac{m}{m-1}}u_n^m(x_0,\tau).
\end{equation}
We rearrange the result in Corollary \ref{cor:LimitEstimate3}:
\begin{equation*}
\begin{split}
\int_{\tau_*}^\tau u_n^m(x_0,t)\dd t&=u_n(x_0,\tau_*)-u_n(x_0,\tau)+\int_{\tau_*}^\tau\int_{\R^N}u_n^m(x,t)J^{x_0}(x)\dd x\dd t\\
&\leq u(x_0,\tau_*)+\int_{\tau_*}^\tau\int_{\R^N}u_n^m(x,t)J^{x_0}(x)\dd x\dd t.
\end{split}
\end{equation*}
Arguing by time-monotonicity \eqref{eq:Monotonicity0}, as in the proof of Theorem \ref{thm:FundamentalOpRe}(a), leads to
\begin{equation*}
\begin{split}
u_n^m(x_0,\tau_*)\leq \bigg(\frac{\tau}{\tau_*}\bigg)^{\frac{m}{m-1}}\frac{1}{\tau-\tau_*}\bigg(u_n(x_0,\tau_*)+\int_{\tau_*}^\tau\int_{\R^N}u_n^m(x,t)J^{x_0}(x)\dd x\dd t\bigg).
\end{split}
\end{equation*}
Choose $\tau=2\tau_*$ to obtain, for all $\tau_*>0$,
\begin{equation*}
\begin{split}
u_n^m(x_0,\tau_*)\leq 2^{\frac{m}{m-1}}\bigg( \frac{u_n(x_0,\tau_*)}{\tau_*}+\frac{1}{\tau_*}\int_{\tau_*}^{2\tau_*}\int_{\R^N}u_n^m(x,t)J^{x_0}(x)\dd x\dd t\bigg).
\end{split}
\end{equation*}
This completes the proof.
\end{proof}

\begin{proposition}[Smoothing effects]\label{prop:preCollectingResults2}
Assume $q=p/(p-1)\in[1,\infty)$ and $r\in(1,m]$. Under the assumptions of Theorem \ref{thm:FundamentalOpRe0} and \eqref{J_2}, we have that:
\begin{enumerate}[{\rm (a)}]
\item \textup{($L^r$--$L^\infty$-smoothing)}
For a.e. $\tau_*>0$,
\begin{equation*}
\begin{split}
\|u_n(\cdot,\tau_*)\|_{L^\infty(\R^N)}&\leq\Big(\frac{q(m-1)^{\frac{2m-1}{m-1}}}{r-1}\Big)^{\frac{1}{m}}\Big(\frac{C(m)}{(m-1)}\Big)^{\frac{1}{m-1}}\tau_*^{-\frac{1}{m-1}}\\
&\quad+\Big(\frac{rC(m)^{\frac{m}{r}}}{r-1}C_{J,p}^{\frac{m}{r}}\Big)^{\frac{q}{m}}\|u_n(\cdot,\tau_*)\|_{L^r(\R^N)}
\end{split}
\end{equation*}
where $C(m)=2^{\frac{m}{m-1}}$.
\item \textup{($L^1$--$L^\infty$-smoothing)}
For a.e. $\tau_*>0$,
\begin{equation*}
\begin{split}
\|u_n(\cdot,\tau_*)\|_{L^\infty(\R^N)}
&\leq (mqC(m)^{\frac{m}{m-1}})^{\frac{1}{m}}\tau_*^{-\frac{1}{m-1}}+\Big(\frac{mC(m)}{m-1}C_{J,p}\Big)^q\|u_n(\cdot,\tau_*)\|_{L^{1}(\R^N)}.
\end{split}
\end{equation*}
\end{enumerate}
\end{proposition}

\begin{proof}
By Theorem \ref{thm:FundamentalOpRe0},
$$
u_n^m(x_0,\tau_*)\leq \frac{C(m)}{\tau_*}u_n(x_0,\tau_*)+\frac{C(m)}{\tau_*}\int_{\tau_*}^{2\tau_*}\int_{\R^N}u_n^m(x,t)J^{x_0}(x)\dd x\dd t=:I+II.
$$

\smallskip
\noindent(a) We use Lemma \ref{lem:APriori0}(b)(ii) with $p=\infty$ to get
$$
II\leq \|u_n(\cdot,\tau_*)\|_{L^\infty(\R^N)}^{m-r}\frac{C(m)}{\tau_*}\int_{\tau_*}^{2\tau_*}\int_{\R^N}u_n^r(x,t)J^{x_0}(x)\dd x\dd t.
$$
By the Young inequality \eqref{eq:Young} with $\vartheta=\frac{m}{r}$ and $\vartheta=m$, we estimate
$$
II\leq \frac{m-r}{m}\|u_n(\cdot,\tau_*)\|_{L^\infty(\R^N)}^m+\frac{r}{m}\bigg(\frac{C(m)}{\tau_*}\int_{\tau_*}^{2\tau_*}\int_{\R^N}u_n^r(x,t)J^{x_0}(x)\dd x\dd t\bigg)^{\frac{m}{r}}
$$
and
$$
I\leq \frac{1}{m}u_n^m(x_0,\tau_*)+\frac{m-1}{m}\Big(\frac{C(m)}{\tau_*}\Big)^{\frac{m}{m-1}}.
$$
Collecting the terms yields
\begin{equation*}
\begin{split}
u_n^m(x_0,\tau_*)&\leq \frac{m}{m-1}\frac{m-r}{m}\|u_n(\cdot,\tau_*)\|_{L^\infty}^m+\frac{m}{m-1}\frac{m-1}{m}\Big(\frac{C(m)}{\tau_*}\Big)^{\frac{m}{m-1}}\\
&\quad+\frac{m}{m-1}\frac{r}{m}\bigg(\frac{C(m)}{\tau_*}\int_{\tau_*}^{2\tau_*}\int_{\R^N}u_n^r(x,t)J^{x_0}(x)\dd x\dd t\bigg)^{\frac{m}{r}}\\
&= \frac{m-r}{m-1}\|u_n(\cdot,\tau_*)\|_{L^\infty}^m+\Big(\frac{C(m)}{\tau_*}\Big)^{\frac{m}{m-1}}\\
&\quad+\frac{r}{m-1}\bigg(\frac{C(m)}{\tau_*}\int_{\tau_*}^{2\tau_*}\int_{\R^N}u_n^r(x,t)J^{x_0}(x)\dd x\dd t\bigg)^{\frac{m}{r}}.
\end{split}
\end{equation*}
Since
$$
\frac{m-r}{m-1}=\frac{m-1+1-r}{m-1}=1-\frac{r-1}{m-1},
$$
we can only absorb the $L^\infty$-norm on the left-hand side when $r>1$. Indeed,
\begin{equation*}
\begin{split}
\|u_n(\cdot,\tau_*)\|_{L^\infty(\R^N)}^m&\leq \frac{m-1}{r-1}\Big(\frac{C(m)}{\tau_*}\Big)^{\frac{m}{m-1}}\\
&\quad+\frac{r}{r-1}\esssup_{x_0\in\R^N}\bigg(\frac{C(m)}{\tau_*}\int_{\tau_*}^{2\tau_*}\int_{\R^N}u_n^r(x,t)J^{x_0}(x)\dd x\dd t\bigg)^{\frac{m}{r}},
\end{split}
\end{equation*}
and since $\|J^{x_0}\|_{L^p(\R^N)}=\|J\|_{L^p(\R^N)}$, we obtain by the H\"older inequality and Lemma \ref{lem:APriori0}(b)(ii) with $p=q$ that
\begin{equation*}
\begin{split}
\|u_n(\cdot,\tau_*)\|_{L^\infty}^m&\leq \frac{m-1}{r-1}\Big(\frac{C(m)}{\tau_*}\Big)^{\frac{m}{m-1}}+\frac{rC(m)^{\frac{m}{r}}}{r-1}\|J\|_{L^p}^{\frac{m}{r}}\|u_n^r(\cdot,\tau_*)\|_{L^{q}}^{\frac{m}{r}}\\
&\leq \frac{m-1}{r-1}\Big(\frac{C(m)}{\tau_*}\Big)^{\frac{m}{m-1}}+\frac{rC(m)^{\frac{m}{r}}}{r-1}\|J\|_{L^p}^{\frac{m}{r}}\|u_n(\cdot,\tau_*)\|_{L^{\infty}}^{\frac{m(q-1)}{q}}\|u_n(\cdot,\tau_*)\|_{L^{r}}^{\frac{m}{q}}.
\end{split}
\end{equation*}
Apply the Young inequality \eqref{eq:Young} with $\vartheta=q$ to obtain
\begin{equation*}
\begin{split}
\|u_n(\cdot,\tau_*)\|_{L^\infty}^m&\leq\frac{m-1}{r-1}\Big(\frac{C(m)}{\tau_*}\Big)^{\frac{m}{m-1}}+\frac{q-1}{q}\|u_n(\cdot,\tau_*)\|_{L^{\infty}}^{m}\\
&\quad+\frac{r^qC(m)^{\frac{mq}{r}}}{q(r-1)^q}\|J\|_{L^p}^{\frac{mq}{r}}\|u_n(\cdot,\tau_*)\|_{L^{r}}^{m},
\end{split}
\end{equation*}
or,
\begin{equation}\label{eq:JLrLinftyEstimate}
\begin{split}
\|u_n(\cdot,\tau_*)\|_{L^\infty(\R^N)}^m
&\leq\frac{q(m-1)^{1+\frac{m}{m-1}}}{r-1}\Big(\frac{C(m)}{(m-1)\tau_*}\Big)^{\frac{m}{m-1}}\\
&\quad+\Big(\frac{rC(m)^{\frac{m}{r}}}{r-1}\|J\|_{L^p(\R^N)}^{\frac{m}{r}}\Big)^q\|u_n(\cdot,\tau_*)\|_{L^{r}(\R^N)}^{m}.
\end{split}
\end{equation}
Finally, since $m>1$, $x\mapsto x^{\frac{1}{m}}$ is concave and sub-additive on $[0,\infty)$, and we conclude.

\smallskip
\noindent(b) The H\"older inequality yields
$$
\|u_n(\cdot,\tau_*)\|_{L^m}^m\leq \|u_n(\cdot,\tau_*)\|_{L^\infty}^{m-1}\|u_n(\cdot,\tau_*)\|_{L^1}.
$$
Inserting this estimate into \eqref{eq:JLrLinftyEstimate} with $r=m$, and then applying the Young inequalities with $\vartheta=m$ give
\begin{equation*}
\begin{split}
\|u_n(\cdot,\tau_*)\|_{L^\infty}^m
&\leq\frac{q(m-1)^{1+\frac{m}{m-1}}}{m-1}\Big(\frac{C(m)}{(m-1)\tau_*}\Big)^{\frac{m}{m-1}}\\
&\quad+\frac{m-1}{m}\|u_n(\cdot,\tau_*)\|_{L^{\infty}}^{m}+\frac{1}{m}\Big(\Big(\frac{mC(m)^{\frac{m}{m}}}{m-1}\|J\|_{L^p}^{\frac{m}{m}}\Big)^q\|u_n(\cdot,\tau_*)\|_{L^{1}}\Big)^{m},
\end{split}
\end{equation*}
or
\begin{equation*}
\begin{split}
\|u_n(\cdot,\tau_*)\|_{L^\infty}^m
&\leq m\frac{q(m-1)^{1+\frac{m}{m-1}}}{m-1}\Big(\frac{C(m)}{(m-1)\tau_*}\Big)^{\frac{m}{m-1}}+\Big(\Big(\frac{mC(m)}{m-1}\|J\|_{L^p}\Big)^q\|u_n(\cdot,\tau_*)\|_{L^{1}},\Big)^{m},
\end{split}
\end{equation*}
which concludes the proof since $x\mapsto x^{\frac{1}{m}}$ is sub-additive.
\end{proof}

The above results are not invariant under time-scaling (Lemma \ref{lem:ScalingNonlinearity}). We thus rewrite them in a proper form:

\begin{proposition}[Scaling-invariant smoothing effects]\label{prop:preCollectingResultsScaled2}
Assume $q=p/(p-1)\in[1,\infty)$ and $r\in(1,m]$. Under the assumptions of Theorem \ref{thm:FundamentalOpRe0} and \eqref{J_2}, we have that:
\begin{enumerate}[{\rm (a)}]
\item \textup{($L^r$--$L^\infty$-smoothing)}
\begin{equation*}
\|u_n(\cdot,t)\|_{L^\infty(\R^N)}\leq\begin{cases}
2\Big(\frac{q(m-1)^{\frac{2m-1}{m-1}}}{r-1}\Big)^{\frac{1}{m}}\Big(\frac{C(m)}{(m-1)}\Big)^{\frac{1}{m-1}}t^{-\frac{1}{m-1}} \qquad\qquad&\text{if $0<t\leq t_{0,n}$ a.e.,}\\
2\Big(\frac{rC(m)^{\frac{m}{r}}}{r-1}C_{J,p}^{\frac{m}{r}}\Big)^{\frac{q}{m}}\|u_{0,n}\|_{L^r(\R^N)} \qquad\qquad&\text{if $t>t_{0,n}$ a.e.,}
\end{cases}
\end{equation*}
where
$$
t_{0,n}:=C(m)^{\frac{r-q(m-1)}{r}}C_{J,p}^{-\frac{q(m-1)}{r}}\Big(\frac{q(r-1)^{q-1}(m-1)}{r^q}\Big)^{\frac{m-1}{m}}\|u_{0,n}\|_{L^r(\R^N)}^{-(m-1)}.
$$
\item \textup{($L^1$--$L^\infty$-smoothing)}
\begin{equation*}
\|u_n(\cdot,t)\|_{L^\infty(\R^N)}\leq\begin{cases}
2(mqC(m)^{\frac{m}{m-1}})^{\frac{1}{m}}t^{-\frac{1}{m-1}} \qquad\qquad&\text{if $0<t\leq t_{0,n}$ a.e.,}\\
2\Big(\frac{mC(m)}{m-1}C_{J,p}\Big)^q\|u_{0,n}\|_{L^1(\R^N)} \qquad\qquad&\text{if $t>t_{0,n}$ a.e.,}
\end{cases}
\end{equation*}
where
$$
t_{0,n}:=(mq)^{\frac{m-1}{m}}\Big(\frac{m}{m-1}C_{J,p}\Big)^{-q(m-1)}C(m)^{1-q(m-1)}\|u_{0,n}\|_{L^1(\R^N)}^{-(m-1)}.
$$
\end{enumerate}
\end{proposition}

\begin{proof}
We only provide a proof for part (a) since (b) is similar.

Proposition \ref{prop:preCollectingResults2}(a) gives
\begin{equation*}
\begin{split}
\|u_n(\cdot,\tau_*)\|_{L^\infty}&\leq\Big(\frac{q(m-1)^{\frac{2m-1}{m-1}}C(m)^{\frac{m}{m-1}}}{r-1}\Big)^{\frac{1}{m}}\Big(\frac{1}{(m-1)\tau_*}\Big)^{\frac{1}{m-1}}+\Big(\frac{rC(m)^{\frac{m}{r}}}{r-1}C_{J,p}^{\frac{m}{r}}\Big)^{\frac{q}{m}}\|u_n(\cdot,\tau_*)\|_{L^r},
\end{split}
\end{equation*}
but this result is not respecting the time-scaling (Lemma \ref{lem:ScalingNonlinearity}):
\begin{equation*}
\begin{split}
\Lambda^{\frac{1}{m-1}}\|u_n(\cdot,\Lambda\tau_*)\|_{L^\infty}
&\leq \Big(\frac{q(m-1)^{\frac{2m-1}{m-1}}C(m)^{\frac{m}{m-1}}}{r-1}\Big)^{\frac{1}{m}}\Lambda^{\frac{1}{m-1}}\Big(\frac{1}{(m-1)\Lambda\tau_*}\Big)^{\frac{1}{m-1}}\\
&\quad+\Big(\frac{rC(m)^{\frac{m}{r}}}{r-1}C_{J,p}^{\frac{m}{r}}\Big)^{\frac{q}{m}}\Lambda^{\frac{1}{m-1}}\|u_n(\cdot,\Lambda\tau_*)\|_{L^r}.
\end{split}
\end{equation*}
By Lemma \ref{lem:APriori0}(b)(ii) with $p=r$, we can optimize by requiring that
\begin{equation*}
\begin{split}
&\Big(\frac{q(m-1)^{\frac{2m-1}{m-1}}C(m)^{\frac{m}{m-1}}}{r-1}\Big)^{\frac{1}{m}}\Big(\frac{1}{(m-1)\Lambda\tau_*}\Big)^{\frac{1}{m-1}}=\Big(\frac{rC(m)^{\frac{m}{r}}}{r-1}C_{J,p}^{\frac{m}{r}}\Big)^{\frac{q}{m}}\|u_{0,n}\|_{L^r},
\end{split}
\end{equation*}
or
\begin{equation*}
\begin{split}
\Lambda\tau_*=C(m)^{\frac{r-q(m-1)}{r}}C_{J,p}^{-\frac{q(m-1)}{r}}\Big(\frac{q(r-1)^{q-1}(m-1)}{r^q\|u_{0,n}\|_{L^r}^m}\Big)^{\frac{m-1}{m}}=:t_{0,n}.
\end{split}
\end{equation*}
We obtain that
$$
\|u_n(\cdot,t_{0,n})\|_{L^\infty}\leq 2\Big(\frac{rC(m)^{\frac{m}{r}}}{r-1}C_{J,p}^{\frac{m}{r}}\Big)^{\frac{q}{m}}\|u_{0,n}\|_{L^r}.
$$
To finish, we follow the proof of Proposition \ref{prop:preCollectingResultsScaled}.
\end{proof}

The proof of Theorem \ref{thm:L1ToLinfinitySmoothing0} follows as for Theorem \ref{thm:L1ToLinfinitySmoothing},  except that we verify that the limit is a very weak solution in the sense of Definition \ref{def:VeryWeak} here.

%%%%%%%%%%%%%%%%%%%%%%%%%%%%%%%%%%%%%%%%%%%%%%%%%%%%
%%%%%%%%%%%%%%%%%%%%%NEW SECTION%%%%%%%%%%%%%%%%%%%%%%%
%%%%%%%%%%%%%%%%%%%%%%%%%%%%%%%%%%%%%%%%%%%%%%%%%%%%

\section{Smoothing effects VS Gagliardo-Nirenberg-Sobolev inequalities}
\label{sec:SmoothingAndGNS}

In this section we investigate the connections between the validity of smoothing effects for solutions to diffusion equations and the validity of suitable functional inequalities of Gagliardo-Nirenberg-Sobolev (GNS) type, together with some limiting cases, and their dual counterparts, the Hardy-Littlewood-Sobolev (HLS) type inequalities. As already mentioned, it is well-known since the celebrated work of Nash \cite{Nas58} that the ultracontractive estimate for solutions of the heat equation, i.e. \eqref{GPME} with $m=1$, are equivalent to a special GNS inequality. There has been an extensive literature on this nowadays classical topic, and theorems analogous to Theorem \ref{thm:LinearEquivalences} below can be found in analysis textbooks, e.g., \cite{LiLo01, S-Co02}.

In the nonlinear setting much less is known, a first result in this direction has been given in \cite{BoGr05b} where, adapting the Gross method to the nonlinear setting, logarithmic Sobolev (LS) inequalities (of Euclidean type) implied $L^1$--$L^\infty$-smoothing effects for porous medium-type equations (also on Riemannian manifolds). Indeed, LS inequalities are limiting cases of GNS inequalities, hence it is shown there how GNS inequalities imply smoothing effects. Later the equivalence between GNS and smoothing effects was established in \cite{BoGrVa08, Gri10, GrMuPo13}, see also \cite{CoHa16}. In the nonlinear case, the Nash method does not work, and the classical alternative is provided by the celebrated Moser iteration, which was first introduced for linear parabolic equations \cite{Mos64,Mos67}, then extended by various authors to the nonlinear setting, see \cite{HePi85,DaKe07,JuLi09,BoVa10,GrMuPu15,NgVa18,BoSi19,JiXi19,BoDoNaSi20,JiXi22,LiLi22}. Another classical possibility is the DeGiorgi method, which can be adapted to the nonlinear setting. It also shows how functional inequalities imply regularity properties of solutions, see for instance \cite{DBe93, DBGiVe12}. Once GNS imply smoothing, it is often possible to prove the converse implication, establishing equivalence, cf. Theorem \ref{thm:GNSEquivalentL1Linfty}.

In the pioneering paper \cite{BaCoLeS-C95}, see also \cite{S-Co02}, the equivalences of different Sobolev, GNS, Nash, LS, and Poincar\'e inequalities are established. We recall some of the precise results in Lemma \ref{lem:NashEquivalentGNS}. Roughly speaking, the idea is that all functional inequalities that can be true with a suitable quadratic form are equivalent: We will analyze mainly two classes that we call Sobolev or Poincar\'e, since they are equivalent respectively to $L^q$--$L^\infty$- or to $L^q$--$L^p$-smoothing effects for the associated linear equation, i.e., \eqref{GPME} with $m=1$. We add other equivalences and implications related to \eqref{GPME} with $m>1$, which is the main purpose of this section, see Figure \ref{fig:ImplicationInNonlinearCase} below. Let us also mention that a more direct proof of the equivalence between Nash and LS can be obtained by the methods of \cite{BoDoSc20} combined with the 4-norm inequality of \cite{BoGr05a}.

We want to emphasize that sometimes the nonlinear diffusion enjoys smoothing while the linear counterpart does not. The nonlinear smoothing must then be equivalent to a functional inequality that has to be weaker than any GNS (or any other functional inequality equivalent to Sobolev), otherwise it would imply smoothing in the linear case. We provided explicit examples of this phenomenon in Section \ref{sec:Boundedness0Order}. This allowed to conclude that while linear smoothing implies nonlinear smoothing, the viceversa is not true in general, see Theorem \ref{thm:OverviewBoundedness} and Remark \ref{rem.thm:OverviewBoundedness}. The crucial ingredient to prove the smoothing for the nonlinear (when the linear does not smooth) is the Green function method, developed in the previous sections.

So far, the panorama of implications does not include Green functions, only heat kernels. It can, however, be shown using Legendre transform that Sobolev and HLS are equivalent, see Lemma \ref{lem:SAndHLS}. Dual norms indeed involve Green functions, and an upper bound on the Green function implies the HLS, hence a Sobolev inequality. The Green function method thus replaces the use of Sobolev inequalities and iterations \`a la Moser or \`a la DeGiorgi with simpler integral estimates, and provides a solid alternative to those methods. Moreover, having at disposal estimates on the Green function seems to be more versatile in the sense that the method surprisingly works when the linear counterpart does not smooth. The latter must indicate that the Green function estimate cannot always provide strong functional inequalities which would imply smoothing in the linear case.

In many examples, the Green function estimates necessary for the method to work, are derived from heat kernel bounds, or via Fourier transform, see Section \ref{sec:GreenAndHeat}. As we shall explain below, in the nonlocal case, GNS is not sufficient to prove smoothing effects via Moser iterations. One also needs the Stroock-Varopoulos inequality \cite{Str84,Var85}, which somehow replaces the Sobolev chain rule in the local case.

One of the merits of this paper is represented from the fact that having Green function estimates---and then also boundedness estimates---allow to prove GNS (with the quadratic form already adjusted to the operator) that then have many other applications. In the nonlocal case, proving GNS is not an easy task for general quadratic forms, see \cite{DyFr12,BeFr14,ChKa20}. We provide here a PDE proof of many functional inequalities that can have their own interest. We also prove the validity of some weak GNS, as a consequence of the nonlinear smoothing. See Section \ref{sec:GreenAndHeat} for a rich list of examples of operators included in our theory.

Optimal or explicit constants in fractional GNS type inequalities is mostly an unexplored topic. In the local case, the sharp classical Nash inequality has been proven in \cite{CaLo93, BoDoSc20} by different methods. Sharp GNS have been proven in \cite{DPDo02} by entropy methods and nonlinear flows, and by mass transportation techniques in \cite{C-ENaVi04}. Quantitative and constructive stability for GNS has recently been proven in \cite{BoDoNaSi20}, to which we refer the reader for thorough historical and bibliographical information, also on Sobolev and related inequalities. We refrain from a thorough discussion here. As for functional inequalities related to nonlocal operators or fractional Sobolev spaces, to the best of our knowledge only a few contributions are present in literature: optimal fractional Sobolev inequalities are discussed in \cite{CotTa04}, while optimal fractional GNS in \cite{BeFr14}. Fractional Hardy inequalities are studied in \cite{DyFr12}. Improved Sobolev have been studied in \cite{PaPi14} by means of concentration compactness methods. We apologize in advance, in case we are missing important contributions in these directions, but in this paper we do not address the question of optimal inequalities, we just establish their validity with a (computable) constant.

Throughout this section, $C>0$ is a constant (that might change) which depends on $N$, $\alpha$, $m$, and the underlying Green function, but not on any norm of $u$ or $u_0$. We will use the notation
$$
\mathcal{Q}_{-\Operator}[f,g]:=\int f(-\Operator)[g]\qquad\text{and}\qquad \mathcal{Q}_{-\Operator}[f]:=\mathcal{Q}_{-\Operator}[f,f],
$$
and we will write, for $q>0$,
$$
\bigg(\int_{\R^N}|f(x)|^{q}\dd x\bigg)^{\frac{1}{q}}=\|f\|_{L^{q}(\R^N)}
$$
even though it is not a proper norm when $q\in (0,1)$.

\subsection{The well-known linear case (\texorpdfstring{$m=1$}{m equal 1})}
\label{sec:TheWell-KnownLinearCase}
We state and prove the following form of the Nash-type theorem \cite{Nas58}, adapted to our setting.

\begin{theorem}[Linear equivalences]\label{thm:LinearEquivalences}
Assume $\alpha\in (0,2]$ and $2^*=2N/(N-\alpha)$. The following statements are equivalent:
\begin{enumerate}[{\rm (a)}]
\item \textup{($L^1$--$L^\infty$-smoothing)} Let $u$ be a solution of \eqref{GPME} with $m=1$ and initial data $u_0\in L^1(\R^N)$, then
$$
\|u(\cdot, t)\|_{L^\infty(\R^N)}\leq Ct^{-\frac{N}{\alpha}}\|u_0\|_{L^1(\R^N)}.
$$
\item \textup{(Sobolev)} For all $f\in L^1(\R^N)\cap {\rm dom}(\mathcal{Q}_{-\Operator})$,
$$
\|f\|_{L^{2^*}(\R^N)}\leq C\mathcal{Q}_{-\Operator}[f]^{\frac{1}{2}}.
$$
\item \textup{(Nash)} For all $f\in L^1(\R^N)\cap {\rm dom}(\mathcal{Q}_{-\Operator})$,
$$
\|f\|_{L^{2}(\R^N)}\leq C\|f\|_{L^{1}(\R^N)}^\vartheta \mathcal{Q}_{-\Operator}[f]^{\frac{1}{2}(1-\vartheta)}\qquad\text{where}\qquad \vartheta:=\frac{1}{2}\frac{2^*-2}{2^*-1}.
$$
\item \textup{(On-diagonal heat kernel bounds)} The heat kernel $\mathbb{H}_{-\Operator}^{x_0}$ satisfies
$$
0\leq \mathbb{H}_{-\Operator}^{x_0}(x)\leq C t^{-\frac{N}{\alpha}}.
$$
\end{enumerate}
\end{theorem}

\begin{remark}\label{rem:LinearEquivalences}
The case $\alpha=2$ is well-known, and we refer the reader to e.g. Lemma 2.1.2 and Theorem 2.4.6 in \cite{Dav89} (see also \cite{LiLo01}). In the context of L\'evy operators $\Operator=\Levy^\mu$ with an absolutely continuous measure $\mu$, it is worth mentioning that as long as
$$
\frac{\dd \mu}{\dd z}(z)\gtrsim \frac{1}{|z|^{N+\alpha}}
$$
in \eqref{muas}, we are in the case $\alpha\in(0,2)$, cf. \cite[Proposition 2.6]{GrHuLa14}. One can also replace $N/\alpha$ by $\sum_{i=1}^N(\alpha_i)^{-1}$ as in the Sobolev inequality corresponding to the sum of onedimensional fractional Laplacians \cite[Theorem 2.4]{ChKa20}. Examples of (some of the above) equivalences in the nonpower case can be found in e.g. Proposition 3 and Lemma 5 in \cite{KnSc12}. We also refer to \cite{BrDP18} which explores various equivalences between Nash inequalities and $L^q$--$L^p$-smoothing estimates for L\'evy operators (see also \cite{CaKuSt87}).
\end{remark}

%%%%%%%%%%%%%%%%%%%%%%%%%%%%%%%%%%%%%%%%%%%%%%%%%%%%
%%%%%%%%%%%%%%%%%%%%%FIGURE%%%%%%%%%%%%%%%%%%%%%%%
%%%%%%%%%%%%%%%%%%%%%%%%%%%%%%%%%%%%%%%%%%%%%%%%%%%%
\begin{figure}[h!]
\centering
\begin{tikzpicture}
\color{black}

%Green function estimate
\node[draw,
    minimum width=2cm,
    minimum height=1.2cm,
    %fill=Rhodamine!50
] (l1) at (0,0){$\begin{array}{ll}\mathbb{G}_{-\Operator}^{x_0}(x)\\\lesssim |x-x_0|^{-(N-\alpha)}\end{array}$};

%Sobolev inequality
\node[draw,
    minimum width=2cm,
    minimum height=1.2cm,
    right=2cm of l1
] (r1){$\text{Sobolev inequality}$};

%L^1--L^\infty smoothing
\node [draw,
    minimum width=2cm,
    minimum height=1.2cm,
    below=1cm of r1
] (r2) {$L^1\text{--}L^\infty\text{-smoothing}$};

%Nash inequality
\node [draw,
    minimum width=2cm,
    minimum height=1.2cm,
    right=2cm of r2
] (rr2) {$\text{Nash inequality}$};

%On-diagonal heat kernel bounds
\node [draw,
    minimum width=2cm,
    minimum height=1.2cm,
    below=1cm of r2
] (r3) {$\begin{array}{cc}\text{On-diagonal}\\\text{heat kernel bounds}\end{array}$};

%Off-diagonal heat kernel bounds
\node [draw,
    minimum width=2cm,
    minimum height=1.2cm,
    left=2.1cm of r3
] (l2) {$\begin{array}{cc}\text{Off-diagonal}\\\text{heat kernel bounds}\end{array}$};

% Arrows with text label
\draw[-{Implies},double,line width=0.7pt] ( $ (l1.north east)!0.5!(l1.east) $ )  -- ( $ (r1.north west)!0.5!(r1.west) $ )
    node[midway,above]{$\text{Via HLS}$};

\draw[-{Implies},double,line width=0.7pt] (r1.east)  -| (rr2.north)
    node[pos=0.25,above]{$L^p\text{-interp.}$};

\draw[-{Implies},double,line width=0.7pt] (r2.north) -- (r1.south)
    node[midway,left]{$\text{Via HLS}$};

\draw[-{Implies},double,line width=0.7pt] ( $ (r2.south west)!0.5!(r2.south) $ )  -- ( $ (r3.north west)!0.56!(r3.north) $ )
    node[midway,left]{$\mathbb{H}_{-\Operator}^{x_0}(\cdot,0)=\delta_{x_0}$};

\draw[-{Implies},double,line width=0.7pt] ( $ (r2.south east)!0.5!(r2.east) $ )  -- ( $ (rr2.south west)!0.5!(rr2.west) $ )
    node[midway,above]{$\text{Energy est.}$};

\draw[-{Implies},double,line width=0.7pt] ( $ (rr2.north west)!0.5!(rr2.west) $ ) -- ( $ (r2.north east)!0.5!(r2.east) $ )
    node[midway,above]{$\text{Dual }L^\infty$};

\draw[-{Implies},double,line width=0.7pt] ( $ (r3.north east)!0.56!(r3.north) $ ) -- ( $ (r2.south east)!0.5!(r2.south) $ )
    node[midway,right]{$\begin{array}{cc}\text{Representation}\\\text{formula}\end{array}$};

\draw[-{Implies},double,line width=0.7pt] (l2.east) -- (r3.west)
    node[midway,above]{$\text{Definition}$};

\draw[-{Implies},double,line width=0.7pt] ( $ (l2.north east)!1.1!(l2.north) $ ) -- ( $ (l1.south east)!1.389!(l1.south) $ )
    node[midway,left]{$\text{Integration}$};

\normalcolor
\end{tikzpicture}
\caption{Implications in the linear case. Note that off-diagonal heat kernel bounds provide the strongest information unless we know how to deduce those bounds from the on-diagonal ones (like in \cite[Section 3]{Dav89} and \cite[Theorem 3.25]{CaKuSt87}). In the latter case, any piece of information is equivalent.}
\label{fig:ImplicationInLinearCase}
\end{figure}
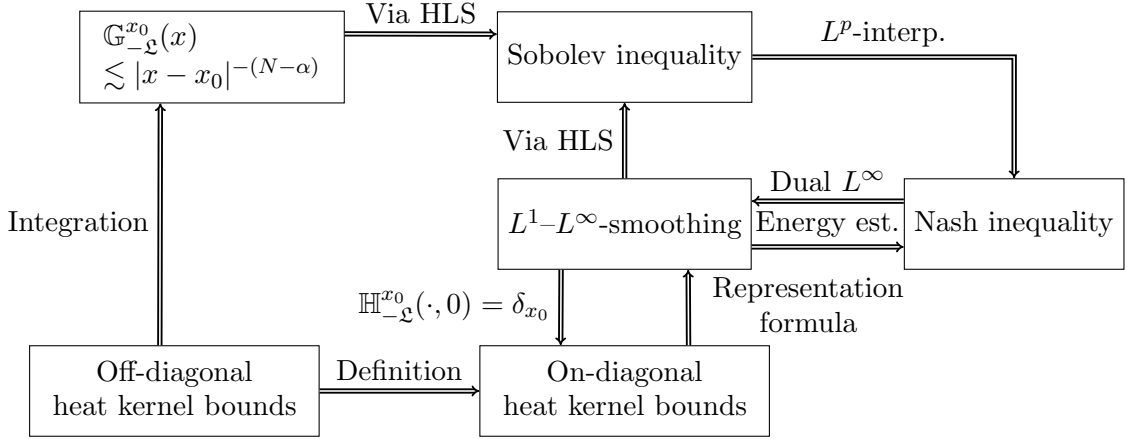

%%%%%%%%%%%%%%%%%%%%%%%%%%%%%%%%%%%%%%%%%%%%%%%%%%%%
%%%%%%%%%%%%%%%%%%%%%END FIGURE%%%%%%%%%%%%%%%%%%%%%%%
%%%%%%%%%%%%%%%%%%%%%%%%%%%%%%%%%%%%%%%%%%%%%%%%%%%%

The proof is divided into several independent results. By interpolation in $L^p$, we immediately have:

\begin{lemma}[Sobolev implies Nash]\label{lem:SobolevImplications}
Assume $\alpha\in(0,2]$ and $2^*=2N/(N-\alpha)$. If
$$
\|f\|_{L^{2^*}(\R^N)}\leq C\mathcal{Q}_{-\Operator}[f]^{\frac{1}{2}},
$$
then
$$
\|f\|_{L^{2}(\R^N)}\leq C\|f\|_{L^{1}(\R^N)}^\vartheta \mathcal{Q}_{-\Operator}[f]^{\frac{1}{2}(1-\vartheta)}\qquad\text{where}\qquad \vartheta:=\frac{1}{2}\frac{2^*-2}{2^*-1}.
$$
\end{lemma}

\begin{lemma}[$L^1$--$L^\infty$-smoothing VS Nash inequality]\label{lem:NashL1Linfty}
Under the assumptions of Theorem \ref{thm:LinearEquivalences}, the following are equivalent:
\begin{enumerate}[{\rm (a)}]
\item \textup{($L^1$--$L^\infty$-smoothing)}
$$
\|u(\cdot,t)\|_{L^\infty(\R^N)}\leq Ct^{-\frac{N}{\alpha}}\|u_0\|_{L^1(\R^N)}.
$$
\item \textup{(Nash)}
$$
\|f\|_{L^{2}(\R^N)}\leq C\|f\|_{L^{1}(\R^N)}^\vartheta \mathcal{Q}_{-\Operator}[f]^{\frac{1}{2}(1-\vartheta)}\qquad\text{where}\qquad \vartheta:=\frac{1}{2}\frac{2^*-2}{2^*-1}.
$$
\end{enumerate}
\end{lemma}

\begin{proof}
Follows by Theorem 8.16 in \cite{LiLo01}\footnote{We warn the reader about a small typo in the remark after Theorem 8.16 in \cite{LiLo01}: $(f,Lf)=\int fLf$ is indeed $\|\nabla f\|_{L^2}^2$ when $L=-\Delta$.} (see also Section 4.1 in \cite{S-Co02}). There, the Nash inequality is equivalent with an $L^1$--$L^\infty$-smoothing effect. An intermediate step is the $L^1$--$L^2$-smoothing effect, which can be extended to $L^\infty$ by the Nash duality trick:
\begin{equation*}
\begin{split}
\|u(t)\|_{L^\infty}&=\sup_{\|\phi\|_{L^1}=1}\bigg|\int u(t)\phi\bigg|=\sup_{\|\phi\|_{L^1}=1}\bigg|\int S_t[u_0]\phi\bigg|=\sup_{\|\phi\|_{L^1}=1}\bigg|\int S_{\frac{t}{2}}[S_{\frac{t}{2}}[u_0]]\phi\bigg|\\
&=\sup_{\|\phi\|_{L^1}=1}\bigg|\int S_{\frac{t}{2}}[u_0]S_{\frac{t}{2}}[\phi]\bigg|\leq \sup_{\|\phi\|_{L^1}=1}\|S_{\frac{t}{2}}[u_0]\|_{L^2}\|S_{\frac{t}{2}}[\phi]\|_{L^2}.\\
\end{split}
\end{equation*}
Here we used that the semigroup $S_t$ is self-adjoint, and the Cauchy-Schwarz inequality.
\end{proof}

\begin{remark}\label{rem:NashAndGNSAndSmoothing}
To obtain the $L^1$--$L^\infty$-smoothing in the nonlinear case ($m>1$), the Nash inequality is usually replaced by the Gagliardo-Nirenberg-Sobolev inequality:
$$
\|f\|_{L^{p}(\R^N)}\leq C\|f\|_{L^{q}(\R^N)}^\vartheta \mathcal{Q}_{-\Operator}[f]^{\frac{1}{2}(1-\vartheta)},
$$
where
$$
2\leq p<2^*,\qquad 1\leq q<p,\qquad \vartheta:=\frac{q}{p}\frac{2^*-p}{2^*-q}.
$$
Then the Moser iteration can be used to obtain the desired result. In the next section, we show that we indeed need less than the above inequality to perform all the necessary steps.
\end{remark}

\begin{lemma}[$L^1$--$L^\infty$-smoothing and heat kernel bounds]\label{lem:HeatL1Linfty}
Under the assumptions of Theorem \ref{thm:LinearEquivalences}, the following are equivalent:
\begin{enumerate}[{\rm (a)}]
\item \textup{($L^1$--$L^\infty$-smoothing)}
$$
\|u(\cdot,t)\|_{L^\infty(\R^N)}\leq Ct^{-\frac{N}{\alpha}}\|u_0\|_{L^1(\R^N)}.
$$
\item \textup{(On-diagonal heat kernel bounds)}
$$
0\leq \mathbb{H}_{-\Operator}^{x_0}(x,t)\leq C t^{-\frac{N}{\alpha}}
$$
\end{enumerate}
\end{lemma}

\begin{proof}
(a)$\Longrightarrow$(b). We apply Theorem 2.6.20 in \cite{Jac01}. Formally, $\mathbb{H}_{-\Operator}^{x_0}$ solves \eqref{GPME} with $m=1$ and $\delta_{x_0}$ as initial data. Hence, by an approximation argument and the lower semicontinuity of the $L^\infty$-norm, we arrive at part (b).

\medskip
\noindent(b)$\Longrightarrow$(a). Since  $C_\textup{c}^\infty(\R^N)$-initial data produce solutions that satisfy the representation formula $u(x,t)=\mathbb{H}_{-\Operator}^{x_0}(\cdot,t)\ast u_0(x)$ and
$$
|u(x,t)|\leq \mathbb{H}_{-\Operator}^{x_0}(\cdot,t)\ast |u_0|(x)\leq C t^{-\frac{N}{\alpha}}\|u_0\|_{L^1(\R^N)},
$$
we can again do an approximation argument to show (a).
\end{proof}

It remains to prove that the Nash inequality implies the Sobolev inequality. For $C_\textup{c}^\infty$-functions, such a result can be found in \cite{BaCoLeS-C95}. However, for semigroups in $L^2$, we will consider an indirect path through the inverse of the square root of the operator.

Legendre duality allows to establish equivalences between Sobolev and Hardy-Littlewood-Sobolev (HLS) inequalities:

\begin{lemma}[Sobolev VS HLS]\label{lem:SAndHLS}
Assume $\alpha\in(0,2]$, $2^*=2N/(N-\alpha)$, and $(2^*)':=2^*/(2^*-1)=2N/(N+\alpha)$. The following inequalities are equivalent:
\begin{enumerate}[{\rm (a)}]
\item \textup{(Sobolev)} For all $f\in L^1(\R^N)\cap {\rm dom}(\mathcal{Q}_{-\Operator})$,
$$
\|f\|_{L^{2^*}(\R^N)}\leq C\|(-\Operator)^{\frac{1}{2}}[f]\|_{L^2(\R^N)}.
$$
\item \textup{(HLS)} For all $g\in L^1(\R^N)\cap {\rm dom}(\mathcal{Q}_{-\Operator})$,
$$
\|(-\Operator)^{-\frac{1}{2}}[g]\|_{L^2(\R^N)}\leq C\|g\|_{L^{(2^*)'}(\R^N)}.
$$
\end{enumerate}
\end{lemma}

The proof is based on the Legendre transform, see e.g. Proposition 7.4 in \cite{BoSiVa15}, that we learned by Lieb \cite{Lie83}. Lemma \ref{lem:NashL1Linfty} already established that the Nash inequality implies the $L^1$--$L^\infty$-smoothing. The next lemma then finishes the proof of Theorem \ref{thm:LinearEquivalences}.

\begin{lemma}[$L^1$--$L^\infty$-smoothing VS HLS]
Assume $\alpha\in (0,2]$ and $(2^*)':=2^*/(2^*-1)=2N/(N+\alpha)$. Then the following are equivalent:
\begin{enumerate}[{\rm (a)}]
\item \textup{($L^1$--$L^\infty$-smoothing)} Let $u$ be a solution of \eqref{GPME} with $m=1$ and initial data $u_0\in L^1(\R^N)$, then
$$
\|u(\cdot,t)\|_{L^\infty(\R^N)}\leq Ct^{-\frac{N}{\alpha}}\|u_0\|_{L^1(\R^N)}.
$$
\item \textup{(HLS)} For all $g\in L^1(\R^N)\cap {\rm dom}(\mathcal{Q}_{-\Operator})$,
$$
\|(-\Operator)^{-\frac{1}{2}}[g]\|_{L^2(\R^N)}\leq C\|g\|_{L^{(2^*)'}(\R^N)}.
$$
\end{enumerate}
\end{lemma}

\begin{proof}
(a)$\Longrightarrow$ (b). We apply Theorem II.2.7 in \cite{VaSa-CoCo92} with $\zeta=\gamma=1$ and $p=(2^*)'$.

\medskip
\noindent(b)$\Longrightarrow$(a). Follows by Lemmas \ref{lem:SAndHLS}, \ref{lem:SobolevImplications}, and \ref{lem:NashL1Linfty}.
\end{proof}

Finally, we relate Green function estimates with all of the above equivalences.

\begin{lemma}[Green VS HLS]\label{lem:GreenFunctionGivesSobolev}
Assume  \eqref{Gas}, $\alpha\in (0,2]$, and $(2^*)'=2^*/(2^*-1)=2N/(N+\alpha)$. If
$$
0\leq \mathbb{G}_{-\Operator}^{x_0}(x)\leq C|x-x_0|^{-(N-\alpha)},
$$
then, for all $g\in L^1(\R^N)\cap {\rm dom}(\mathcal{Q}_{-\Operator})$,
$$
\|(-\Operator)^{-\frac{1}{2}}[g]\|_{L^2(\R^N)}\leq C\|g\|_{L^{(2^*)'}(\R^N)}.
$$
\end{lemma}

\begin{remark}
The above assumption on the Green function is stronger than \eqref{G_1}. We refer the reader to \cite{KaKiLe21} for a discussion on the validity of such an upper bound.
\end{remark}

\begin{proof}[Proof of Lemma \ref{lem:GreenFunctionGivesSobolev}]
This is essentially Theorem 7.5 in \cite{BoSiVa15}, which we restate here for completeness.

A direct calculation gives
\begin{equation*}
\begin{split}
&\|(-\Operator)^{-\frac{1}{2}}[f]\|_{L^2(\R^N)}^2=\int_{\R^N}(-\Operator)^{-\frac{1}{2}}[f](-\Operator)^{-\frac{1}{2}}[f]\dd x=\int_{\R^N}f(-\Operator)^{-1}[f]\dd x\\
&=\int_{\R^N}f(x)\bigg(\int_{\R^N}\mathbb{G}_{-\Operator}^{x}(y)f(y)\dd y\bigg)\dd x\leq C\int_{\R^N}f(x)\bigg(\int_{\R^N}|x-y|^{-(N-\alpha)}f(y)\dd y\bigg)\dd x\\
&=C\|(-\Delta)^{-\frac{\alpha}{4}}[f]\|_{L^2(\R^N)}^2.
\end{split}
\end{equation*}
The classical Hardy-Littlewood-Sobolev
$$
\|(-\Delta)^{-\frac{\alpha}{4}}[f]\|_{L^2(\R^N)}\leq C\|f\|_{L^{(2^*)'}(\R^N)}
$$
then provides the result.
\end{proof}

\subsection{The nonlinear case (\texorpdfstring{$m>1$}{m greater 1})}
\label{sec:EquivalencesGNSHomogeneous}
While in the linear case, the Nash method works perfectly, in the nonlinear case, the Nash method simply does not work since the ``nonlinear heat semigroup'' is not symmetric. On the other hand, the Moser iteration, which provides an alternative proof in the linear case, can be adapted to work also in the nonlinear case, and it shows how to prove smoothing effects from GNS inequalities also in the nonlinear setting. However GNS are not sufficient to perform Moser iteration in the nonlocal setting, another ingredient is needed: the so-called Stroock-Varopoulos inequalities. Let us briefly explain how this works.

%%%%%%%%%%%%%%%%%%%%%%%%%%%%%%%%%%%%%%%%%%%%%%%%%%%%
%%%%%%%%%%%%%%%%%%%%%FIGURE%%%%%%%%%%%%%%%%%%%%%%%
%%%%%%%%%%%%%%%%%%%%%%%%%%%%%%%%%%%%%%%%%%%%%%%%%%%%

\begin{figure}[h!]
\centering
\begin{tikzpicture}
\color{black}

%Hom. nonlinear L^1--L^\infty smoothing
\node[draw,
    minimum width=2cm,
    minimum height=1.2cm,
    %fill=Rhodamine!50
] (l1) at (0,0){$\begin{array}{cc}\text{Hom. nonlinear}\\L^1\text{--}L^\infty\text{-smoothing}\end{array}$};

%GNS inequality
\node[draw,
    minimum width=2cm,
    minimum height=1.2cm,
    right=2cm of l1
] (r1){$\text{GNS inequality}$};

%Sobolev inequality
\node[draw,
    minimum width=2cm,
    minimum height=1.2cm,
    right=2cm of r1
] (rr1){$\text{Sobolev inequality}$};

%Green function estimate
\node[draw,
    minimum width=2cm,
    minimum height=1.2cm,
    below=1cm of rr1
] (rr2){$\begin{array}{ll}\mathbb{G}_{-\Operator}^{x_0}(x)\\\lesssim |x-x_0|^{-(N-\alpha)}\end{array}$};

%Off-diagonal heat kernel bounds
\node [draw,
    minimum width=2cm,
    minimum height=1.2cm,
    below=1cm of rr2
] (rr3) {$\begin{array}{cc}\text{Off-diagonal}\\\text{heat kernel bounds}\end{array}$};

%Hom. linear L^1--L^\infty smoothing
\node [draw,
    minimum width=2cm,
    minimum height=1.2cm,
    below=1cm of r1
] (r2) {$\begin{array}{cc}\text{Hom. linear}\\L^1\text{--}L^\infty\text{-smoothing}\end{array}$};

%On-diagonal heat kernel bounds
\node [draw,
    minimum width=2cm,
    minimum height=1.2cm,
    below=1cm of r2
] (r3) {$\begin{array}{cc}\text{On-diagonal}\\\text{heat kernel bounds}\end{array}$};

%Resolvent condition
\node [draw,
    minimum width=2cm,
    minimum height=1.2cm,
    left=2cm of r3
] (l3) {$\begin{array}{cc}\eqref{G_3}\\\mathbb{G}_{I-\Operator}^{x_0}\in L^p(\R^N)\end{array}$};

%Nonhom. nonlinear L^1--L^\infty smoothing
\node [draw,
    minimum width=2cm,
    minimum height=1.2cm,
    above=1cm of l2
] (l2) {$\begin{array}{cc}\text{Nonhom. nonlinear}\\L^1\text{--}L^\infty\text{-smoothing}\end{array}$};

% Arrows with text label
\draw[-{Implies},double,line width=0.7pt] ( $ (l1.north east)!0.5!(l1.east) $ )  -- ( $ (r1.north west)!0.5!(r1.west) $ )
    node[midway,above]{$\text{\cite{BaCoLeS-C95}}$};

\draw[-{Implies},double,line width=0.7pt] ( $ (r1.south west)!0.5!(r1.west) $ ) -- ( $ (l1.south east)!0.5!(l1.east) $ )
    node[midway,below]{$\begin{array}{cc}\text{S.-V.}\\\text{and Moser}\end{array}$};

\draw[-{Implies},double,line width=0.7pt] ( $ (r1.north east)!0.5!(r1.east) $ ) -- ( $ (rr1.north west)!0.5!(rr1.west) $ )
    node[midway,above]{$\text{\cite{BaCoLeS-C95}}$};

\draw[-{Implies},double,line width=0.7pt]  ( $ (rr1.south west)!0.5!(rr1.west) $ ) -- ( $ (r1.south east)!0.5!(r1.east) $ )
    node[midway,below]{$\begin{array}{cc}L^p\text{-}\\\text{interp.}\end{array}$};

\draw[-{Implies},double,line width=0.7pt]  (rr2.north) -- (rr1.south)
    node[midway,right]{$\text{Via HLS}$};

\draw[-{Implies},double,line width=0.7pt]  (rr3.north) -- (rr2.south)
    node[midway,right]{$\text{Integration}$};

\draw[-{Implies},double,line width=0.7pt]  (rr3.west) -- (r3.east)
    node[midway,below]{$\text{Def.}$};

\draw[-{Implies},double,line width=0.7pt] ( $ (r1.south west)!0.355!(r1.south) $ )  -- ( $ (r2.north west)!0.5!(r2.north) $ )
    node[midway,right]{$\text{Nash}$};

\draw[-{Implies},double,line width=0.7pt] ( $ (r2.north east)!0.5!(r2.north) $ ) -- ( $ (r1.south east)!0.355!(r1.south) $ )
    node[midway,right]{$\text{\cite{BaCoLeS-C95}}$};

\draw[-{Implies},double,line width=0.7pt] ( $ (r2.south west)!0.545!(r2.south) $ )  -- ( $ (r3.north west)!0.56!(r3.north) $ )
    node[midway,left]{$\mathbb{H}_{-\Operator}^{x_0}(\cdot,0)=\delta_{x_0}$};

\draw[-{Implies},double,line width=0.7pt] ( $ (r3.north east)!0.515!(r3.north) $ ) -- ( $ (r2.south east)!0.5!(r2.south) $ )
    node[midway,right]{$\begin{array}{cc}\text{Representation}\\\text{formula}\end{array}$};

\draw[-{Implies},double,line width=0.7pt]  (r3.west) -- (l3.east)
    node[midway,below]{$\text{Integration}$};

\draw[-{Implies},double,line width=0.7pt] ( $ (l3.north east)!1.0!(l3.north) $ ) -- ( $ (l2.south east)!0.88!(l2.south) $ )
    node[midway,left]{$\text{Theorem \ref{thm:L1ToLinfinitySmoothing}}$};

\draw[-{Implies},double,line width=0.7pt] ( $ (l2.north east)!0.88!(l2.north) $ ) -- ( $ (l1.south east)!1.145!(l1.south) $ )
    node[midway,left]{$\begin{array}{cc}\text{Homogeneous}\\\text{operator}\end{array}$};

\normalcolor
\end{tikzpicture}
\caption{Implications in the nonlinear case. Note that still the off-diagonal heat kernel bounds provide the strongest piece of information. However, we also see that because of \eqref{G_3}, on-diagonal heat kernel bounds ensure a closed loop in the nonlinear case, assuming that \cite{BaCoLeS-C95} applies.}
\label{fig:ImplicationInNonlinearCase}
\end{figure}
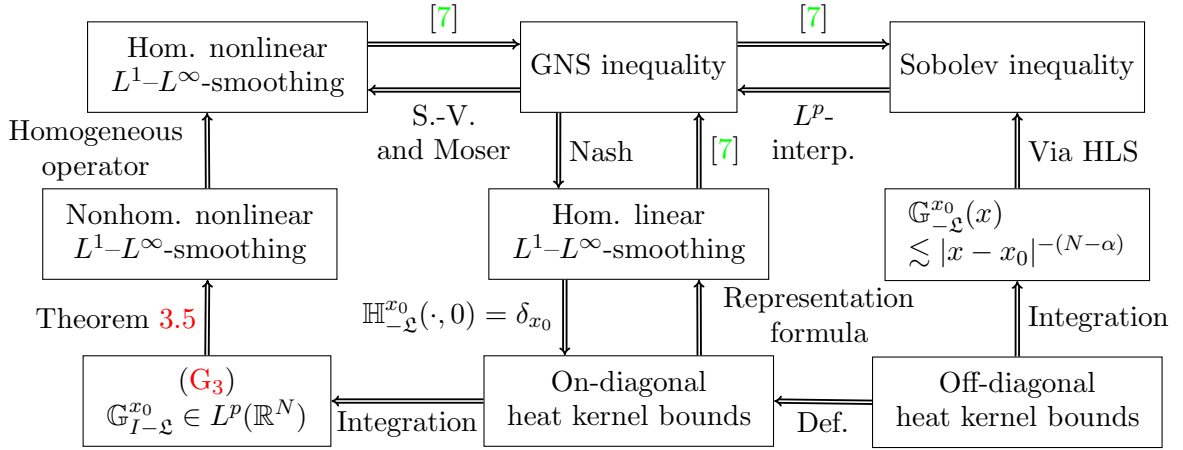

%%%%%%%%%%%%%%%%%%%%%%%%%%%%%%%%%%%%%%%%%%%%%%%%%%%%
%%%%%%%%%%%%%%%%%%%%%END FIGURE%%%%%%%%%%%%%%%%%%%%%%%
%%%%%%%%%%%%%%%%%%%%%%%%%%%%%%%%%%%%%%%%%%%%%%%%%%%%

Assume that there exists $2^*\ge 2$ such that the Sobolev-Poincar\'e type inequality holds ($2^*=2$ being the Poincar\'e case)
\begin{equation*}\label{Sobolev.Moser}
	\|f\|^2_{L^{2^*}(\R^N)}\leq C\int_{\R^N} f (-\Operator) f\dd x =C\mathcal{Q}_{-\Operator}[f]=C\|(-\Operator)^{\frac{1}{2}}[u]\|^2_{L^2(\R^N)},
\end{equation*}
where the last equality is true whenever the operator $-\Operator$ has an extension to $L^2(\R^N)$. By simple interpolation of $L^p$-norms, for $\tilde{p}\in [(1+m)/m,2)$ and $\tilde{q}\in[1/m,\tilde{p})$,
\begin{equation*}%\label{GNS.Moser}\tag{GNS-Moser}
\|f\|_{L^{\tilde{p}}(\R^N)}\leq C\|f\|_{L^{\tilde{q}}(\R^N)}^\vartheta \mathcal{Q}_{-\Operator}[f]^{\frac{1}{2}(1-\vartheta)}\quad\text{where}\quad \vartheta:=\vartheta(2^*).
\end{equation*}
When dealing with energy estimates in the local case, a calculus equality allows us to do the Moser iteration:
\[
\int_{\R^N} u^{p-1} (-\Delta)[ u^m]\dd x = \int_{\R^N} \nabla u^{p-1} \cdot\nabla u^m \dd x
= \frac{4m(p-1)}{(p+m-1)^2}\int_{\R^N} \big|\nabla u^{\frac{p+m-1}{2}} \big|^2 \dd x.
\]
However, we just need an inequality, which in the nonlocal case has been proven by Stroock and Varopoulos \cite{Str84,Var85} (cf. \cite[Proposition 4.11]{BrDP18} or \cite[Lemma 4.10]{DTEnJa18a}): For the same constant as above,
\begin{equation*}\label{str-var.Moser}
\int_{\R^N} u^{p-1}(-\Operator)[u^m] \dd x  \gtrsim \int_{\R^N} u^{\frac{p+m-1}{2}}(-\Operator)[u^{\frac{p+m-1}{2}}]\dd x \eqsim \big\|(-\Operator)^{\frac{1}{2}}[u^{\frac{p+m-1}{2}}]\big\|^2_{L^2(\R^N)}.
\end{equation*}
Combining the two above inequalities, one gets
\begin{equation}\label{M}\tag{M}
\int u^{p-1}(-\Operator)[u^m]\ge \frac{4m(p-1)}{(p+m-1)^2}\mathcal{Q}_{-\Operator}[u^{\frac{p+m-1}{2}}]
\ge\frac{4m(p-1)}{C^2(p+m-1)^2}\frac{\|u\|_{L^{\tilde{p}\frac{p+m-1}{2}}}^{\frac{2}{1-\vartheta}\frac{p+m-1}{2}}}
{\|u\|_{L^{\tilde{q}\frac{p+m-1}{2}}}^{\frac{2\vartheta}{1-\vartheta}\frac{p+m-1}{2}}}.
\end{equation}

The above condition is the key to prove the following:

\begin{theorem}[Green functions satisfying \eqref{G_1}]\label{thm:GNSEquivalentL1Linfty}
Assume \eqref{phias}. Then
the following statements are equivalent:
\begin{enumerate}[{\rm (i)}]
\item \textup{($L^1$--$L^\infty$-smoothing)} Let $u$ be a weak dual solution of \eqref{GPME} with initial data $u_0\in L^1(\R^N)$, then
$$
\|u(\cdot,t)\|_{L^\infty(\R^N)}\leq Ct^{-N\theta_1}\|u_0\|_{L^1(\R^N)}^{\alpha\theta_1},
$$
where $\theta_1=(\alpha+N(m-1))^{-1}$.
\item \textup{(Subcritical GNS)} For $\tilde{p}\in [(1+m)/m,2)$ and $\tilde{q}\in[1/m,\tilde{p})$, and for all $f\in L^{\tilde{q}}(\R^N)\cap {\rm dom}(\mathcal{Q}_{-\Operator})$ we have
$$
\|f\|_{L^{\tilde{p}}(\R^N)}\leq C\|f\|_{L^{\tilde{q}}(\R^N)}^\vartheta \mathcal{Q}_{-\Operator}[f]^{\frac{1}{2}(1-\vartheta)}\quad\text{where}\quad \vartheta:=\frac{\tilde{q}}{\tilde{p}}\frac{2^*-\tilde{p}}{2^*-\tilde{q}}.
$$
\end{enumerate}
\end{theorem}

When we have at our disposal Green functions satisfying \eqref{G_1} then, by Theorem \ref{thm:L1ToLinfinitySmoothing2}(a), we have the above nonlinear smoothing effect. This turns out to be equivalent to a family of subcritical Gagliardo-Nirenberg-Sobolev, which then are equivalent to Sobolev, see Lemma \ref{lem:NashEquivalentGNS} below. In order to show that subcritical GNS imply nonlinear smoothing, we will perform a Moser iteration. We refer to \cite[Section 3]{BoIbIs22} for a more detailed exposition of the Moser iteration in the fast diffusion case $0<m<1$, in the context of bounded domains. It also contains a detailed discussion about the Green function method versus the Moser iteration.

When we have integrable Green functions \eqref{G_2}, we can obtain absolute bounds (i.e. independent of the initial datum), as in the case of bounded domains \cite{BoVa15,BoVa16,BoFiVa18a}. Such bounds imply weak GNS inequalities, which are equivalent to Poincaré inequalities (Lemma \ref{lem:NashEquivalentGNS}). We notice that it is not possible (to the best of our knowledge) to prove the converse implication via the Moser iteration. In fact, the constant simply blows up at the limit $p\to\infty$. A similar discussion can be found in \cite{Gri10, GrMuPo13}. However we have seen in Theorem \ref{thm:AbsBounds} a simple proof of the absolute bounds with the Green function method, so that we can conclude that integrable Green functions imply Poincar\'e-type inequalities as follows:

\begin{proposition}[Green functions satisfying \eqref{G_2}]\label{abs.VS.poinc}
Assume \eqref{phias}. Given the following statements:
\begin{enumerate}[{\rm (i)}]
\item \textup{(Absolute bound)} Let $u$ be a weak dual solution of \eqref{GPME} with initial data $u_0$, then
$$
\|u(\cdot,t)\|_{L^\infty(\R^N)}\leq Ct^{-1/(m-1)}.
$$
\item \textup{(Subcritical GNS)} For $\tilde{p}\in [(1+m)/m,2)$ and $\tilde{q}\in[1/m,\tilde{p})$, and for all $f\in L^{\tilde{q}}(\R^N)\cap {\rm dom}(\mathcal{Q}_{-\Operator})$ we have
$$
\|f\|_{L^{\tilde{p}}(\R^N)}\leq C\|f\|_{L^{\tilde{q}}(\R^N)}^\vartheta \mathcal{Q}_{-\Operator}[f]^{\frac{1}{2}(1-\vartheta)}\quad\text{where}\quad \vartheta:=\frac{\tilde{q}}{\tilde{p}}\frac{2-\tilde{p}}{2-\tilde{q}}.
$$
\end{enumerate}
Then (i)$\Longrightarrow$(ii).
\end{proposition}

In order to prove the above theorem and proposition, we need a few results that we prefer to state and prove separately since they have their own interest. We shall start with the fact that (any) GNS is equivalent to some $L^q$--$L^p$-smoothing. This can be directly seen by the Stroock-Varopoulos inequality.

\begin{proposition}[$L^q$--$L^p$-smoothing VS subcritical GNS]\label{prop:LqLpsmoothing}
Assume \eqref{phias}.
\begin{enumerate}[{\rm (a)}]
\item \textup{(Green functions satisfying \eqref{G_1})} The following statements are equivalent:
\begin{enumerate}[{\rm (i)}]
\item \textup{($L^q$--$L^{p}$-smoothing)} For $p\in[1+m,\infty)$ and $q\in[1,p)$, let $u$ be a weak dual solution of \eqref{GPME} with initial data $u_0\in L^q(\R^N)$, then
$$
\|u(\cdot,t)\|_{L^p(\R^N)}\leq C\bigg(\frac{(p+m-1)^2}{4m(m-1)(p-1)}\frac{1}{t}\bigg)^{\frac{N(p-q)\theta_{q}}{p}}\|u_0\|_{L^{q}(\R^N)}^{\frac{q}{p}\frac{\theta_{q}}{\theta_p}},
$$
where $\theta_r:=(\alpha r+N(m-1))^{-1}$ and $C>0$ is independent of $p,q$.
\item \textup{(Subcritical GNS)} For $\tilde{p}\in [(1+m)/m,2)$ and $\tilde{q}\in[1/m,\tilde{p})$, and for all $f\in L^{\tilde{q}}(\R^N)\cap {\rm dom}(\mathcal{Q}_{-\Operator})$ we have
$$
\|f\|_{L^{\tilde{p}}(\R^N)}\leq C\|f\|_{L^{\tilde{q}}(\R^N)}^\vartheta \mathcal{Q}_{-\Operator}[f]^{\frac{1}{2}(1-\vartheta)}\quad\text{where}\quad \vartheta:=\frac{\tilde{q}}{\tilde{p}}\frac{2^*-\tilde{p}}{2^*-\tilde{q}}.
$$
\end{enumerate}
\item \textup{(Green functions satisfying \eqref{G_2})} The following statements are equivalent:
\begin{enumerate}[{\rm (i)}]
\item \textup{($L^q$--$L^{p}$-smoothing)} For $p\in[1+m,\infty)$ and $q\in[1,p)$, let $u$ be a weak dual solution of \eqref{GPME} with initial data $u_0\in L^q(\R^N)$, then
$$
\|u(\cdot,t)\|_{L^p(\R^N)}\leq C\bigg(\frac{(p+m-1)^2}{4m(m-1)(p-1)}\frac{1}{t}\bigg)^{\frac{p-q}{p(m-1)}}\|u_0\|_{L^{q}(\R^N)}^{\frac{q}{p}},
$$
where $C>0$ is independent of $p,q$.
\item \textup{(Subcritical GNS)} For $\tilde{p}\in [(1+m)/m,2)$ and $\tilde{q}\in[1/m,\tilde{p})$, and for all $f\in L^{\tilde{q}}(\R^N)\cap {\rm dom}(\mathcal{Q}_{-\Operator})$ we have
$$
\|f\|_{L^{\tilde{p}}(\R^N)}\leq C\|f\|_{L^{\tilde{q}}(\R^N)}^\vartheta \mathcal{Q}_{-\Operator}[f]^{\frac{1}{2}(1-\vartheta)}\quad\text{where}\quad \vartheta:=\frac{\tilde{q}}{\tilde{p}}\frac{2-\tilde{p}}{2-\tilde{q}}.
$$
\end{enumerate}
\end{enumerate}
\end{proposition}

\begin{remark}
\begin{enumerate}[{\rm (a)}]
\item Let us make some comments on part (a). First of all, note that $\vartheta$ is nothing but the standard quantity appearing in interpolation between $L^p$-norms with $\tilde{q}<\tilde{p}<2^*$. We also see that when formally $m\to1^{-}$, $\tilde{p}=2$ and $\tilde{q}=1$ in the above subcritical GNS inequality, and we recover the critical Nash inequality (see Section \ref{sec:TheWell-KnownLinearCase}). In our case $m>1$, and that is why we call it subcritical. Note, however, that the standard GNS inequality (see Remark \ref{rem:NashAndGNSAndSmoothing}) is not included as a special case here.
\item
The proof reveals that GNS inequalities always imply $L^q$--$L^{p}$-smoothing effects, but the opposite implication requires further assumptions. Actually, the equivalence which is always true is the one between $L^1$--$L^{m+1}$-smoothing effects and subcritical Nash inequalities. In fact, even operators only yielding boundedness estimates in the form of Theorem \ref{thm:L1ToLinfinitySmoothing} (see also Theorem \ref{thm:L1ToLinfinitySmoothing0}), still enjoy the latter equivalence.
\end{enumerate}
\end{remark}

To provide a proof, we need:
\begin{lemma}[\cite{BaCoLeS-C95, S-Co02}]\label{lem:NashEquivalentGNS}
Assume \eqref{phias} and $f\in C_\textup{c}^\infty(\R^N)$. Then:
\begin{enumerate}[{\rm (a)}]
\item \textup{(Sobolev)} The following statements are equivalent:
\begin{enumerate}[{\rm (i)}]
\item \textup{(Sobolev)}
$$
\|f\|_{L^{2^*}(\R^N)}\leq C\mathcal{Q}_{-\Operator}[f]^{\frac{1}{2}}.
$$
\item \textup{(Subcritical GNS)} For $\tilde{p}\in [(1+m)/m,2)$ and $\tilde{q}\in[1/m,\tilde{p})$,
$$
\|f\|_{L^{\tilde{p}}(\R^N)}\leq C\|f\|_{L^{\tilde{q}}(\R^N)}^\vartheta \mathcal{Q}_{-\Operator}[f]^{\frac{1}{2}(1-\vartheta)}\quad\text{where}\quad \vartheta:=\frac{\tilde{q}}{\tilde{p}}\frac{2^*-\tilde{p}}{2^*-\tilde{q}}.
$$
\end{enumerate}

\item \textup{(Poincar\'e)} The following statements are equivalent:
\begin{enumerate}[{\rm (i)}]
\item \textup{(Poincar\'e)}
$$
\|f\|_{L^{2}(\R^N)}\leq C\mathcal{Q}_{-\Operator}[f]^{\frac{1}{2}}.
$$
\item \textup{(Subcritical GNS)} For $\tilde{p}\in [(1+m)/m,2)$ and $\tilde{q}\in[1/m,\tilde{p})$,
$$
\|f\|_{L^{\tilde{p}}(\R^N)}\leq C\|f\|_{L^{\tilde{q}}(\R^N)}^\vartheta \mathcal{Q}_{-\Operator}[f]^{\frac{1}{2}(1-\vartheta)}\quad\text{where}\quad \vartheta:=\frac{\tilde{q}}{\tilde{p}}\frac{2-\tilde{p}}{2-\tilde{q}}.
$$
\end{enumerate}
\end{enumerate}
\end{lemma}

\begin{remark}
\begin{enumerate}[{\rm (a)}]
\item In both cases, we simply check that $q$ in \cite{BaCoLeS-C95} is respectively given by $2^*$ and $2$. It is also worth noting, that ``any'' family of Sobolev/Poincar\'e-type inequalities is equivalent with the Sobolev/Poincar\'e inequality. As a consequence, subcritical GNS inequalities are equivalent with subcritical Nash inequalities, and then also equivalent with the standard Nash and GNS inequalities, respectively. However, the subcritical ones might be easier to prove in the nonlinear setting.
\item The case $q=2$ is kind of curious since it yields a Poincar\'e inequality in $\R^N$. Such an inequality is not fulfilled in e.g. the case $\Operator=\Delta$ as the spectrum is nonnegative. Hence, it provides the intuition that absolute bounds holds if the spectrum of the operator is positive (e.g. as in the case $-\Operator=(I-\Delta)^{\frac{\alpha}{2}}$).
\end{enumerate}
\end{remark}

\begin{proof}[Proof of Proposition \ref{prop:LqLpsmoothing}]
\noindent(a) (i)$\Longrightarrow$(ii). By the particular choices $1=q<p=m+1$, we immediately have the corresponding $L^1$--$L^{m+1}$-smoothing effect. Then, direct and formal\footnote{We have decided to present the differential version of these estimates because the main idea is easier to follow. This is rigorous for strong solutions, for instance when $\partial_tu\in L^1$. It often happens that bounded weak solutions are strong, possibly under some additional assumptions on $\Operator$, see \cite{DPQuRoVa12}. These formal computations can be justified rigorously in several different ways. One possibility is to use Steklov averages and Gr\"onwall-type inequalities. It is beyond the scope of this paper to justify these computations, but we remark that they can be shown to hold for the mild solutions constructed in Appendix \ref{sec:ExistenceAPriori},  through standard approximations. The energy computations, namely the ones involving the $L^{m+1}$ norm, are always true for weak energy solutions (of which mild solutions are a subclass of).} computations show that
\begin{equation}\label{eq:decayOfLpNorms}
\frac{\dd}{\dd t}\int u^{m+1}=\int \dell_t(u^{m+1})=(m+1)\int u^{m}\dell_tu=-(m+1)\int u^{m}(-\Operator)[u^m],
\end{equation}
and also, since $\Operator$ is symmetric and $u$ solves \eqref{GPME},
\begin{equation}\label{eq:energyDecay}
\begin{split}
\frac{\dd}{\dd t}\int u^m(-\Operator)[u^m]&=2\int\dell_t(u^m)(-\Operator)[u^m]=2m\int u^{m-1}\dell_tu(-\Operator)[u^m]\\
&=-2m\int u^{m-1}\Big((-\Operator)[u^m]\Big)^2\leq 0.
\end{split}
\end{equation}
Hence, by \eqref{eq:energyDecay}, we obtain in \eqref{eq:decayOfLpNorms} that
\begin{equation*}
\frac{\dd}{\dd t}\int u^{m+1}\geq -(m+1)\int u_0^m(-\Operator)[u_0^m]=-(m+1)\mathcal{Q}_{-\Operator}[u_0^m].
\end{equation*}
We then integrate over $(0,T)$ and use $L^1$--$L^{m+1}$-smoothing effect to get
\begin{equation*}
\begin{split}
-(m+1)Q_{-\Operator}[u_0^m]T
&\leq \|u(T)\|_{L^{m+1}}^{m+1}-\|u_0\|_{L^{m+1}}^{m+1}\leq CT^{-N((m+1)-1)\theta_1}\|u_0\|_{L^1}^{\frac{\theta_1}{\theta_{m+1}}}-\|u_0\|_{L^{m+1}}^{m+1},
\end{split}
\end{equation*}
or, by taking $f:=u_0^m$,
$$
\|f\|_{L^{\frac{m+1}{m}}}^{\frac{m+1}{m}}\leq F(T):=C\|f\|_{L^{\frac{1}{m}}}^{\frac{1}{m}\frac{\theta_1}{\theta_{m+1}}}T^{-N((m+1)-1)\theta_1}+(m+1)Q_{-\Operator}[f]T.
$$
The inequality is still valid if we infimize $F$ over $T>0$, and this gives the subcritical Nash inequality, i.e., $\tilde{p}=(1+m)/m$ and $\tilde{q}=1/m$ in the stated subcritical GNS. Since the subcritical Nash inequality is a subfamily of the subcritical GNS, it is equivalent with the Sobolev inequality and then equivalent with GNS by Lemma \ref{lem:NashEquivalentGNS}.

\smallskip
\noindent(a) (ii)$\Longrightarrow$(i). Note that (ii) with $f=u^{\frac{p+m-1}{2}}$ and the Stroock-Varopoulos inequality gives \eqref{M}. Direct and formal\footnote{Again, it is beyond the scope of this paper to justify these computations, but they hold for e.g. the mild solutions constructed in Appendix \ref{sec:ExistenceAPriori} under some possibly additional assumptions on $\Operator$. See also the previous footnote.} calculations and the $L^p$-decay (Proposition \ref{prop:APriori}(b)(ii)) give
\begin{equation}\label{eq:Lp-Lq.diff.ineq}
\begin{split}
\frac{\dd}{\dd t}\int u^p&=\int \dell_t(u^p)=p\int u^{p-1}\dell_tu=-p\int u^{p-1}(-\Operator)[u^m]\\
&\le -\frac{4mp(p-1)}{C^2(p+m-1)^2}\frac{\|u\|_{L^{\tilde{p}\frac{p+m-1}{2}}}^{\frac{2}{1-\vartheta}\frac{p+m-1}{2}}}
{\|u_0\|_{L^{\tilde{q}\frac{p+m-1}{2}}}^{\frac{2\vartheta}{1-\vartheta}\frac{p+m-1}{2}}}
= -\frac{4mp(p-1)}{C^2(p+m-1)^2}\frac{\|u\|_{L^{p}}^{1+\sigma}}
{\|u_0\|_{L^{q}}^{\frac{2(2^*-\tilde{p})}{2^*(\tilde{p}-\tilde{q})}q}}
\end{split}
\end{equation}
Where we have chosen:
\[
\tilde{p}:=\frac{2p}{p+m-1},\qquad \tilde{q}:= \frac{2q}{p+m-1}\qquad\mbox{and}\qquad \sigma:=\frac{2^*(2-\tilde{p})+\tilde{q}(2^*-2)}{2^*(\tilde{p}-\tilde{q})},
\]
and this choice is consistent with our assumptions. We also have that
$$
\frac{2}{1-\vartheta}\frac{p+m-1}{2}=\frac{2(2^*-\tilde{q})}{2^*(\tilde{p}-\tilde{q})}p \qquad\text{and}\qquad \frac{2\vartheta}{1-\vartheta}\frac{p+m-1}{2}=\frac{2(2^*-\tilde{p})}{2^*(\tilde{p}-\tilde{q})}q.
$$
so that, integrating the differential inequality we get (i). The proof of part (a) is concluded.

\smallskip
\noindent(b) (i)$\Longrightarrow$(ii). Follows by a similar argument as in (a) (i)$\Longrightarrow$(ii), except that
$$
F(T):=C\|f\|_{L^{\frac{1}{m}}}^{\frac{1}{m}}T^{-\frac{m}{m-1}}+(m+1)Q_{-\Operator}[f]T.
$$

\smallskip
\noindent(b) (ii)$\Longrightarrow$(i). We argue exactly as in (a) (ii)$\Longrightarrow$(i), but now
$$
\frac{2}{1-\vartheta}\frac{p+m-1}{2}=\frac{2-\tilde{q}}{\tilde{p}-\tilde{q}}p, \qquad \frac{2\vartheta}{1-\vartheta}\frac{p+m-1}{2}=\frac{2-\tilde{p}}{\tilde{p}-\tilde{q}}q,
$$
$$
\frac{1}{\sigma}=\frac{p-q}{m-1},\qquad\text{and}\qquad \frac{2-\tilde{p}}{\tilde{p}-\tilde{q}}\frac{1}{\sigma}=1.
$$
This yields the desired estimate.
\end{proof}

\begin{proof}[Proof of Theorem \ref{thm:GNSEquivalentL1Linfty}]
\noindent(ii)$\Longrightarrow$(i). In what follows, we will just sketch the essential parts of the proof, in order to focus on the main ideas. The proof can moreover be made rigorous by standard approximation techniques. Let us first remark that it is enough to prove the following $L^{m+1}$--$L^\infty$-smoothing effect:
\begin{equation}\label{Moser.m+1}
\|u(t)\|_{L^\infty(\R^N)}\leq C t^{-N\theta_{m+1}}\|u_0\|_{L^{m+1}(\R^N)}^{\alpha(m+1)\theta_{m+1}}.
\end{equation}
Indeed, the claimed $L^1$--$L^\infty$-smoothing effects is then deduced by applying Lemma \ref{lem:SmoothingImplications} of Appendix \ref{sec:SmoothingImplications}.

In order to prove the smoothing effect \eqref{Moser.m+1}, we will iterate ``Moser style'' the $L^p$--$L^q$-smoothing effects Proposition \ref{prop:LqLpsmoothing}(a)(i): Let us define $p_0=m+1$ and $p_k=2^kp_0$ for each $k\geq 1$, and $t_k$ such that $t_k-t_{k-1}=\frac{t-t_0}{2^k}$, so that inequality Proposition \ref{prop:LqLpsmoothing}(a)(i)  becomes
	\begin{align*}\label{Moser}
	\|u(t_k)\|_{L^{p_k}}\leq
	I_{k}^{\frac{N(p_k-p_{k-1})}{p_k}\theta_{k-1}}\|u(t_{k-1})\|_{L^{p_{k-1}}}^{\frac{p_{k-1}\,\theta_{k-1}}{p_k\,\theta_k}}
                        \qquad\mbox{with}\qquad I_k  \eqsim \frac{p_k}{t_k-t_{k-1}}\eqsim 4^k
	\end{align*}
	where $\theta_k:=\theta_{p_k}=(\alpha p_k+ N(m-1))^{-1}$. More precisely, we have that we can estimate $I_k$ uniformly as follows:
\[
I_k:= C \frac{(p_k+m-1)^2}{4m(m-1)(p_k-1)}\frac{1}{t_k-t_{k-1}}
\le 4^k\frac{C}{t-t_0},
\]
for some constant $C>0$ that depends only on $m,N$.

Then we iterate
	\begin{align*}
	\|u(t_k)\|_{L^{p_k}}&\leq
I_{k}^{\frac{N(p_k-p_{k-1})}{p_k}\theta_{k-1}}
\|u(t_{k-1})\|_{L^{p_{k-1}}}^{\frac{p_{k-1}\,\theta_{k-1}}{p_k\,\theta_k}}\\	
&\le I_k^{\frac{N(p_k-p_{k-1})}{p_k}\theta_{k-1}}I_{k-1}^{\frac{N(p_{k-1}-p_{k-2})}{\cancel{p_{k-1}}}\theta_{k-2}\frac{\cancel{p_{k-1}}\theta_{k-1}}{p_k\theta_k}}
    \|u(t_{k-2})\|_{L^{p_{k-2}}}^{\frac{p_{k-2}\theta_{k-2}}{\cancel{p_{k-1}}\cancel{\theta_{k-1}}}\frac{\cancel{p_{k-1}}\,\cancel{\theta_{k-1}}}{p_k\,\theta_k}}\\
	&\;\;\vdots\\
    &\le \prod_{j=1}^{k} I_j^{\frac{N(p_j-p_{j-1})}{p_k}\frac{\theta_j\theta_{j-1}}{\theta_k}}\;\|u(t_0)\|_{L^{p_0}}^{\frac{p_0\,\theta_{p_0}}{p_k\theta_{k}}}
 \le\left[\prod_{j=1}^{k}\bigg(4^j\frac{\overline{c}}{t-t_0}\bigg)^{\frac{N(\theta_{j-1}-\theta_{j})}{\alpha}}\right]^{\frac{1}{p_k\theta_{k}}} \!\!\!\!\!\! \|u(t_0)\|_{L^{p_0}}^{\frac{p_0\,\theta_{p_0}}{p_k\theta_{k}}}.
	\end{align*}
Finally, letting $k\to \infty$, it is easy to see that
\[
\prod_{j=1}^{k}\bigg(4^j\frac{C}{t-t_0}\bigg)^{\frac{N(\theta_{j-1}-\theta_{j})}{\alpha}}
\le 2^\frac{N}{\alpha^2p_0\,p_k\theta_{k}}
        \left(\frac{C}{t-t_0}\right)^{\frac{N(\theta_{p_0}-\theta_{k})}{\alpha}}
\]
so that, using the lower semicontinuity of the $L^\infty$ norm, we get
\begin{align*}
	\|u(t)\|_{L^\infty}&\leq\lim\limits_{k\rightarrow\infty}\|u(t_k)\|_{L^{p_k}}\nonumber\\
	&\leq\lim_{k\rightarrow\infty}\left(2^\frac{N}{\alpha^2{p_0}\,p_k\theta_{k}}
        \left(\frac{C}{t-t_0}\right)^{\frac{N(\theta_{p_0}-\theta_{k})}{\alpha}}\right)^{\frac{1}{p_k\theta_{k}}}
       \|u(t_0)\|_{L^{p_0}}^{\frac{p_0\,\theta_{p_0}}{p_k\theta_{k}}}\nonumber
	\leq C \frac{\|u(t_0)\|_{L^{p_0}}^{\alpha p_0\,\theta_{p_0}}}{(t-t_0)^{N\theta_{p_0}}}\,.
	\end{align*}
This proves the desired inequality \eqref{Moser.m+1} and concludes the proof  of (ii)$\Longrightarrow$(i).

\smallskip
\noindent (i)$\Longrightarrow$(ii). Follows by Theorem \ref{thm:SmoothingImplications} of Appendix \ref{sec:SmoothingImplications},  which states that $L^1$--$L^\infty$ imply $L^p$--$L^q$-smoothing effects, which in turn imply subcritical GNS, by Proposition \ref{prop:LqLpsmoothing}(a). This concludes the proof.
\end{proof}
\begin{remark}This proof holds for all $m\ge 1$, so in particular it also shows that subcritical GNS imply smoothing also in the linear case $m=1$, providing an alternative proof of the implication (b)$\Longrightarrow$(a) or (c)$\Longrightarrow$(a) in Theorem \ref{thm:LinearEquivalences}. When $m\in (0,1)$, which corresponds to the fast diffusion case, the same proof works as well, but we need to require further integrability on the initial datum in order to perform the iteration, as thoroughly explained in e.g. \cite[Section 3]{BoIbIs22} (and also \cite{DaKe07,BoVa10, BoSi19} for the local case).
\end{remark}

Finally, we have:

\begin{proof}[Proof of Proposition \ref{abs.VS.poinc}]
By Lemma \ref{lem:SmoothingImplications} (with $\gamma=0$) of Appendix \ref{sec:SmoothingImplications}, item (i) implies Proposition \ref{prop:LqLpsmoothing}(b)(i), and hence, item (ii) holds.
\end{proof}

%%%%%%%%%%%%%%%%%%%%%%%%%%%%%%%%%%%%%%%%%%%%%%%%%%%%
%%%%%%%%%%%%%%%%%%%%%NEW SECTION%%%%%%%%%%%%%%%%%%%%%%%
%%%%%%%%%%%%%%%%%%%%%%%%%%%%%%%%%%%%%%%%%%%%%%%%%%%%

\section{Various examples}
\label{sec:GreenAndHeat}

This section is devoted to study the operators whose Green functions satisfy assumptions \eqref{G_1}--\eqref{G_3}, and hence, which smoothing effects are satisfied by such operators. As a consequence of Proposition \ref{prop:TheInverseOperatorA-1}, we have:

\begin{proposition}\label{prop:GreenFormulas}
Assume that the operator $-\Operator$ is linear, symmetric, nonnegative, and moreover, densely defined, $\mathfrak{m}$-accretive, and Dirichlet in $L^1(\R^N)$. Then \eqref{Gas} holds, and the Green functions of $-\Operator$ and $I-\Operator$ are respectively given by
$$
\mathbb{G}_{-\Operator}^{x_0}(x)=\int_{0}^\infty \mathbb{H}_{-\Operator}^{x_0}(x,t)\dd t
$$
and
$$
\mathbb{G}_{I-\Operator}^{x_0}(x)=\int_{0}^\infty \e^{-t}\mathbb{H}_{-\Operator}^{x_0}(x,t)\dd t,
$$
where $\mathbb{H}_{-\Operator}^{x_0}$ is the corresponding heat kernel of $-\Operator$.
\end{proposition}

Let us illustrate these formulas through Fourier analysis. Denote by $\sigma_{-\Operator}$ the Fourier symbol of the operator $-\Operator$. Then the heat kernel can be expressed as
$$
\mathbb{H}_{-\Operator}^{x_0}(x,t)=\mathcal{F}^{-1}\big[\e^{-\sigma_{-\Operator}(\cdot)t}\big](x-x_0)=\int_{\R^N}\e^{-\sigma_{-\Operator}(\xi)t}\e^{2\pi\textup{i}(x-x_0)\cdot\xi}\dd \xi,
$$
and
\begin{equation*}
\begin{split}
\mathbb{G}_{-\Operator}^{x_0}(x)&=\int_0^\infty \mathbb{H}_{-\Operator}^{x_0}(x,t)\dd t=\int_{\R^N}\bigg(\int_0^\infty\e^{-\sigma_{-\Operator}(\xi)t}\dd t\bigg)\e^{2\pi\textup{i}(x-x_0)\cdot\xi}\dd \xi\\
&=\int_{\R^N}\frac{1}{\sigma_{-\Operator}(\xi)}\e^{2\pi\textup{i}(x-x_0)\cdot\xi}\dd \xi=\mathcal{F}^{-1}\Big[\frac{1}{\sigma_{-\Operator}(\xi)}\Big](x-x_0).
\end{split}
\end{equation*}
We also refer to the well-written book \cite{BoByKuRySoVo09}, which provides many examples of Green functions and a good introduction to potential theory.

In the examples that follows, we will need
\begin{equation}\label{GammaFunction}
\int_{0}^{\infty}\e^{-tr}r^\vartheta\dd r=\frac{\Gamma(\vartheta+1)}{t^{\vartheta+1}}<\infty \qquad\text{whenever $\vartheta>-1$},
\end{equation}
where $\Gamma$ is the gamma function.

\subsection{On the assumption \eqref{G_1}}

As demonstrated in Theorem \ref{thm:L1ToLinfinitySmoothing2}, assumption \eqref{G_1} leads to the estimate
\begin{equation*}
\|u(\cdot,t)\|_{L^\infty(\R^N)}\lesssim t^{-N\theta_{\alpha}}\|u_0\|_{L^1(\R^N)}^{\alpha\theta_{\alpha}}\qquad\text{for a.e. $t>0$}
\end{equation*}
for weak dual solutions of \eqref{GPME} with initial data $u_0$. Let us provide some concrete examples of operators $-\Operator$ in \eqref{GPME} whose Green functions satisfy \eqref{G_1}.

\begin{lemma}\label{lem:FractionaLaplaceG_2'}
The fractional Laplacian/Laplacian $(-\Delta)^{\frac{\alpha}{2}}$ with $\alpha\in(0,2]$ has a Green function which satisfies \eqref{G_1}.
\end{lemma}

\begin{remark}
Let us mention that heat kernel estimates for the Laplacian and the fractional Laplacian dates back to Fourier \cite[Chapter IX Section II]{Fou55} (see also \cite[Section 2.3]{Eva10}) and Blumenthal and Getoor \cite{BlGe60}, respectively.
\end{remark}

\begin{proof}[Proof of Lemma \ref{lem:FractionaLaplaceG_2'}]
Assume $\alpha\in(0,2)$. By Lemma 2 in Chapter V.1 in \cite{Ste70},
$$
\mathbb{G}_{(-\Delta)^{\frac{\alpha}{2}}}^{x_0}(x)=\mathscr{F}^{-1}\big[|\cdot|^{-\alpha}\big](x-x_0)\eqsim |x-x_0|^{-(N-\alpha)}.
$$
Now,
$$
\int_{B_R(x_0)}\mathbb{G}_{(-\Delta)^{\frac{\alpha}{2}}}^{x_0}(x)\dd x\eqsim \int_0^{R}r^{-(N-\alpha)}r^{N-1}\dd r\eqsim R^{\alpha}.
$$
Moreover, for any $x\in \R^N\setminus B_R(x_0)$,
$$
\mathbb{G}_{(-\Delta)^{\frac{\alpha}{2}}}^{x_0}(x)\lesssim R^{-(N-\alpha)},
$$
and the result follows.

Assume $\alpha=2$. The result is classical and can e.g. be found in \cite[Section 1.1.8]{Dav89}. We get that $\mathbb{G}_{-\Delta}^{x_0}(x)\eqsim |x-x_0|^{-(N-2)}$ which satisfies \eqref{G_1} with $\alpha=2$.
\end{proof}

\begin{corollary}
Any operator $\Operator$ whose Green function satisfies
$$
\mathbb{G}_{-\Operator}^{x_0}(x)\lesssim |x-x_0|^{-(N-\alpha)} \qquad\text{for some $\alpha\in(0,2]$}
$$
will fulfil \eqref{G_1}.
\end{corollary}

\begin{remark}
By Lemma \ref{lem:GreenFunctionGivesSobolev}, the above assumption on the Green function implies that the corresponding operator satisfies the Sobolev inequality. Again, we also refer to \cite{KaKiLe21} for a further discussion.
\end{remark}

\begin{lemma}\label{lem:GreenForUniformlyElliptic}
Assume that the real matrix $[a_{ij}]_{i,j=1,\ldots,N}$ is nonnegative and symmetric and $\Operator=\sum_{i,j=1}^{N}a_{ij}\dell_{x_ix_j}^2$. Given the following statements:
\begin{enumerate}[{\rm (i)}]
\item There exist constants $C,c>0$ such that
$$
c|y|^2\leq \sum_{i,j=1}^Na_{ij}y_iy_j\leq C|y|^2.
$$
\item There exist constants $C,c>0$ such that
$$
\mathbb{H}_{-\Levy^{\mu}}^{x_0}(x,t)\leq ct^{-\frac{N}{2}}\textup{exp}\Big(-C\frac{|x-x_0|^2}{t}\Big).
$$
\item There exists a constant $C>0$ such that
$$
\mathbb{G}_{-\Levy^{\mu}}^{x_0}(x)\leq C|x-x_0|^{-(N-2)}.
$$
\end{enumerate}
We have (i)$\Longrightarrow$(ii)$\Longrightarrow$(iii).
\end{lemma}

\begin{remark}
\begin{enumerate}[{\rm (a)}]
\item The heat kernel bound is (up to constants) the same as for the regular Laplacian. This is not surprising in the constant coefficient case since the operator is, up to a translation, the Laplacian. The more interesting case is of course when the coefficients are $(x,t)$-dependent, see also \cite[Section 2.9]{Jac02} and the classical \cite{Aro68}. For a similar result in the fractional setting, we refer to \cite{KaWe21}.
\item The upper bound in statement (i) might seem superfluous, but the constant inside the exponential function in (ii) depends on it.
\end{enumerate}
\end{remark}

\begin{proof}[Proof of Lemma \ref{lem:GreenForUniformlyElliptic}]
\noindent(i)$\Longrightarrow$(ii). Follows by \cite[Corollary 3.2.8]{Dav89} (see also \cite{Aro68}).

\smallskip
\noindent(ii)$\Longrightarrow$(iii). Follows by \cite[Theorem 3.1.1]{Dav89} (see also \cite{Aro68}).
\end{proof}

The nonlocal counterpart is somehow when the L\'evy measure is comparable to the measure of the fractional Laplacian.

\begin{lemma}\label{lem:GreenComparableWithFractionalLaplacian}
Assume $\Operator=\Levy^{\mu}$ and \eqref{muas}. Given the following statements, for $\alpha\in(0,2)$:
\begin{enumerate}[{\rm (i)}]
\item There exist constants $C,c>0$ such that
$$
\frac{c}{|z|^{N+\alpha}}\leq \frac{\dd \mu}{\dd z}(z)\leq \frac{C}{|z|^{N+\alpha}}.
$$
\item There exists a constant $C>0$ such that
$$
\mathbb{H}_{-\Levy^{\mu}}^{x_0}(x,t)\leq C\min\big\{t^{-\frac{N}{\alpha}},t|x-x_0|^{-(N+\alpha)}\big\}.
$$
\item There exists a constant $C>0$ such that
$$
\mathbb{G}_{-\Levy^{\mu}}^{x_0}(x)\leq C|x-x_0|^{-(N-\alpha)}.
$$
\end{enumerate}
We have (i)$\Longrightarrow$(ii)$\Longrightarrow$(iii).
\end{lemma}

\begin{remark}
We can also slightly weaken the assumption on the lower bound: There exist constants $C,c>0$ such that
$$
c\veps^{-\alpha}\leq \int_{|z|>\veps}\dd\mu(z)\quad\text{$\forall\,\veps>0$} \qquad\text{and}\qquad  \frac{\dd \mu}{\dd z}(z)\leq \frac{C}{|z|^{N+\alpha}}.
$$
The estimates on the heat kernel and Green function still hold \cite[Theorem 2]{Szt10b} with $f(x,y)=f(|y-x|)=\dd \mu/\dd z$, see also the recent \cite{KaKiLe21}.
\end{remark}

\begin{proof}[Proof of Lemma \ref{lem:GreenComparableWithFractionalLaplacian}]
\noindent(i)$\Longrightarrow$(ii). Follows by \cite[Theorem 1.2]{ChKu08} with $\rho(x,y)=|x-y|$, $V(r)=r^N$, $\gamma_1=\gamma_2=0$, $\psi(r)=1$, and $\phi_1(r)=r^\alpha$.

\smallskip
\noindent(ii)$\Longrightarrow$(iii). By direct computations,
\begin{equation*}
\begin{split}
\mathbb{G}_{-\Levy^{\mu}}^{x_0}(x)&=\int_0^\infty \mathbb{H}_{-\Levy^{\mu}}^{x_0}(x,t)\dd t\\
&\lesssim\int_0^{|x-x_0|^\alpha} t|x-x_0|^{-(N+\alpha)}\dd t+\int_{|x-x_0|^\alpha}^{\infty}t^{-\frac{N}{\alpha}}\dd t\\
&\lesssim|x-x_0|^{-(N-\alpha)}. \qedhere
\end{split}
\end{equation*}
\end{proof}

\subsection{Combinations of assumption \eqref{G_1}}\label{sec:CombinationsOfAssumptionG_1}
Sometimes the Green function has different power behaviours at zero and at infinity. As demonstrated in Theorem \ref{thm:L1ToLinfinitySmoothing3}, such a case leads to the estimate
$$
\|u(\cdot,t)\|_{L^\infty(\R^N)}\lesssim t^{-N\theta_{\alpha}}\|u_0\|_{L^1(\R^N)}^{\alpha\theta_{\alpha}}+
t^{-N\theta_{2}}\|u_0\|_{L^1(\R^N)}^{2\theta_{2}}\qquad\text{for a.e. $t>0$}
$$
for weak dual solutions of \eqref{GPME} with initial data $u_0$. Let us provide some concrete examples of operators $-\Operator$ in \eqref{GPME} whose Green functions satisfy combinations of \eqref{G_1}.

We start with one of the most basic operators giving such an estimate:

\begin{lemma}\label{lem:SumOfLapAndFracLap}
Assume $\alpha\in(0,2)$ and $-\Operator=(-\Delta)+(-\Delta)^{\frac{\alpha}{2}}=:(-\Delta)+(-\Levy^\mu)$. Given the following statements:
\begin{enumerate}[{\rm (i)}]
\item For some constant $C>0$, we consider
$$
\frac{\dd \mu}{\dd z}(z)=\frac{C}{|z|^{N+\alpha}}.
$$
\item There exists a constant $C>0$ such that
$$
\mathbb{H}_{-\Operator}^{x_0}(x,t)\leq C
\begin{cases}
f(x-x_0,t)+\mathbb{H}_{(-\Delta)^{\frac{\alpha}{2}}}^{x_0}(x,t)\qquad\qquad&\text{if $0<t<|x-x_0|^2\leq 1$,}\\
\mathbb{H}_{-\Delta}^{x_0}(x,t)\qquad\qquad&\text{if $|x-x_0|^2<t<|x-x_0|^\alpha\leq 1,$}\\
\mathbb{H}_{-\Delta}^{x_0}(x,t)\qquad\qquad&\text{if $|x-x_0|^\alpha\leq t\leq 1$,}\\
\mathbb{H}_{(-\Delta)^{\frac{\alpha}{2}}}^{x_0}(x,t)\qquad\qquad&\text{if $t\geq1$ or $|x-x_0|\geq 1$,}
\end{cases}
$$
where $f(x-x_0,t):=(4\pi t)^{-N/2}\textup{exp}(-|x-x_0|^2/(16t))$.
\item There exists a constant $C>0$ such that
$$
\mathbb{G}_{-\Operator}^{x_0}(x)\leq C
\begin{cases}
|x-x_0|^{-(N-2)}\qquad\qquad&\text{if $|x|\leq 1$,}\\
|x-x_0|^{-(N-\alpha)}\qquad\qquad&\text{if $|x|>1$.}\\
\end{cases}
$$
\end{enumerate}
We have (i)$\Longrightarrow$(ii)$\Longrightarrow$(iii).
\end{lemma}

\begin{remark}\label{rem:BigSmallR}
Note that when $0<R\leq 1$, we get
$$
\int_{B_R(x_0)}\mathbb{G}_{-\Operator}^{x_0}(x)\dd x\lesssim R^{2},\qquad \mathbb{G}_{-\Operator}^{x_0}(x)\lesssim R^{-(N-2)}\quad\text{for $x\in \R^N\setminus B_R(x_0)$,}
$$
and when $R>1$,
$$
\int_{B_R(x_0)}\mathbb{G}_{-\Operator}^{x_0}(x)\dd x\lesssim R^{\alpha},\qquad \mathbb{G}_{-\Operator}^{x_0}(x)\lesssim R^{-(N-\alpha)}\quad\text{for $x\in \R^N\setminus B_R(x_0)$.}
$$
Hence, we are in the setting of Theorem \ref{thm:L1ToLinfinitySmoothing3} although in this example the small time behaviour is governed by the Laplacian and the large time by the fractional Laplacian.
\end{remark}

\begin{proof}[Proof of Lemma \ref{lem:SumOfLapAndFracLap}]
\noindent(i)$\Longrightarrow$(ii). Follows by \cite[Theorem 2.13]{SoVo07}.

\smallskip
\noindent(ii)$\Longrightarrow$(iii). Assume $|x|\leq 1$. Then
\begin{equation*}
\begin{split}
\mathbb{G}_{-\Operator}^{0}(x)&=\int_0^\infty\mathbb{H}_{-\Operator}^{0}(x,t)\dd t=\bigg(\int_0^{|x|^2}+\int_{|x|^2}^{|x|^\alpha}+\int_{|x|^\alpha}^1+\int_1^\infty\bigg)\mathbb{H}_{-\Operator}^{0}(x,t)\dd t.
\end{split}
\end{equation*}
Let us start with the integral involving $f(x,t)$. The change of variables $|x|^2/(16t)\mapsto t$ gives
\begin{equation*}
\begin{split}
\int_0^{|x|^2}f(x,t)\dd t\eqsim\int_0^{|x|^2}t^{-N/2}\textup{e}^{\frac{-|x|^2}{16t}}\dd t=\Big(\frac{|x|^2}{16}\Big)^{-\frac{N}{2}+1}\int_{\frac{1}{16}}^{\infty}\tau^{\frac{N}{2}-2}\textup{e}^{-\tau}\dd \tau\eqsim |x|^{-(N-2)},
\end{split}
\end{equation*}
where we estimated the final integral by \eqref{GammaFunction}. The integrals involving $\mathbb{H}_{-\Delta}^{0}$ can be estimated in a similar way. It remains to estimate the contribution from $\mathbb{H}_{(-\Delta)^{\frac{\alpha}{2}}}^{0}$:
\begin{equation*}
\begin{split}
&\bigg(\int_0^{|x|^2}+\int_1^\infty\bigg)\mathbb{H}_{(-\Delta)^{\frac{\alpha}{2}}}^{0}(x,t)\dd t\lesssim \bigg(\int_0^{|x|^2}+\int_1^\infty\bigg)\min\{t^{-N/\alpha},t|x|^{-N-\alpha}\}\dd t\\
&=\int_0^{|x|^2}t|x|^{-N-\alpha}\dd t+\int_1^\infty t^{-N/\alpha}\dd t\eqsim|x|^{-(N-2)+2-\alpha}+1.
\end{split}
\end{equation*}
Since $|x|\leq 1$, we have $|x|^{-(N-2)+2-\alpha}\leq |x|^{-(N-2)}$ and $1\leq |x|^{-(N-2)}$.

Assume $|x|>1$. Then
\begin{equation*}
\begin{split}
\mathbb{G}_{-\Operator}^{0}(x)&=\int_0^\infty\mathbb{H}_{-\Operator}^{0}(x,t)\dd t\lesssim\int_0^\infty\mathbb{H}_{(-\Delta)^{\frac{\alpha}{2}}}^{0}(x,t)\dd t\lesssim|x|^{-(N-\alpha)}.
\end{split}
\end{equation*}

We combine the results to complete the proof.
\end{proof}

Let us now consider the relativistic Schr\"odinger type operators like
\begin{equation}\label{eq:RelativisticSchrodinger}
(\kappa^2I-\Delta)^\frac{\alpha}{2}-\kappa^{\alpha} I\qquad\text{with $\kappa>0$ and $\alpha\in(0,2)$}
\end{equation}
are L\'evy operators \cite[Lemma 2]{Ryz02} (see also \cite{GrRy08} and \cite[Appendix B]{FaFe15}), i.e., they can be written on the form $\Levy^\mu$, see \eqref{def:LevyOperators}, with a measure satisfying \eqref{muas}.

\begin{lemma}\label{lem:GreenRelativisticSchrodinger}
Assume $-\Operator$ is given by \eqref{eq:RelativisticSchrodinger}. Given the following statements, for $\alpha\in(0,2)$:
\begin{enumerate}[{\rm (i)}]
\item For some constant $C>0$, we consider
$$
\frac{\dd \mu}{\dd z}(z)=\frac{C}{|z|^{\frac{N+\alpha}{2}}}K_{\frac{N+\alpha}{2}}(\kappa|z|),
$$
where $K_a$ is the modified Bessel function of the second kind with index $a\in\R$.
\item There exists a constant $C>0$ such that, for all $\gamma>2-\alpha$,
$$
\mathbb{H}_{-\Levy^{\mu}}^{x_0}(x,t)\leq C\min\bigg\{t^{-\frac{N}{\alpha}},\frac{t}{|x-x_0|^{N+\alpha}(1+|x-x_0|)^{\gamma}}\bigg\}\qquad\text{if $0<t<1$,}
$$
and
$$
\mathbb{H}_{-\Levy^{\mu}}^{x_0}(x,t)\leq C\min\bigg\{t^{-\frac{N}{2}},\frac{t}{|x-x_0|^{N+\alpha}(1+|x-x_0|)^{\gamma}}\bigg\}\qquad\text{if $t>1$.}
$$
\item There exists a constant $C>0$ such that
$$
\mathbb{G}_{-\Levy^{\mu}}^{x_0}(x)\leq C\big(|x-x_0|^{-(N-\alpha)}+|x-x_0|^{-(N-2)}\big).
$$
\end{enumerate}
We have (i)$\Longrightarrow$(ii)$\Longrightarrow$(iii).
\end{lemma}

\begin{remark}
It is interesting to note that
$$
\frac{\dd \mu}{\dd z}(z)\eqsim\frac{1}{|z|^{N+\alpha}}\quad\text{as $|z|\to0$}\qquad\text{and}\qquad \frac{\dd \mu}{\dd z}(z)\eqsim\frac{1}{|z|^{\frac{N+\alpha+1}{2}}}\e^{-\kappa|z|}\quad\text{as $|z|\to\infty$.}
$$
\end{remark}

\begin{proof}[Proof of Lemma \ref{lem:GreenRelativisticSchrodinger}]
\noindent(i)$\Longrightarrow$(ii). Follows by \cite[Section 5]{Szt10a}.

\smallskip
\noindent(ii)$\Longrightarrow$(iii). Since
$$
\mathbb{G}_{-\Levy^{\mu}}^{0}(x)=\int_0^\infty \mathbb{H}_{-\Levy^{\mu}}^{0}(x,t)\dd t,
$$
we will have to consider three cases
$$
{\rm (I)}\qquad \frac{t}{|x|^{N+\alpha}(1+|x|)^\gamma}>t\qquad\Longleftrightarrow\qquad |x|^{N+\alpha}(1+|x|)^\gamma<1,
$$
$$
{\rm (II)}\qquad \frac{t}{|x|^{N+\alpha}(1+|x|)^\gamma}=t\qquad\Longleftrightarrow\qquad |x|^{N+\alpha}(1+|x|)^\gamma=1,
$$
and
$$
{\rm (III)}\qquad \frac{t}{|x|^{N+\alpha}(1+|x|)^\gamma}<t\qquad\Longleftrightarrow\qquad |x|^{N+\alpha}(1+|x|)^\gamma>1.
$$

In the case of (I), we have three different behaviours:
\begin{equation*}
\begin{split}
\mathbb{G}_{-\Levy^{\mu}}^{0}(x)&\eqsim\int_0^{|x|^\alpha(1+|x|)^{\gamma\frac{\alpha}{N+\alpha}}}\frac{t}{|x|^{N+\alpha}(1+|x|)^\gamma}\dd t+\int_{|x|^\alpha(1+|x|)^{\gamma\frac{\alpha}{N+\alpha}}}^1t^{-\frac{N}{\alpha}}\dd t+\int_1^\infty t^{-\frac{N}{2}}\dd t\\
&=\frac{1}{2}|x|^{-(N-\alpha)}(1+|x|)^{-\gamma\frac{N-\alpha}{N+\alpha}}+\frac{\alpha}{N-\alpha}|x|^{-(N-\alpha)}(1+|x|)^{-\gamma\frac{N-\alpha}{N+\alpha}}-\frac{\alpha}{N-\alpha}+\frac{2}{N-2}\\
&= \frac{1}{2}\frac{N-\alpha}{N+\alpha}|x|^{-(N-\alpha)}(1+|x|)^{-\gamma\frac{N-\alpha}{N+\alpha}}+\frac{N(2-\alpha)}{(N-2)(N-\alpha)}\\
&\leq \Big(\frac{1}{2}\frac{N-\alpha}{N+\alpha}+\frac{N(2-\alpha)}{(N-2)(N-\alpha)}\Big)|x|^{-(N-\alpha)}\qquad\text{in $\{x \,:\, |x|^{N+\alpha}(1+|x|)^\gamma<1\}$,}
\end{split}
\end{equation*}
were we used $1+|x|\geq 1$ to get
$$
\frac{1}{(1+|x|)^{\gamma\frac{N-\alpha}{N+\alpha}}}\leq 1.
$$

In the case of (II), we have two different behaviours:
\begin{equation*}
\begin{split}
\mathbb{G}_{-\Levy^{\mu}}^{0}(x)&\eqsim\int_0^1t\dd t +\int_1^\infty t^{-\frac{N}{2}}\dd t=\frac{1}{2}+\frac{2}{N-2}=\frac{1}{2}\frac{N-2}{N+2}.
\end{split}
\end{equation*}

In the case of (III), we have two different behaviours:
\begin{equation*}
\begin{split}
\mathbb{G}_{-\Levy^{\mu}}^{0}(x)&\eqsim\int_0^{|x|^{(N+\alpha)\frac{2}{N+2}}(1+|x|)^{\gamma\frac{2}{N+2}}}\frac{t}{|x|^{N+\alpha}(1+|x|)^\gamma}\dd t+\int_{|x|^{(N+\alpha)\frac{2}{N+2}}(1+|x|)^{\gamma\frac{2}{N+2}}}^\infty t^{-\frac{N}{2}}\dd t\\
&=\frac{1}{2}|x|^{-(N-2)}\Big(\frac{|x|^{2-\alpha}}{(1+|x|)^\gamma}\Big)^{\frac{N-2}{N+2}}+\frac{2}{N-2}|x|^{-(N-2)}\Big(\frac{|x|^{2-\alpha}}{(1+|x|)^\gamma}\Big)^{\frac{N-2}{N+2}}\\
&=\frac{1}{2}\frac{N-2}{N+2}|x|^{-(N-2)}\Big(\frac{|x|^{2-\alpha}}{(1+|x|)^\gamma}\Big)^{\frac{N-2}{N+2}}.
\end{split}
\end{equation*}
By the assumption $\gamma>2-\alpha$, we get
\begin{equation*}
\begin{split}
\mathbb{G}_{-\Levy^{\mu}}^{0}(x)\lesssim |x|^{-(N-2)} \qquad\text{in $\{x \,:\, |x|^{N+\alpha}(1+|x|)^\gamma>1\}$.}
\end{split}
\end{equation*}
We then collect the three cases in one estimate to complete the proof.
\end{proof}

Finally, we also consider the generator of a finite range isotropically symmetric $\alpha$-stable process in $\R^N$ with jumps of size larger than $1$ removed.

\begin{lemma}\label{lem:GreenFiniteRange}
Assume $\Operator=\Levy^\mu$ with a measure $\mu$ satisfying \eqref{muas}. Given the following statements, for $\alpha\in(0,2)$:
\begin{enumerate}[{\rm (i)}]
\item There exists a constant $C>0$ such that
$$
\frac{\dd \mu}{\dd z}(z)=\frac{C}{|z|^{N+\alpha}}\mathbf{1}_{|z|\leq 1}.
$$
\item There exist constants $C,c>0$ and $0<C_*, R_*<1$ such that,
$$
\mathbb{H}_{-\Levy^{\mu}}^{x_0}(x,t)\leq C \min\big\{t^{-\frac{N}{\alpha}},t|x-x_0|^{-(N+\alpha)}\big\}, \qquad\text{$0<t<R_*^\alpha$, $|x-x_0|\leq R_*$,}
$$
$$
\mathbb{H}_{-\Levy^{\mu}}^{x_0}(x,t)\leq C \exp\Big(-c|x-x_0|\log\Big(\frac{|x-x_0|}{t}\Big)\Big), \qquad\text{$|x-x_0|\geq\max\{t/C_*,R_*\}$,}
$$
and
$$
\mathbb{H}_{-\Levy^{\mu}}^{x_0}(x,t)\leq Ct^{-\frac{N}{2}}\exp\Big(-c\frac{|x-x_0|^2}{t}\Big), \qquad\text{$t>R_*^\alpha$, $|x-x_0|\leq t/C_*$.}
$$
\item There exists a constant $C>0$ such that
$$
\mathbb{G}_{-\Levy^{\mu}}^{x_0}(x)\leq C\big(|x-x_0|^{-(N-\alpha)}+|x-x_0|^{-(N-2)}\big).
$$
\end{enumerate}
We have (i)$\Longrightarrow$(ii)$\Longrightarrow$(iii).
\end{lemma}

\begin{remark}
We are again in the setting of Theorem \ref{thm:L1ToLinfinitySmoothing3} by Remark \ref{rem:BigSmallR}.
\end{remark}

\begin{proof}[Proof of Lemma \ref{lem:GreenFiniteRange}]
\noindent(i)$\Longrightarrow$(ii). Follows by \cite[Proposition 2.1 and Theorem 2.3]{ChKiKu08}.

\smallskip
\noindent(ii)$\Longrightarrow$(iii). Follows by the proof of \cite[Theorem 4.7]{ChKiKu08}.
\end{proof}

\subsection{On the assumption \eqref{G_2}}
If the Green function decays fast enough at infinity, the function itself will not only be $L_\textup{loc}^1$ but indeed $L^1$, see \eqref{G_2}. As demonstrated in Theorem \ref{thm:AbsBounds}, such a case leads to the estimate
$$
\|u(\cdot,t)\|_{L^\infty(\R^N)}\lesssim t^{-1/(m-1)}\qquad\text{for a.e. $t>0$},
$$
for weak dual solutions of \eqref{GPME} with initial data $u_0$. Let us provide some concrete examples of operators $-\Operator$ in \eqref{GPME} whose Green functions satisfy \eqref{G_2}. Actually, any operator of the form $I-\Operator$ has a Green function  satisfying \eqref{G_2}. In what follows, we will explain this result, and illustrate it with other examples as well.

\begin{lemma}\label{lem:GreenResolventIntegrable}
Under the assumptions of Proposition \ref{prop:GreenFormulas},
$$
\|\mathbb{G}_{I-\Operator}^{x_0}\|_{L^1(\R^N)}=\|\mathbb{G}_{I-\Operator}^{0}\|_{L^1(\R^N)}\leq 1.
$$
Hence, the operator $I-\Operator$ has a Green function which satisfies \eqref{G_2}.
\end{lemma}

\begin{proof}%[Proof of Lemma \ref{lem:GreenResolventIntegrable}]
Assumption \eqref{Gas}, an application of the Tonelli lemma, and the fact that, by Remark \ref{rem:APrioriCasem1} (i.e., decay of $L^1$-norm), $\int \mathbb{H}_{-\Operator}^{x_0}(\cdot,t)=\int \mathbb{H}_{-\Operator}^{0}(\cdot,t)\leq 1$ for every fixed $t>0$ concludes the proof.
\end{proof}

\begin{remark}\label{rem:G4NotWhenConservationOfMass}
We immediately see that the presence of the identity operator is crucial. In fact, if $-\Operator$ is such that the corresponding heat equation preserves mass, then $\|\mathbb{G}_{-\Operator}^{x_0}\|_{L^1(\R^N)}=\infty$ (cf. Proposition \ref{prop:GreenFormulas}). Examples of mass preserving operators are L\'evy operators \eqref{def:LevyOperators} with $c=0$.
\end{remark}

Let us begin by considering the extreme case $-\Operator=I$ for which the PDE in \eqref{GPME} reads
\begin{equation*}%\label{eq:PureAbsorbtion}
\dell_tu=-u^m.
\end{equation*}
For any function $t\mapsto Y(t)$, that equation is an ODE of the form
$$
Y'(t)= -Y(t)^{1+(m-1)} \quad\Longrightarrow\quad
Y(t)\leq \Big(\frac{1}{(m-1)t}\Big)^{\frac{1}{m-1}}.
$$
Hence, by the comparison principle for \eqref{GPME} with $-\Operator=I$ (where we take $Y(0)=\infty$), we get the absolute bound
$$
\|u(\cdot,t)\|_{L^\infty(\R^N)}\leq Y(t)\leq \Big(\frac{1}{(m-1)t}\Big)^{\frac{1}{m-1}}.
$$
See also Section III.C in \cite{Ver79}. The above is also contained in the following lemma:

\begin{lemma}\label{eq:G4Identity}
The identity operator $-\Operator=I$ has a Green function which satisfies \eqref{G_2}, i.e.,
$$
\|\mathbb{G}_{I}^{x_0}\|_{L^1(\R^N)}=\|\mathbb{G}_{I}^{0}\|_{L^1(\R^N)}\leq C_1<\infty.
$$
\end{lemma}

\begin{proof}
We obtain
$$
\frac{\dd }{\dd t}\int_{\R^N}\mathbb{H}_I^{x_0}(x,t)\dd x=-\int_{\R^N}\mathbb{H}_I^{x_0}(x,t)\dd x,
$$
i.e., $\int \mathbb{H}_I^{x_0}(\cdot,t)=\e^{-t}$ for all $t>0$. By the definition of the Green function (cf. Proposition \ref{prop:GreenFormulas}), we conclude.
\end{proof}

\begin{remark}
\begin{enumerate}[{\rm (a)}]
\item The proof also demonstrates that the operator $-\Operator=I$ is not conserving mass. Indeed, in the corresponding heat equation, it decays with time.
\item Moreover, it provides a trivial example of an operator which yields boundedness in the nonlinear case $(m>1)$, but not in the linear case ($m=1$).
\end{enumerate}
\end{remark}

Of course, we can reapply the same strategy of comparison with $Y(t)$ for \eqref{GPME} with $-\Operator\mapsto I-\Operator$ to get
$$
\|u(\cdot,t)\|_{L^\infty(\R^N)}\leq Y(t)\leq \Big(\frac{1}{(m-1)t}\Big)^{\frac{1}{m-1}}
$$
independently of $\Operator$!

\begin{remark}\label{rem:SmoothingLinearAbsorption}
When $m=1$, we need to adapt another strategy (since $Y(t)=Y(0)\e^{-t})$, but recall that by defining
$$
u(x,t):=\textup{e}^{-t}v(x,t)
$$
where $v$ solves \eqref{GPME} with $m=1$, then $u$ solves
\begin{equation*}
\begin{cases}
\dell_tu-\Operator[u]+u=0 \qquad\qquad&\text{in}\qquad \R^N\times(0,T],\\
u(\cdot,0)=u_0 \qquad\qquad&\text{on}\qquad \R^N.
\end{cases}
\end{equation*}
Hence, in this case, $L^1$--$L^\infty$-smoothing follows as long as it holds for $v$, i.e., as long as $\Operator$ is strong enough to provide it.
\end{remark}

When $\kappa=1$ in \eqref{eq:RelativisticSchrodinger}, we get the L\'evy operator
$$
-\Levy^{\mu_\text{RS}}:=(I-\Delta)^\frac{\alpha}{2}-I,
$$
i.e., $(I-\Delta)^{\frac{\alpha}{2}}$ is of the form $(I-\Levy^{\mu_\text{RS}})$. Hence:

\begin{lemma}\label{lem:OperatorWithBesselPotential}
The operator $-\Operator=(I-\Delta)^{\frac{\alpha}{2}}$ with $\alpha\in(0,2)$ has a Green function which satisfies \eqref{G_2}, i.e.,
$$
\|\mathbb{G}_{(I-\Delta)^{\frac{\alpha}{2}}}^{x_0}\|_{L^1(\R^N)}=\|\mathbb{G}_{(I-\Delta)^{\frac{\alpha}{2}}}^{0}\|_{L^1(\R^N)}\leq C_1<\infty.
$$
\end{lemma}

\begin{remark}\label{rem:OperatorWithBesselPotential}
The operator $-\Operator=(I-\Delta)^{\frac{\alpha}{2}}$ with $\alpha=1$ appears e.g. in \cite{BaDaQu11} for the linear equation \eqref{GPME} with $m=1$. Since the mentioned operator is of the form $(I-\Levy^{\mu_\text{RS}})$, $L^1$--$L^\infty$-smoothing holds whenever it holds for $\Levy^{\mu_\text{RS}}$ (see Remark \ref{rem:SmoothingLinearAbsorption}). The heat kernel bounds of Lemma \ref{lem:GreenRelativisticSchrodinger} then provides the result through Theorem \ref{thm:LinearEquivalences}. In the nonlinear case ($m>1$), however, the above lemma ensures that \eqref{G_2} holds and we deduce absolute bounds.
\end{remark}

\begin{proof}[Proof of Lemma \ref{lem:OperatorWithBesselPotential}]
Note that
$$
\mathbb{G}_{(I-\Delta)^{\frac{\alpha}{2}}}^{x_0}(x)=\mathscr{F}^{-1}\big[(1+|\cdot|^2)^{-\frac{\alpha}{2}}\big](x-x_0),
$$
i.e., the Bessel potential. The result then follows by Proposition 2 in Chapter V.3 in \cite{Ste70}. Or, we can simply note that $\mathbb{G}_{(I-\Delta)^{\alpha/2}}^{x_0}=\mathbb{G}_{I-\Levy^{\mu_\text{RS}}}^{x_0}$ and $-\Levy^{\mu_\text{RS}}$ is such that the heat equation has $L^1$-decay, cf. Lemma \ref{lem:GreenResolventIntegrable}.
\end{proof}

Again, the operator related to the Bessel potential does not conserve mass since $(I-\Delta)^\frac{\alpha}{2}-I$ does so. Moreover, in the latter case, assumption \eqref{G_2} cannot hold (cf. Remark \ref{rem:G4NotWhenConservationOfMass}) and we can write the PDE in \eqref{GPME} as
$$
\dell_t u+(I-\Delta)^\frac{\alpha}{2}[u^m]=u^m.
$$
This is in contrast with the operators $I-\Operator$ in Lemma \ref{lem:GreenResolventIntegrable} which satisfies the PDE
$$
\dell_tu-\Operator[u^m]=-u^m.
$$
Taking $-\Operator=(I-\Delta)^\frac{\alpha}{2}$ in the latter, we see that assumption \eqref{G_2} relies on either the (strong) absorption term being present or the operator itself being \emph{positive}, or also both being present of course.

\subsection{On the assumption \eqref{G_3}}
By Corollary \ref{cor:GreenNonnegative}, $\mathbb{G}_{I-\Operator}^{x_0}$ has at least as good integrability properties as $\mathbb{G}_{-\Operator}^{x_0}$. It is, moreover, always defined for descent operators $-\Operator$, see the discussion in Section \ref{sec:InverseOfLinearmAccretiveDirichlet}. There should therefore be no surprise that assumption \eqref{G_3} is quite general, however, as shown in Theorem \ref{thm:L1ToLinfinitySmoothing}, it provides a rather poor smoothing estimate:
\begin{equation}\label{eq:PoorSmoothing}
\|u(\cdot,t)\|_{L^\infty(\R^N)}\lesssim t^{-\frac{1}{m-1}} +\|u_0\|_{L^1(\R^N)}\qquad\text{for a.e. $t>0$,}
\end{equation}
for weak dual solutions of \eqref{GPME} with initial data $u_0$. Let us provide some concrete examples of operators $-\Operator$ in \eqref{GPME} whose Green functions satisfy \eqref{G_3}, and let us also see how to improve the above estimate. To continue, we advice the reader to recall \eqref{eq:LpNormOfGreenOfResolvent}.

\begin{lemma}\label{lem:GreenOfResolventOfFractionalLaplacian}
The fractional Laplacian/Laplacian $(-\Delta)^{\frac{\alpha}{2}}$ with $\alpha\in(0,2]$ has a Green function which satisfies \eqref{G_3}, i.e.,
$$
\|\mathbb{G}_{I+(-\Delta)^{\frac{\alpha}{2}}}^{x_0}\|_{L^p(\R^N)}=\|\mathbb{G}_{I+(-\Delta)^{\frac{\alpha}{2}}}^{0}\|_{L^p(\R^N)}\leq C_{p}<\infty
$$
for some $p\in(1,N/(N-\alpha))$.
\end{lemma}

\begin{proof}
We use \eqref{GammaFunction} with $\vartheta:=\frac{N}{\alpha}-1$ to obtain
\begin{equation}\label{eq:BoundOnHeatKernelFractionalLaplacian}
\begin{split}
&\|\mathbb{H}_{(-\Delta)^{\frac{\alpha}{2}}}^{0}(\cdot,t)\|_{L^\infty(\R^N)}=\|\mathscr{F}^{-1}[\e^{-t|\xi|^\alpha}]\|_{L^\infty(\R^N)}\leq \int_{\R^N}\e^{-t|\xi|^\alpha}\dd \xi\\
&\eqsim\int_0^{\infty}\e^{-tr^\alpha}r^{N-1}\dd r=\frac{1}{\alpha}\int_0^{\infty}\e^{-tr}r^{\vartheta}\dd r=\frac{1}{\alpha}t^{-\frac{N}{\alpha}}\Gamma\Big(\frac{N}{\alpha}\Big),
\end{split}
\end{equation}
and hence, by Remark \ref{rem:APrioriCasem1} (i.e., $L^1$-decay for $m=1$),
\begin{equation*}
\begin{split}
\|\mathbb{G}_{I+(-\Delta)^{\frac{\alpha}{2}}}^{0}\|_{L^p(\R^N)}&\leq \int_{0}^\infty\textup{e}^{-t}\|\mathbb{H}_{(-\Delta)^{\frac{\alpha}{2}}}^{0}(\cdot,t)\|_{L^p(\R^N)}\dd t\\
&\leq \int_{0}^\infty\textup{e}^{-t}\|\mathbb{H}_{(-\Delta)^{\frac{\alpha}{2}}}^{0}(\cdot,t)\|_{L^\infty(\R^N)}^{\frac{p-1}{p}}\|\mathbb{H}_{(-\Delta)^{\frac{\alpha}{2}}}^{0}(\cdot,t)\|_{L^1(\R^N)}^{\frac{1}{p}}\dd t\\
&\lesssim \int_0^\infty \e^{-t}t^{-\frac{N}{\alpha}\frac{p-1}{p}}\dd t=\tau^{-\frac{N}{\alpha}\frac{p-1}{p}+1}\int_0^\infty \e^{-\tau r}r^{-\frac{N}{\alpha}\frac{p-1}{p}}\dd r,
\end{split}
\end{equation*}
which is finite if $p<N/(N-\alpha)$ due to \eqref{GammaFunction} again.
\end{proof}

The fractional Laplacian, however, satisfies our strongest assumption \eqref{G_1} as well. We will therefore consider an operator which satisfies \eqref{G_3}, but for which it is not possible to verify \eqref{G_1} or \eqref{G_1'}. To that end, consider the sum of onedimensional fractional Laplacians:
\begin{equation}\label{eq:SumFracLap}
-\Operator=\sum_{i=1}^N(-\dell_{x_ix_i}^2)^{\frac{\alpha_i}{2}}\qquad\text{with $\alpha_i\in(0,2)$.}
\end{equation}
It can be written on the form $\Levy^\mu$ with $\mu$, for some constant $C>0$, given by
$$
\dd \mu(z)=C\sum_{i=1}^N\frac{1}{|z_i|^{1+\alpha_i}}\dd z_i\prod_{j\neq i}\dd \delta_0(z_j).
$$
This measure satisfies \eqref{muas} since each onedimensional fractional Laplacian measure does, and we have:

\begin{lemma}\label{lem:GreenAnisotropicLaplacians}
Assume $\Operator$ is given by \eqref{eq:SumFracLap} and
$$
\sum_{i=1}^N\frac{1}{\alpha_i}>1\qquad\text{where $\alpha_i\in(0,2)$}.
$$
Then
$$
\|\mathbb{G}_{I-\Levy^\mu}^{x_0}\|_{L^p(\R^N)}=\|\mathbb{G}_{I-\Levy^\mu}^{0}\|_{L^p(\R^N)}\leq C_{p}<\infty
$$
for some
$$
p\in\Big(1,\frac{\sum_{i=1}^N\frac{1}{\alpha_i}}{\sum_{i=1}^N\frac{1}{\alpha_i}-1}\Big).
$$
\end{lemma}

\begin{remark}\label{ref:GreenAnisotropicLaplacians}
\begin{enumerate}[{\rm (a)}]
\item Note that if $\alpha_i=\alpha$ for all $i\in\{1,\ldots,N\}$, then
$$
\sum_{i=1}^N\frac{1}{\alpha_i}>1\qquad\Longrightarrow\qquad \frac{N}{\alpha}>1.
$$
We thus recover the condition of Lemma \ref{lem:GreenOfResolventOfFractionalLaplacian}.
\item Various extensions within the framework of anisotropic fractional Laplacians can be found in \cite{BoSz07, Szt11, Xu13, BoSzKn20}.
\item By \cite[Section 3]{BaCh06},
\begin{equation}\label{eq:ProductOf1DFracLap}
\mathbb{H}_{-\Levy^{\mu}}^{x_0}(x,t)=\prod_{i=1}^N\mathbb{H}_{(-\dell_{x_ix_i}^2)^{\frac{\alpha_i}{2}}}^{x_0}(x,t)\leq C\prod_{i=1}^N\rho_i^{x_{0,i}}(x_i,t),
\end{equation}
where
$$
\rho_i^{x_{0,i}}(x_i,t)=\min\big\{t^{-\frac{1}{\alpha_i}},t|x_i-x_{0,i}|^{-(1+\alpha_i)}\big\}.\footnotemark
$$
\footnotetext{Optimal bounds when $\alpha_i=\alpha$ can be found in \cite{KaKiKu19}.}
However, the example stated at the end of \cite{BoSz05} shows that for $\alpha_i=\alpha$ with $\alpha\leq (N-1)/2<N$, small times (hence all times), and the choice $x=\xi e_1$ with $\xi>0$, yields $\mathbb{G}_{-\Levy^\mu}^{0}(x,t)=\infty$. We thus conclude that at least in this case, it is not possible to verify the second parts of \eqref{G_1} or \eqref{G_1'}.
\item Solutions of \eqref{GPME} with $-\Operator$ defined by \eqref{eq:SumFracLap} satisfy \eqref{eq:PoorSmoothing}. Once this estimate is established, we can, moreover, use the scaling of the operator to get it on an invariant form. We restrict to the case $\alpha_i=\alpha$ for all $i\in\{1,\ldots,N\}$. As in Remark \ref{rem:L1ToLinfinitySmoothing2},
if $u$ solves \eqref{GPME}, then
$$
u_{\kappa,\Xi,\Lambda}(x,t):=\kappa u(\Xi x,\Lambda t) \qquad\text{for all $\kappa,\Xi,\Lambda>0$}
$$
also solves \eqref{GPME} as long as $\kappa^{m-1}\Xi^\alpha=\Lambda$. This means that
$$
\|u_{\kappa,\Xi,\Lambda}(\cdot,t)\|_{L^\infty(\R^N)}\lesssim t^{-\frac{1}{m-1}}+\|u_{\kappa,\Xi,\Lambda}(\cdot,0)\|_{L^1(\R^N)}
$$
or
$$
\|u(\Xi\cdot,\Lambda t)\|_{L^\infty(\R^N)}\lesssim \Xi^{\frac{\alpha}{m-1}}(\Lambda t)^{-\frac{1}{m-1}}+\Xi^{-N}\|u(\cdot,0)\|_{L^1(\R^N)}.
$$
The choice $\Xi=\|u_0\|_{L^1(\R^N)}^{(m-1)\theta}(\Lambda t)^\theta$ and $\Lambda t\mapsto t$, then gives
$$
\|u(\cdot,t)\|_{L^\infty(\R^N)}\leq\frac{C}{t^{N\theta}}\|u_0\|_{L^1(\R^N)}^{\alpha\theta}\qquad\text{for a.e. $t>0$},
$$
where $\theta=(\alpha+N(m-1))^{-1}$ and $C$ now depending on $C_p$ instead of $K_1$ and $K_2$.\footnote{In the case $\alpha_i=\alpha$ for all $i\in\{1,\ldots,N\}$, the bilinear form of the operator \eqref{eq:SumFracLap} is comparable to the bilinear form of the fractional Laplacian, and one could instead use the Sobolev inequality for the latter operator (see e.g. \cite{DPQuRi22}) together with a Moser iteration to obtain the $L^1$--$L^\infty$-smoothing.} This in turn implies the corresponding Nash inequality (see Section \ref{sec:EquivalencesGNSHomogeneous}). Note that, even in the case $\alpha_i\neq \alpha_j$, one can deduce the Sobolev inequality, from which the Nash inequality follows, by scratch \cite[Theorem 2.4]{ChKa20}. This will then ensure the $L^1$--$L^\infty$-smoothing estimate both in the linear and nonlinear case.
\end{enumerate}
\end{remark}

\begin{proof}[Proof of Lemma \ref{lem:GreenAnisotropicLaplacians}]
Recall that the heat kernel is given by \eqref{eq:ProductOf1DFracLap}. Since we are considering a L\'evy operator, it provides $L^1$-decay, and by \eqref{eq:BoundOnHeatKernelFractionalLaplacian} with $N=1$, we get
\begin{equation*}
\begin{split}
&\|\mathbb{G}_{I-\Levy^\mu}^{0}\|_{L^p(\R^N)}\leq \int_{0}^\infty\textup{e}^{-t}\|\mathbb{H}_{-\Levy^\mu}^{0}(\cdot,t)\|_{L^p(\R^N)}\dd t\\
&\leq \int_{0}^\infty\textup{e}^{-t}\|\mathbb{H}_{-\Levy^\mu}^{0}(\cdot,t)\|_{L^\infty(\R^N)}^{\frac{p-1}{p}}\|\mathbb{H}_{-\Levy^\mu}^{0}(\cdot,t)\|_{L^1(\R^N)}^{\frac{1}{p}}\dd t= \int_{0}^\infty\textup{e}^{-t}\bigg\|\prod_{i=1}^N \mathbb{H}_{(-\dell_{x_ix_i}^2)^{\alpha_i/2}}^{0}(\cdot_i,t)\bigg\|_{L^\infty(\R^N)}^{\frac{p-1}{p}}\dd t\\
&\leq \int_{0}^\infty\textup{e}^{-t}\prod_{i=1}^N\| \mathbb{H}_{(-\dell_{x_ix_i}^2)^{\alpha_i/2}}^{0}(\cdot_i,t)\|_{L^\infty(\R)}^{\frac{p-1}{p}}\dd t\lesssim \int_{0}^\infty\textup{e}^{-t}\prod_{i=1}^N\bigg(\frac{1}{\alpha_i}t^{-\frac{1}{\alpha_i}}\Gamma\Big(\frac{1}{\alpha_i}\Big)\bigg)^{\frac{p-1}{p}}\dd t\\
&\eqsim \int_{0}^\infty\textup{e}^{-t}t^{-\frac{p-1}{p}\sum_{i=1}^n\frac{1}{\alpha_i}}\dd t.
\end{split}
\end{equation*}
Again, by \eqref{GammaFunction}, the result follows.
\end{proof}

Note that we have exploited the fact that $\mathbb{H}_{-\Levy^{\mu}}^{x_0}(x,t)$ indeed has the on-diagonal upper bound $t^{-\sum_{i=1}^N\frac{1}{\alpha_i}}$, see e.g. Corollary 3.2 in \cite{Xu13}. Since the off-diagonal bound cannot give a useful Green function estimate (in all cases), we resort to our assumption \eqref{G_3}. We refer the reader to Remark \ref{rem:LinearEquivalences} which provides various examples of on-diagonal bounds.

Let us turn our attention to another interesting example where only a useful on-diagonal bound can be deduced:

\begin{lemma}\label{lem:OperatorWithDerivativesAtZero}
Assume $\Operator=\Levy^\mu$ with a measure $\mu$ satisfying \eqref{muas} and, for $\alpha\in(0,2)$ and constants $C_1,C_2, C_3>0$,
$$
\frac{C_1}{|z|^{N+\alpha}}\mathbf{1}_{|z|\leq 1}\leq\frac{\dd \mu}{\dd z}(z)\leq \frac{C_2}{|z|^{N+\alpha}}\mathbf{1}_{|z|\leq 1}\qquad\text{and}\qquad \frac{\dd \mu}{\dd z}(z)\leq C_3\mathbf{1}_{|z|>1}.
$$
Then, there exists a constant $C>0$ such that
$$
\mathbb{H}_{-\Levy^{\mu}}^{x_0}(x,t)\leq Ct^{-\frac{N}{\alpha}}\e^{t},
$$
and, moreover,
$$
\|\mathbb{G}_{I-\Levy^\mu}^{x_0}\|_{L^p(\R^N)}=\|\mathbb{G}_{I-\Levy^\mu}^{0}\|_{L^p(\R^N)}\leq C_{p}<\infty
$$
for some $p\in(1,N/(N-\alpha))$.
\end{lemma}

\begin{proof}
The estimate on the heat kernel follows by the beginning of Section 2 in \cite{ChKiKu09} with $V(r)=r^N$ and $\phi(r)=r^\alpha$, and since the heat kernel is proven to be H\"older continuous in Section 3 of the same reference (so that the exceptional set is empty). Note that the proof uses that $\mathbb{H}_{-\Levy^{\mu}}^{x_0}(x,t)=\e^{t}\mathbb{H}_{I-\Levy^{\mu}}^{x_0}(x,t)$. We then get
\begin{equation*}
\begin{split}
\|\mathbb{G}_{I-\Levy^\mu}^{0}\|_{L^p(\R^N)}&\leq \int_{0}^\infty\textup{e}^{-t}\|\mathbb{H}_{-\Levy^\mu}^{0}(\cdot,t)\|_{L^p(\R^N)}\dd t\\
&\leq \int_{0}^\infty\textup{e}^{-t}\|\mathbb{H}_{-\Levy^\mu}^{0}(\cdot,t)\|_{L^\infty(\R^N)}^{\frac{p-1}{p}}\|\mathbb{H}_{-\Levy^\mu}^{0}(\cdot,t)\|_{L^1(\R^N)}^{\frac{1}{p}}\dd t\\
&\eqsim \int_{0}^\infty\textup{e}^{-t}\big(t^{-\frac{N}{\alpha}}\e^{t}\big)^{\frac{p-1}{p}}\dd t=\int_{0}^\infty\textup{e}^{-\frac{t}{p}}t^{-\frac{N}{\alpha}\frac{p-1}{p}}\dd t.
\end{split}
\end{equation*}
Again, by \eqref{GammaFunction}, the result follows.
\end{proof}

\subsection{A nonexample of our theory}
We will now consider a L\'evy operator which does not satisfy any of \eqref{G_1}--\eqref{G_3}.

Consider the generator of a subordinate Brownian motion with Fourier symbol $\phi(|\xi|^2)$ where $\phi(\lambda):=\log(1+\lambda^{\frac{\alpha}{2}})$. This process is known as a rotationally invariant \emph{geometric} $\alpha$-stable process, see Section 5 in \cite{BoByKuRySoVo09}.

\begin{lemma}\label{lem:GeometricAlphaStableGreenResolvent}
Assume $\Operator=\Levy^\mu$ with a measure $\mu$ satisfying \eqref{muas}. If $\Levy^\mu$ has Fourier symbol given by
$$
\log(1+|\xi|^\alpha),
$$
then the heat kernel is given by
$$
\mathbb{H}_{-\Levy^{\mu}}^{x_0}(x,t)=\frac{1}{\Gamma(t)}\int_0^\infty \mathbb{H}_{(-\Delta)^{\frac{\alpha}{2}}}^{x_0}(x,s)s^{t-1}\e^{-s}\dd s.
$$
Moreover,
$$
\|\mathbb{G}_{I-\Levy^{\mu}}^{x_0}\|_{L^1(\R^N)}=\|\mathbb{G}_{I-\Levy^{\mu}}^{0}\|_{L^1(\R^N)}\leq C_{1}<\infty.
$$
\end{lemma}

\begin{remark}
The operator is indeed a L\'evy operator, therefore it is not surprising that it provides $L^1$-decay. It is, moreover, worth noting that Theorems 5.45 and 5.46 in \cite{BoByKuRySoVo09} establish that the density of the L\'evy measure corresponding to the rotationally invariant geometric $\alpha$-stable process satisfies
$$
\frac{\dd\mu}{\dd z}(z)\eqsim \frac{1}{|z|^N}\quad\text{as $|z|\to0$}\qquad\text{and}\qquad\frac{\dd\mu}{\dd z}(z)\eqsim \frac{1}{|z|^{N+\alpha}}\quad\text{as $|z|\to\infty$}.
$$
In fact,
$$
\frac{\dd\mu}{\dd z}(z)=\int_0^\infty \mathbb{H}_{(-\Delta)^{\frac{\alpha}{2}}}^{0}(z,s)s^{-1}\e^{-s}\dd s,
$$
see equation (5.69) in \cite{BoByKuRySoVo09}.
\end{remark}

\begin{proof}[Proof of Lemma \ref{lem:GeometricAlphaStableGreenResolvent}]
The formula for the heat kernel is given by equation (5.68) in \cite{BoByKuRySoVo09}. Moreover, since $(-\Delta)^{\frac{\alpha}{2}}$ actually provides conservation of mass, we get
\begin{equation*}
\begin{split}
\|\mathbb{G}_{I-\Levy^{\mu}}^{0}\|_{L^1(\R^N)}&= \int_{0}^\infty\textup{e}^{-t}\|\mathbb{H}_{-\Levy^{\mu}}^{0}(\cdot,t)\|_{L^1(\R^N)}\dd t\\
&= \int_{0}^\infty\textup{e}^{-t}\frac{1}{\Gamma(t)}\int_0^\infty \|\mathbb{H}_{(-\Delta)^{\frac{\alpha}{2}}}^{0}(\cdot,s)\|_{L^1(\R^N)}s^{t-1}\e^{-s}\dd s\dd t\\
&= \int_{0}^\infty\textup{e}^{-t}\frac{1}{\Gamma(t)}\int_0^\infty s^{t-1}\e^{-s}\dd s\dd t=\int_{0}^\infty\textup{e}^{-t}\frac{\Gamma(t)}{\Gamma(t)}\dd t=1. \qedhere
\end{split}
\end{equation*}
\end{proof}

The proof also demonstrates that \eqref{G_2} cannot hold, and moreover, neither can \eqref{G_1}, \eqref{G_1'}, and \eqref{G_3}:

\begin{lemma}\label{lem:GeometricAlphaStableGreenResolvent2}
Assume $\Operator=\Levy^\mu$ with a measure $\mu$ satisfying \eqref{muas}. If $\Levy^\mu$ has Fourier symbol given by
$$
\log(1+|\xi|^\alpha),
$$
then there is some $R>0$ such that
$$
\int_{B_R(x_0)}\mathbb{G}_{-\Levy^{\mu}}^{x_0}(x)\dd x> CR^{\alpha} \qquad\text{for all $\alpha\in(0,2)$ and all $C>0$,}
$$
and
$$
\|\mathbb{G}_{I-\Levy^{\mu}}^{x_0}\|_{L^p(\R^N)}=\|\mathbb{G}_{I-\Levy^{\mu}}^{0}\|_{L^p(\R^N)}=\infty\qquad\text{for all $p>1$.}
$$
\end{lemma}

\begin{proof}
By Theorem 5.35 in \cite{BoByKuRySoVo09},
$$
\mathbb{G}_{-\Levy^{\mu}}^{0}(x)\eqsim |x-x_0|^{-N}\big(-\log(|x|^2))^{-2}\eqsim|x-x_0|^{-N}\big(\log(|x|))^{-2}\qquad\text{as $|x|\to0$.}
$$
Then for small enough $0<R<<1$,
$$
\int_{B_R(0)}\mathbb{G}_{-\Levy^{\mu}}^{0}(x)\dd x\eqsim \int_0^Rr^{-1}(\log(r))^{-2}\dd r\eqsim\Big(\log\Big(\frac{1}{R}\Big)\Big)^{-1}.
$$
Now, the statement
$$
\Big(\log\Big(\frac{1}{R}\Big)\Big)^{-1}> CR^{\alpha}\qquad\text{for some $0<R<<1$}
$$
is equivalent to
$$
\frac{1}{R}< \textup{exp}\Big(C\Big(\frac{1}{R}\Big)^{\alpha}\Big)\qquad\text{for some $0<R<<1$,}
$$
which is clearly true for all $\alpha\in(0,2)$ and all $C>0$.

We already know that
$$
\|\mathbb{G}_{I-\Levy^{\mu}}^{0}\|_{L^p(\R^N)}= \int_{0}^\infty\textup{e}^{-t}\|\mathbb{H}_{-\Levy^{\mu}}^{0}(\cdot,t)\|_{L^p(\R^N)}\dd t.
$$
By Theorem 5.5.2 in \cite{BoByKuRySoVo09}, for all $0<t\leq \min\{1,N/(2\alpha)\}$,
\begin{equation*}
\begin{split}
\|\mathbb{H}_{-\Levy^{\mu}}^{0}(\cdot,t)\|_{L^p(\R^N)}^p&\geq Ct^p\bigg(\int_{|x|<1}|x|^{-p(N-t\alpha)}\dd x+\int_{|x|>1}|x|^{-p(N+\alpha)}\dd x\bigg)\\
&\gtrsim t^p\int_0^1r^{N-1-p(N-t\alpha)}\dd r.
\end{split}
\end{equation*}
If $1<p\leq 2$, then
$$
\frac{N}{\alpha}\frac{p-1}{p}\leq \min\Big\{1,\frac{N}{2\alpha}\Big\},
$$
and
\begin{equation*}
\begin{split}
\|\mathbb{G}_{I-\Levy^{\mu}}^{0}\|_{L^p(\R^N)}&\geq\int_0^{\frac{N}{\alpha}\frac{p-1}{p}}\textup{e}^{-t}\|\mathbb{H}_{-\Levy^{\mu}}^{0}(\cdot,t)\|_{L^p(\R^N)}\dd t\\
&\gtrsim \int_0^{\frac{N}{\alpha}\frac{p-1}{p}}\textup{e}^{-t}t\bigg(t^p\lim_{\xi\to0^+}\int_\xi^1r^{N-1-p(N-t\alpha)}\dd r\bigg)^\frac{1}{p}\dd t\\
&=\int_0^{\frac{N}{\alpha}\frac{p-1}{p}}\textup{e}^{-t}t^2\bigg(\frac{1}{p\alpha}\frac{1}{\frac{N}{\alpha}\frac{(p-1)}{p}- t}\bigg)^{\frac{1}{p}}\bigg(\lim_{\xi\to0^+}\frac{1}{\xi^{p\alpha(\frac{N}{\alpha}\frac{p-1}{p}-t)}}-1\bigg)^{\frac{1}{p}}\dd t\\
&\geq \int_0^{\frac{N}{\alpha}\frac{p-1}{p}}\textup{e}^{-t}t^2\bigg(\frac{1}{N(p-1)}\bigg)^{\frac{1}{p}}\bigg(\lim_{\xi\to0^+}\frac{1}{\xi^{p\alpha(\frac{N}{\alpha}\frac{p-1}{p}-t)}}-1\bigg)^{\frac{1}{p}}\dd t=\infty.
\end{split}
\end{equation*}
If $p>2$, then
$$
\frac{N}{\alpha}\frac{p-1}{p}>\frac{N}{\alpha}\frac{1}{2},
$$
and we simply consider
$$
\|\mathbb{G}_{I-\Levy^{\mu}}^{0}\|_{L^p(\R^N)}\geq\int_0^{\min\{1,N/(2\alpha)\}}\textup{e}^{-t}\|\mathbb{H}_{-\Levy^{\mu}}^{0}(\cdot,t)\|_{L^p(\R^N)}\dd t
$$
to reach the same conclusion.
\end{proof}

%%%%%%%%%%%%%%%%%%%%%%%%%%%%%%%%%%%%%%%%%%%%%%%%%%%%
%%%%%%%%%%%%%%%%%%%%%NEW SECTION%%%%%%%%%%%%%%%%%%%%%%%
%%%%%%%%%%%%%%%%%%%%%%%%%%%%%%%%%%%%%%%%%%%%%%%%%%%%

\section*{Acknowledgements}
%\addcontentsline{toc}{section}{Acknowledgements}

J. Endal has received funding from the European Union’s Horizon 2020 research and innovation programme under the Marie Sk\l odowska-Curie grant agreement no. 839749 ``Novel techniques for quantitative behavior of convection-diffusion equations (techFRONT)'', and from the Research Council of Norway under the MSCA-TOPP-UT grant agreement no. 312021.

M. Bonforte was partially supported by the Projects MTM2017-85757-P
and PID2020-113596GB-I00 (Spanish Ministry of Science and Innovation). M. Bonforte moreover acknowledges financial support from the Spanish Ministry of Science and Innovation, through the ``Severo Ochoa Programme
for Centres of Excellence in R\&D'' (CEX2019-000904-S) and by the European Union’s Horizon 2020 research and innovation programme under the Marie Sk\l odowska-Curie grant agreement no. 777822.

We would like to thank Nikita Simonov for fruitful discussions, and the anonymous referee for a thorough reading and insightful suggestions that helped us improve the paper.

{\bf Conflict of interest statement. }On behalf of all authors, the corresponding author states that there is no conflict of interest.

{\bf Data Availability Statements. }All data generated or analysed during this study are included in this published article.

%%%%%%%%%%%%%%%%%%%%%%%%%%%%%%%%%%%%%%%%%%%%%%%%%%%%
%%%%%%%%%%%%%%%%%%%%%APPENDIX%%%%%%%%%%%%%%%%%%%%%%%
%%%%%%%%%%%%%%%%%%%%%%%%%%%%%%%%%%%%%%%%%%%%%%%%%%%%

\appendix

%%%%%%%%%%%%%%%%%%%%%%%%%%%%%%%%%%%%%%%%%%%%%%%%%%%%
%%%%%%%%%%%%%%%%%%%%%NEW SECTION%%%%%%%%%%%%%%%%%%%%%%%
%%%%%%%%%%%%%%%%%%%%%%%%%%%%%%%%%%%%%%%%%%%%%%%%%%%%

\section{Technical lemmas}

Implicitly, we use the following in the Moser iteration:

\begin{lemma}\label{lem:LIAsLimitOfLp}
Assume $K(p)>0$ is such that $\lim_{p\to\infty}K(p)<\infty$, $p_0\geq1$, and
$$
\psi\in L^p(\R^N) \qquad\text{and}\qquad \|\psi\|_{L^p(\R^N)}\leq K(p) \quad\text{for all $p\in[p_0,\infty)$.}
$$
Then $\psi\in L^\infty(\R^N)$, and moreover,
$$
\|\psi\|_{L^\infty(\R^N)}\leq \lim_{p\to\infty}K(p).
$$
\end{lemma}

\begin{proof}
Define $K(\infty):=\lim_{p\to\infty}K(p)$, and consider
$$
\Psi:=|\psi|\mathbf{1}_{\{|\psi|\leq K(\infty)+1\}}+(K(\infty)+1)\mathbf{1}_{\{|\psi|>K(\infty)+1\}}=\min\{|\psi|,K(\infty)+1\}
$$
from which it follows that $\Psi\leq K(\infty)+1$ and $\Psi\leq |\psi|$. Then $\Psi\in L^\infty(\R^N)$ and $\|\Psi\|_{L^p}\leq \|\psi\|_{L^p}\leq K(p)$ for all $p\in[p_0,\infty)$, and hence, $\|\Psi\|_{L^\infty}=\lim_{p\to\infty}\|\Psi\|_{L^p}\leq  K(\infty)$. But then $\min\{|\psi|,K(\infty)+1\}\leq K(\infty)$ which implies $\|\psi\|_{L^\infty}\leq K(\infty)$.
\end{proof}

The next lemma is classical, but we state and prove it for completeness.

\begin{lemma}[A DeGiorgi-type lemma]\label{lem:DeGiorgi}
Assume that $z\in \R^N$, and that $f\in L^\infty(B_{3R}(z))$ with $R>0$ fixed. If, for any $R\leq\rho<\bar{\rho}\leq 3R$, some $\delta\in(0,1)$, and some $M>0$ independent of $\rho, \bar{\rho}$, we have that
$$
\|f\|_{L^\infty(B_{\rho}(z))}\leq \delta \|f\|_{L^\infty(B_{\bar{\rho}}(z))}+M,
$$
then
$$
\|f\|_{L^\infty(B_{\rho}(z))}\leq \frac{1}{1-\delta}M.
$$
\end{lemma}

\begin{proof}
We follow the proof of Lemma 1.2 in Chapter 4 of \cite{HaLi97}. Fix $\rho\geq R$.  For some $0<\eta<1$ we consider the sequence $\{\rho_i\}$ defined recursively by
$$
\rho_0=\rho \qquad\text{and}\qquad \rho_{i+1}:=\rho_i+(1-\eta)\eta^{i}\rho.
$$
Note that $\rho_\infty=2\rho$.
Since $2\rho=\rho_\infty>\ldots>\rho_1>\rho_0=\rho$,
\begin{equation*}
\begin{split}
\|f\|_{L^\infty(B_{\rho}(z))}&=\|f\|_{L^\infty(B_{\rho_0}(z))}\leq \delta \|f\|_{L^\infty(B_{\rho_1}(z))}+M\leq \delta^2\|f\|_{L^\infty(B_{\rho_2}(z))}+(1+\delta)M\\
&\leq \ldots\leq \delta^k\|f\|_{L^\infty(B_{\rho_k}(z))}+M\sum_{i=0}^{k-1}\delta^{i}.
\end{split}
\end{equation*}
The conclusion follows by letting $k\to\infty$.
\end{proof}

%%%%%%%%%%%%%%%%%%%%%%%%%%%%%%%%%%%%%%%%%%%%%%%%%%%%
%%%%%%%%%%%%%%%%%%%%%NEW SECTION%%%%%%%%%%%%%%%%%%%%%%%
%%%%%%%%%%%%%%%%%%%%%%%%%%%%%%%%%%%%%%%%%%%%%%%%%%%%

\section{\texorpdfstring{$L^1$}{L1}--\texorpdfstring{$L^\infty$}{Linfty}-smoothing controls \texorpdfstring{$L^q$}{Lq}--\texorpdfstring{$L^p$}{Lp}-smoothing}\label{sec:SmoothingImplications}

Throughout this section, $F>0$ is some nonincreasing function and $C>0$ is some constant (which might change from line to line) not depending on any norm of $u$.

\begin{theorem}\label{thm:SmoothingImplications}
Assume that $0\leq u_0=u(\cdot,0)\in (L^1\cap L^\infty)(\R^N)$, and that, for a.e. $t_1> t_0\geq0$,
$$
\|u(t_1)\|_{L^\infty}\leq F(t_1-t_0)\|u(t_0)\|_{L^q}^\gamma\qquad\text{for some $0\leq \gamma<1$ and $q\in[1,\infty)$.}
$$
Then, for all $1\leq q\leq p\leq \infty$,
$$
\|u(t_1)\|_{L^{p}}\leq G(t_1-t_0)H(\|u(t_0)\|_{L^{q}}),
$$
where $G,H\geq0$ are functions depending on $F, \gamma, p, q$ and $\gamma, p, q$, respectively, which has to be determined in each case.
\end{theorem}

The proof is a consequence of several results in this section. We start by investigating immediate consequences of $L^q$--$L^\infty$-smoothing effects through Young and H\"older inequalities, in addition to a DeGiorgi type lemma.

\begin{lemma}\label{lem:SmoothingImplications}
Assume that $0\leq u_0=u(\cdot,0)\in (L^1\cap L^\infty)(\R^N)$, and that, for a.e. $t_1> t_0\geq0$,
$$
\|u(t_1)\|_{L^\infty}\leq F(t_1-t_0)\|u(t_0)\|_{L^q}^\gamma\qquad\text{for some $0\leq \gamma<1$ and $q\in[1,\infty)$.}
$$
Then:
\begin{enumerate}[{\rm (a)}]
\item \textup{($L^{\leq q}$--$L^\infty$-smoothing)}
$$
\|u(t_1)\|_{L^\infty}\leq CF(t_1-t_0)^{\frac{q}{(1-\gamma)q+\gamma r}}\|u(t_0)\|_{L^r}^{\frac{\gamma r}{(1-\gamma)q+\gamma r}}\qquad\text{for $r\in[1,q]$.}
$$
\item \textup{($L^{\leq q}$--$L^{> q}$-smoothing)}
$$
\|u(t_1)\|_{L^p}\leq CF(t_1-t_0)^{\frac{q}{p}\frac{p-r}{(1-\gamma)q+\gamma r}}\|u(t_0)\|_{L^r}^{\frac{r}{p}\frac{(1-\gamma)q+\gamma p}{(1-\gamma)q+\gamma r}}\qquad\text{for $p\in(r,\infty)$ and $r\in[1,q]$.}
$$
\end{enumerate}
\end{lemma}

\begin{remark}
The homogeneous smoothing estimates (cf. Theorem \ref{thm:L1ToLinfinitySmoothing2}(a) and Theorem \ref{thm:AbsBounds}) can respectively be recovered by choosing $q=1=r$ and:
\begin{enumerate}[{\rm (i)}]
\item If $\gamma=\alpha\theta$, then
$$
\frac{\gamma(p-1)+1}{p}=\frac{1}{p}\frac{\theta_1}{\theta_p}.
$$
\item If $\gamma=0$, then
$$
\frac{\gamma(p-1)+1}{p}=\frac{1}{p}.
$$
\end{enumerate}
\end{remark}

\begin{proof}[Proof of Lemma \ref{lem:SmoothingImplications}]
\noindent (a) Since $q\geq r\geq 1$, we have
$$
\|u(t_0)\|_{L^q}^{\gamma}\leq \|u(t_0)\|_{L^\infty}^{\frac{\gamma(q-r)}{q}}\|u(t_0)\|_{L^r}^{\frac{\gamma r}{q}}.
$$
Applying the Young inequality \eqref{eq:Young} with $\vartheta=\frac{q}{\gamma(q-r)}>1$ yields
$$
\|u(t_1)\|_{L^\infty}\leq \frac{\gamma(q-r)}{q}\|u(t_0)\|_{L^\infty}+CF(t_1-t_0)^{\frac{q}{(1-\gamma)q+\gamma r}}\|u(t_0)\|_{L^r}^{\frac{\gamma r}{(1-\gamma)q+\gamma r}}.
$$
Since $\gamma(q-r)/q<1$, we can reabsorb the $L^\infty$-norm by a variant of a classical lemma due to DeGiorgi (cf. Lemma \ref{lem:DeGiorgi}).

\smallskip
\noindent (b) When $p\geq r\geq 1$, we use Proposition \ref{prop:APriori}(b)(ii) to get
$$
\|u(t_1)\|_{L^p}\leq \|u(t_1)\|_{L^\infty}^{\frac{p-r}{p}}\|u(t_1)\|_{L^r}^{\frac{r}{p}}\leq \|u(t_1)\|_{L^\infty}^{\frac{p-r}{p}}\|u(t_0)\|_{L^r}^{\frac{r}{p}}.
$$
Now, part (a) yields
\begin{equation*}
\begin{split}
\|u(t_1)\|_{L^p}&\leq CF(t_1-t_0)^{\frac{q}{(1-\gamma)q+\gamma r}\frac{p-r}{p}}\|u(t_0)\|_{L^r}^{\frac{\gamma r}{(1-\gamma)q+\gamma r}\frac{p-r}{p}}\|u(t_0)\|_{L^r}^{\frac{r}{p}}\\
&=CF(t_1-t_0)^{\frac{q}{p}\frac{p-r}{(1-\gamma)q+\gamma r}}\|u(t_0)\|_{L^r}^{\frac{r}{p}\frac{(1-\gamma)q+\gamma p}{(1-\gamma)q+\gamma r}}.
\end{split}
\end{equation*}
The proof is complete.
\end{proof}

As we saw, the easy consequences of smoothing effects, is to loose integrability on the right-hand side. Now, we instead want to gain integrability, i.e., we want $L^1$--$L^\infty$ to $L^{\geq 1}$--$L^\infty$.

To gain integrability, however, requires a refined technique. In bounded domains $|\Omega|<\infty$, this is usually accomplished by the fact that $L^{\tilde{q}}\subseteq L^q$ with $1\leq q\leq \tilde{q}\leq \infty$, i.e., by the H\"older inequality
$$
\|u\|_{L^q(\Omega)}^q=\int_{\Omega}|u|^q\dd x\leq \bigg(\int_{\Omega}\big(|u|^q\big)^{\frac{\tilde{q}}{q}}\dd x\bigg)^{\frac{q}{\tilde{q}}}\bigg(\int_{\Omega}\big(1\big)^{\frac{\tilde{q}}{\tilde{q}-q}}\dd x\bigg)^{\frac{\tilde{q}-q}{\tilde{q}}}=\|u\|_{L^{\tilde{q}}(\Omega)}^q|\Omega|^{\frac{\tilde{q}-q}{\tilde{q}}}.
$$
So, while such a statement is trivial in bounded domains, the story is quite different in $\R^N$. The reason for this can be seen by the following estimate:
$$
\|u\|_{L^q}^q=\int_{B_R(0)}|u|^q+\int_{\R^N\setminus B_R(0)}|u|^q.
$$
On the small ball, we use that $L^{\tilde{q}}\subseteq L^q$, but what to do on the complementary set? In fact, there the ``natural'' ordering is
$$
\frac{1}{(1+|x|)^{\tilde{q}}}\leq \frac{1}{(1+|x|)^q},
$$
i.e., opposite of the small ball. Hence,
$$
\int_{\R^N\setminus B_R(0)}|u|^q\leq C_R\bigg(\int_{\R^N\setminus B_R(0)}|u|^{\tilde{q}}\bigg)^{\frac{q}{\tilde{q}}}
$$
cannot be true for all functions, and must be a property of the equation itself.

We therefore rely on a nice idea taken from Section 3.1 in \cite{Vaz06}: Consider \eqref{GPME} with the nonlinearity
$$
\varphi: r\mapsto (r+\veps)^m-\veps^m \qquad\text{for some $\veps>0$ and some $m>1$.}
$$
In the case of L\'evy operators \eqref{def:LevyOperators} with $c=0$, this equation has been studied in \cite{DTEnJa17a}, and in the case of the more general setting of possibly $x$-dependent $\mathfrak{m}$-accretive operators, this equation has been studied in detail in e.g. \cite[Appendix B]{BoFiR-O17} and \cite{DPQuRo16}. E.g., existence, uniqueness, and the comparison principle holds for sign-changing solutions. In our setting, we also have:

\begin{lemma}\label{lem:EquivalenceOfL1LinftySmoothing}
Assume \eqref{phias} and $\veps>0$. Let $u$ be a weak dual solution of \eqref{GPME} with initial data $0\leq u_0\in L^1(\R^N)$, and $v$ be a weak dual solution of \eqref{GPME} with nonlinearity $\varphi$ and initial data $0\leq v_0\in L^1(\R^N)$. Under suitable assumptions on the associated Green function:

\smallskip
\noindent{\rm (i)} For a.e. $t_1> t_0\geq0$,
$$
\|u(t_1)\|_{L^\infty}\leq F(t_1-t_0)\|u(t_0)\|_{L^1}^\gamma\qquad\text{for some $0\leq \gamma<1$.}
$$

\smallskip
If, moreover, $u(x,t)\leq v(x,t)+\veps$ for a.e. $(x,t)\in Q_T$, then:

\smallskip
\noindent{\rm (ii)} For a.e. $t_1> t_0\geq0$,
$$
\|v(t_1)\|_{L^\infty}\leq F(t_1-t_0)\|v(t_0)\|_{L^1}^\gamma+C\veps,
$$
with the same $F, \gamma$ as given above.
\end{lemma}

\begin{proof}
\noindent(i) This already holds, see Theorems \ref{thm:L1ToLinfinitySmoothing2} and \ref{thm:AbsBounds}.

\smallskip
\noindent(ii) Even though the nonlinearity $\varphi(r)$ is different from $r^m$, we will see that we can repeat the steps leading up to Theorems \ref{thm:L1ToLinfinitySmoothing2} and \ref{thm:AbsBounds} when we in addition know that $u\leq v+\veps$.

Recall that by Lemma \ref{prop:Monotonicity}, we have
$$
t\mapsto t^{\frac{m}{m-1}}u^m(\cdot,t)\qquad\text{is nondecreasing for a.e. $x\in \R^N$,}
$$
and since $u\leq v+\veps$, we get the following time-monotonicity for $v+\veps$:
$$
t\mapsto t^{\frac{m}{m-1}}(v(\cdot,t)+\veps)^m\qquad\text{is nondecreasing for a.e. $x\in \R^N$.}
$$
Then
\begin{equation*}
\begin{split}
&\int_{\tau_*}^\tau (v(x_0,t)+\veps)^m\dd t-\veps^m(\tau-\tau_*)= \int_{\tau_*}^\tau \varphi(v(x_0,t))\dd t\\
&= \int_{\R^N}v(x,\tau_*)\mathbb{G}_{-\Operator}^{x_0}(x)\dd x-\int_{\R^N}v(x,\tau)\mathbb{G}_{-\Operator}^{x_0}(x)\dd x\leq \int_{\R^N}v(x,\tau_*)\mathbb{G}_{-\Operator}^{x_0}(x)\dd x,
\end{split}
\end{equation*}
or
\begin{equation*}
\begin{split}
\int_{\tau_*}^\tau (v(x_0,t)+\veps)^m\dd t\leq \int_{\R^N}v(x,\tau_*)\mathbb{G}_{-\Operator}^{x_0}(x)\dd x+\veps^m(\tau-\tau_*).
\end{split}
\end{equation*}
By the time-monotonicity,
$$
(v(x_0,\tau_*)+\veps)^m\leq \frac{C(m)}{\tau_*}\int_{\R^N}v(x,\tau_*)\mathbb{G}_{-\Operator}^{x_0}(x)\dd x+C(m)\veps^m.
$$
Since $r\mapsto r^m$ is superadditive on $\R_+$, we also get
$$
v^m(x_0,\tau_*)\leq \frac{C(m)}{\tau_*}\int_{\R^N}v(x,\tau_*)\mathbb{G}_{-\Operator}^{x_0}(x)\dd x+(C(m)-1)\veps^m.
$$
Hence, we get the stated $L^1$--$L^\infty$-smoothing for $v$ by simply following the proof for $u$, and using that $r\mapsto r^{\frac{1}{m}}$ is subadditive on $\R_+$.
\end{proof}

\begin{proposition}\label{prop:SmoothingImplications}
Assume that $0\leq u_0=u(\cdot,0)\in (L^1\cap L^p)(\R^N)$ for some $p\in(1,\infty)$, and that, for a.e. $t_1> t_0\geq0$,
$$
\|u(t_1)\|_{L^\infty}\leq F(t_1-t_0)\|u(t_0)\|_{L^1}^{\gamma}\qquad\text{for some $\gamma\in [0,1)$.}
$$
Moreover, if the comparison principle holds for sign-changing weak dual solutions of \eqref{GPME} with the nonlinearity $\varphi$, then
$$
\|u(t_1)\|_{L^\infty}\leq CF(t_1-t_0)^{\frac{1}{\gamma(p-1)+1}}\|u(t_0)\|_{L^p}^{\frac{\gamma p}{\gamma(p-1)+1}}.
$$
\end{proposition}

\begin{remark}
The comparison principle indeed holds for sign-changing weak dual solutions of \eqref{GPME} with the nonlinearity $\varphi$: We simply repeat the existence proof in this setting.
\end{remark}

\begin{remark}
Let us check that we indeed recover the different homogeneous cases:
\begin{enumerate}[{\rm (i)}]
\item If $\gamma_1=\gamma=\alpha\theta$, $F(t)=Ct^{-\gamma_2}$, and $\gamma_2=N\theta$, then
$$
\frac{\gamma_1 p}{\gamma_1(p-1)+1}=\alpha p\theta_p\qquad\text{and}\qquad \frac{\gamma_2}{\gamma_1(p-1)+1}=N\theta_p.
$$
\item If $\gamma_1=\gamma=0$, $F(t)=Ct^{-\gamma_2}$, and $\gamma_2=1/(m-1)$, then
$$
\frac{\gamma_1 p}{\gamma_1(p-1)+1}=0\qquad\text{and}\qquad \frac{\gamma_2}{\gamma_1(p-1)+1}=\frac{1}{m-1}.
$$
\end{enumerate}
\end{remark}

\begin{proof}[Proof of Proposition \ref{prop:SmoothingImplications}]
Consider the function $v_\veps:=u-\veps$ with $\veps>0$, where $u$ solves \eqref{GPME} with initial data $u_0\geq0$. Note that $v_\veps$ also solves \eqref{GPME} with $\varphi$ as nonlinearity (we have subtracted the term $\veps^m$ for normalization purposes), and unsigned initial data $u_0-\veps$.

Now, consider the solution $\tilde{v}_\veps$ of \eqref{GPME} with nonlinearity $\varphi$ and initial data $(u_0-\veps)_+$. Due to the H\"older inequality,
\begin{equation*}
\begin{split}
\int (u_0-\veps)_+=\int_{\{u_0>\veps\}}(u_0-\veps)\leq \int_{\{u_0>\veps\}}u_0=\int u_0\mathbf{1}_{\{u_0>\veps\}}\leq \|u_0\|_{L^p}|\{u_0>\veps\}|^{\frac{p-1}{p}}.
\end{split}
\end{equation*}
Moreover, for any $f\geq0$,
$$
\|f\|_{L^p}^p=\int_{\{0\leq f\leq \veps\}}|f|^p+\int_{\{f>\veps\}}|f|^p\geq \int_{\{f>\veps\}}f^p\geq \veps^p|\{f>\veps\}|.
$$
Hence,
$$
\int (u_0-\veps)_+\leq \frac{\|u_0\|_{L^p}^p}{\veps^{p-1}}.
$$
In particular, $(u_0-\veps)_+\in L^1$ as long as $u_0\in L^p$.

The comparison principle for sign-changing solutions of \eqref{GPME} with nonlinearity $\varphi$ then gives
$$
u_0(x)-\veps\leq (u_0(x)-\veps)_+\qquad\implies\qquad v_\veps(x,t)\leq \tilde{v}_\veps(x,t).
$$
From which we conclude that
$$
u(x,t)\leq \tilde{v}_\veps(x,t)+\veps \qquad\implies\qquad \|u(t)\|_{L^\infty}\leq \|\tilde{v}_\veps(t)\|_{L^\infty}+\veps.
$$
We are then in the setting of Lemma \ref{lem:EquivalenceOfL1LinftySmoothing}, and, for all $\veps>0$,
$$
\|u(t)\|_{L^\infty}\leq \|\tilde{v}_\veps(t)\|_{L^\infty}+\veps\leq F(t)\|(u_0-\veps)_+\|_{L^1}^{\gamma}+c\veps\leq F(t)\|u_0\|_{L^p}^{\gamma p}\veps^{-\gamma(p-1)}+c\veps.
$$
To conclude, we infimize over $\veps>0$.
\end{proof}

%%%%%%%%%%%%%%%%%%%%%%%%%%%%%%%%%%%%%%%%%%%%%%%%%%%%
%%%%%%%%%%%%%%%%%%%%%NEW SECTION%%%%%%%%%%%%%%%%%%%%%%%
%%%%%%%%%%%%%%%%%%%%%%%%%%%%%%%%%%%%%%%%%%%%%%%%%%%%

\section{Densely defined, \texorpdfstring{$\mathfrak{m}$}{m}-accretive, and Dirichlet in \texorpdfstring{$L^1(\R^N)$}{L1RN}}
\label{sec:mAccretiveDirichlet}

The existence proof for weak dual solutions is based on the concept of \emph{mild} solutions, which again relies on finding a.e.-solutions (!) of the corresponding elliptic problem
$$
\forall \lambda>0\qquad u+\lambda A[u^m]=f\qquad\text{in $\R^N$.}
$$
We will therefore study so-called $\mathfrak{m}$-accretive operators $A$.

\subsection{The setting of abstract solutions}

The Laplacian $(-\Delta)$ (as well as $r\mapsto (-\Delta)[r^m]$) is a well-known example of an operator which is $\mathfrak{m}$\emph{-accretive} in $L^1(\R^N)$ \cite[Section 10.3.2]{Vaz07}. Now, we want to establish that any \emph{symmetric, nonlocal} and \emph{constant coefficient} L\'evy operator
$$
(-\Levy^\mu)[\psi]=-P.V.\int_{\R^N\setminus\{0\} } \big(\psi(x+z)-\psi(x)\big) \dd\mu(z)
$$
is also $\mathfrak{m}$-accretive in $L^1(\R^N)$. Since that operator is moreover \emph{Dirichlet}, we get by Propositions 1 and 2 in \cite{CrPi82}, that $r\mapsto (-\Levy^\mu)[r^m]$ is also $\mathfrak{m}$-accretive and Dirichlet in $L^1(\R^N)$. Indeed, such a result should be well-known, but we did not manage to find a useful reference for it.

Throughout this section, we stick to the usual notation $A:=(-\Levy^\mu)$.

\begin{theorem}\label{thm:LevyOperatorsmAccretive}
Assume \eqref{muas}. Then the linear operator $A: D(A)\subset L^1(\R^N)\to L^1(\R^N)$ satisfies:
\begin{enumerate}[{\rm (i)}]
\item\label{thm:LevyOperatorsmAccretive1} $\overline{D(A)}^{\|\cdot\|_{L^1(\R^N)}}=L^1(\R^N)$.
\item\label{thm:LevyOperatorsmAccretive2} $A$ is accretive in $L^1(\R^N)$.
\item\label{thm:LevyOperatorsmAccretive3} $R(I+\lambda A)=L^1(\R^N)$ for all $\lambda>0$.
\item\label{thm:LevyOperatorsmAccretive4} If $f\in L^1(\R^N)$ and $a,b\in \R$ such that $a\leq f\leq b$ a.e., then $a\leq (I+\lambda A)^{-1}f\leq b$ a.e.
\end{enumerate}
That is, the linear operator $A$ is densely defined, $\mathfrak{m}$-accretive, and Dirichlet in $L^1(\R^N)$.
\end{theorem}

\begin{remark}\label{rem:LinearResolventDensityAndRange}
\begin{enumerate}[{\rm (a)}]
\item Since $D(A)\subset L^1(\R^N)$, we define
$$
D(A):=\{\psi\in L^1(\R^N) \,:\, \widetilde{A}[\psi]\in L^1(\R^N)\}
$$
where $\widetilde{A}$ is the extension of $A$ to $L^1(\R^N)$ (see \eqref{def:DistrOperatorA} below). Moreover, in our case it is well-known that
$$
A: C_\textup{c}^\infty(\R^N)\subset L^1(\R^N)\to L^p(\R^N)\qquad\text{for all $p\in[1,\infty]$}.
$$
Hence, $C_\textup{c}^\infty(\R^N)\subset D(A)$, and thus, we have already proven Theorem \ref{thm:LevyOperatorsmAccretive}\eqref{thm:LevyOperatorsmAccretive1}. However, we need to make sure that when we define the extension $\widetilde{A}$ (see \eqref{def:DistrOperatorA} below), we have
$$
\widetilde{A} |_{C_\textup{c}^\infty(\R^N)}=A\qquad\text{a.e. in $\R^N$.}
$$
\item In our $\R^N$-case, the numbers $a,b$ in Theorem \ref{thm:LevyOperatorsmAccretive}\eqref{thm:LevyOperatorsmAccretive4} has the natural restriction $a\leq b$, $a\leq 0$, and $b\geq 0$.
\end{enumerate}
\end{remark}

\begin{corollary}[{{\cite[Propositions 1 and 2]{CrPi82}}}]\label{cor:PorousMediummAccretive}
Assume \eqref{muas} and \eqref{phias}. Then the nonlinear operator $r\mapsto Ar^m: D(Ar^m)\subset L^1(\R^N)\to L^1(\R^N)$ is densely defined, $\mathfrak{m}$-accretive, and Dirichlet in $L^1(\R^N)$.

In particular, for all $f\in L^1(\R^N)$ such that $a\leq f\leq b$, there exists a unique $u\in L^1(\R^N)$ which satisfies $a\leq u\leq b$, the comparison principle, $L^p$-decay estimate, and
$$
u+\lambda \widetilde{A}[u^m]=f\qquad\text{a.e. in $\R^N$.}
$$
\end{corollary}

Let us start by proving the range condition. To do so, consider
\begin{equation}\label{eq:LinearResolvent}
\forall \lambda>0\qquad u+\lambda A[u]=f \qquad\text{in $\R^N$}.
\end{equation}

\begin{definition}[Very weak solutions]\label{def:LinearResolventVeryWeak}
Assume \eqref{muas}. We say that $u\in L_\textup{loc}^1(\R^N)$ is a \emph{very weak solution} of \eqref{eq:LinearResolvent} with right-hand side $f\in L_\textup{loc}^1(\R^N)$ if
\begin{equation*}
\begin{split}
\int_{\R^N} u\psi\dd x+\lambda\int_{\R^N}uA[\psi]\dd x=\int_{\R^N} f\psi\dd x\qquad\text{for all $\psi\in C_\textup{c}^\infty(\R^N)$.}
\end{split}
\end{equation*}
\end{definition}

We need the following result (take $u\mapsto\lambda u$ for all $\lambda>0$ and choose $\veps=1/\lambda$):

\begin{lemma}[{{\cite[Theorem 3.1]{DTEnJa17a}}}]\label{lem:LinearResolventWellposedness}
Assume \eqref{muas}.
\begin{enumerate}[{\rm (a)}]
\item If $f\in C_\textup{b}^\infty(\R^N)$, then there exists a unique classical solution $u\in C_\textup{b}^\infty(\R^N)$ of \eqref{eq:LinearResolvent}. Moreover, for each multiindex $\alpha\in \mathbb{N}^N$,
$$
\|D^\alpha u\|_{L^\infty(\R^N)}\leq \|D^\alpha f\|_{L^\infty(\R^N)}.
$$
\item If $f\in L^\infty(\R^N)$, then there exists a unique classical solution $u\in L^\infty(\R^N)$ of \eqref{eq:LinearResolvent}. Moreover,
$$
\|u\|_{L^\infty(\R^N)}\leq \|f\|_{L^\infty(\R^N)}.
$$
\item If $f\in L^1(\R^N)$, then there exists a unique classical solution $u\in L^1(\R^N)$ of \eqref{eq:LinearResolvent}. Moreover,
$$
\|u\|_{L^1(\R^N)}\leq \|f\|_{L^1(\R^N)}.
$$
\end{enumerate}
\end{lemma}

\begin{remark}\label{rem:LinearResolventComparison}
Theorem 6.15 in \cite{DTEnJa17a} gives comparison (or $T$-contraction) as well: If $f\in L_\textup{loc}^1(\R^N)$ such that $(f)^+\in L^1(\R^N)$ and $u\in L_\textup{loc}^1(\R^N)$ is a very weak solution of \eqref{eq:LinearResolvent}, then
$$
\int_{\R^N}(u)^+\dd x\leq \int_{\R^N}(f)^+\dd x.
$$
\end{remark}

\begin{lemma}\label{lem:OperatorEstimateInLp}
Assume $\eqref{muas}$, $p\in(1,\infty]$, $f\in (L^1\cap L^\infty)(\R^N)$. Let $u$ be the very weak solution of \eqref{eq:LinearResolvent} with right-hand side $f$. Then
$$
\lambda\|\widetilde{A}[u]\|_{L^p(\R^N)}\leq 2\|f\|_{L^p(\R^N)},
$$
where the extension $\widetilde{A}:D(A)\cap L^p(\R^N)\to L^p(\R^N)$ satisfies
\begin{equation}\label{def:DistrOperatorA}
\int_{\R^N}\widetilde{A}[u]\psi\dd x=\int_{\R^N}uA[\psi]\dd x\qquad\text{for all $\psi\in C_\textup{c}^\infty(\R^N)$.}
\end{equation}
\end{lemma}

The proof follows after an immediate consequence.

\begin{corollary}\label{cor:LinearResolventAE}
Assume $\eqref{muas}$, $p\in(1,\infty]$, $f\in (L^1\cap L^\infty)(\R^N)$. Then the equation \eqref{eq:LinearResolvent} (with $A\mapsto \widetilde{A}$) holds a.e. in $\R^N$.
\end{corollary}

\begin{proof}[Proof of Lemma \ref{lem:OperatorEstimateInLp}]
Definition \ref{def:LinearResolventVeryWeak} gives
\begin{equation*}
\begin{split}
\lambda\int_{\R^N}\widetilde{A}[u]\psi\dd x=\lambda\int_{\R^N}uA[\psi]\dd x=\int_{\R^N} (f-u)\psi\dd x.
\end{split}
\end{equation*}
Now, take $q\in[1,\infty)$ such that $p^{-1}+q^{-1}=1$. We recall Theorem 2.14 in \cite{LiLo01},
$$
\|\phi\|_{L^p(\R^N)}=\sup_{\|\psi\|_{L^q(\R^N)}\leq 1}\bigg|\int_{\R^N}\phi(x)\psi(x)\dd x\bigg|,
$$
to obtain, by the H\"older inequality,
\begin{equation*}
\begin{split}
\lambda\|\widetilde{A}[u]\|_{L^p(\R^N)}&=\lambda\sup_{\|\psi\|_{L^q(\R^N)}\leq 1}\bigg|\int_{\R^N}\widetilde{A}[u]\psi\dd x\bigg|=\sup_{\|\psi\|_{L^q(\R^N)}\leq 1}\bigg|\int_{\R^N} (f-u)\psi\dd x\bigg|\\
&\leq \sup_{\|\psi\|_{L^q(\R^N)}\leq 1}\bigg\{\|f-u\|_{L^p(\R^N)}\|\psi\|_{L^q(\R^N)}\bigg\}\leq \|f-u\|_{L^p(\R^N)}\\
&\leq 2\|f\|_{L^p(\R^N)}.
\end{split}
\end{equation*}
The proof is complete.
\end{proof}

Another consequence is that we can also extend $\widetilde{A}$ to $L^1(\R^N)$, and then make sense of the equation in $L^1(\R^N)$. To do so, we follow \cite{BrSt73}.

\begin{corollary}\label{cor:OperatorEstimateInL1}
Assume $\eqref{muas}$ and $f\in (L^1\cap L^\infty)(\R^N)$. Then
$$
\lambda\|\widetilde{A}[u]\|_{L^1(\R^N)}\leq 2\|f\|_{L^1(\R^N)}.
$$
\end{corollary}

\begin{proof}
Since $f\in (L^1\cap L^\infty)(\R^N)\subset L^1(\R^N)$, Lemma \ref{lem:LinearResolventWellposedness}(c) yields
$$
\|u\|_{L^1(\R^N)}\leq \|f\|_{L^1(\R^N)}.
$$
Moreover, by Corollary \ref{cor:LinearResolventAE}, equation \eqref{eq:LinearResolvent} holds pointwise, i.e.,
$$
\|f-\lambda\widetilde{A}[u]\|_{L^1(\R^N)}\leq \|f\|_{L^1(\R^N)}.
$$
The reverse triangle inequality then provides the result.
\end{proof}

We are now ready to prove Theorem \ref{thm:LevyOperatorsmAccretive}\eqref{thm:LevyOperatorsmAccretive3}.

\begin{proposition}\label{eq:LinearResolventRange}
Assume \eqref{muas}. For all $f\in L^1(\R^N)$, there exists a very weak solution $u\in L^1(\R^N)$ of \eqref{eq:LinearResolvent} such that, for all $\lambda>0$,
$$
u+\lambda\widetilde{A}[u]=f\qquad\text{a.e. in $\R^N$,}
$$
where $\widetilde{A}$ is the extension to $L^1(\R^N)$ of $A$ defined through the relation \eqref{def:DistrOperatorA}.
\end{proposition}

\begin{proof}
Take $\{f_n\}_{n\in\N}\subset (L^1\cap L^\infty)(\R^N)$ such that $f_n\to f$ in $L^1(\R^N)$ as $n\to\infty$. By Lemma \ref{lem:LinearResolventWellposedness}(c),
$$
\|u_n-u_m\|_{L^1(\R^N)}\leq \|f_n-f_m\|_{L^1(\R^N)}.
$$
Hence, $\{u_n\}_{n\in\N}$ is Cauchy in $L^1(\R^N)$ and there exists a $u\in L^1(\R^N)$ such that $u_n\to u$ in $L^1(\R^N)$. In a similar way, through Corollary \ref{cor:OperatorEstimateInL1}, $\{\widetilde{A}[u_n]\}_{n\in\N}$ is Cauchy in $L^1(\R^N)$ and there exists a $U\in L^1(\R^N)$ such that $\widetilde{A}[u_n]\to U$ in $L^1(\R^N)$. The definition of $\widetilde{A}$ \eqref{def:DistrOperatorA} then yields
$$
\int_{\R^N}\widetilde{A}[u_n]\psi\dd x=\int_{\R^N}u_nA[\psi]\dd x\qquad\text{for all $\psi\in C_\textup{c}^\infty(\R^N)$}.
$$
Moreover, since $\psi,A[\psi]\in L^\infty(\R^N)$, the $L^1$-convergence gives $U=\widetilde{A}[u]$. Finally, we take the $L^1$-limit in Definition \ref{def:LinearResolventVeryWeak} and use that fact that all the terms of the equation are elements in $L^1\subset L_\textup{loc}^1$.
\end{proof}

\begin{remark}
In the literature, the property
$$
u_n\to u\quad\text{in $L^1(\R^N)$}\qquad\implies\qquad \widetilde{A}[u_n]\to \widetilde{A}[u] \quad\text{in $L^1(\R^N)$}
$$
is referred to as the operator $A$ being \emph{closed} in $L^1(\R^N)$. Here it automatically follows by the symmetry of the operator and that $A:C_\textup{c}^\infty(\R^N)\to L^\infty(\R^N)$.
\end{remark}

Theorem \ref{thm:LevyOperatorsmAccretive}\eqref{thm:LevyOperatorsmAccretive2} is a consequence of the $L^1$-contraction obtained in Lemma \ref{lem:LinearResolventWellposedness}(c) and Proposition \ref{eq:LinearResolventRange} since then
\begin{equation}\label{eq:LinearResolventAccretive}
\|(I+\lambda \widetilde{A})[u]\|_{L^1(\R^N)}=\|f\|_{L^1(\R^N)}\geq \|u\|_{L^1(\R^N)}.
\end{equation}

We are then in the setting of the classical result:

\begin{proposition}[Hille-Yosida/Lumer-Phillips {{\cite[Theorem 4.1.33]{Jac01}}}]\label{prop:LumerPhillips}
A linear operator $(A,D(A))$ on $L^1(\R^N)$ is the generator of a strongly continuous contraction semigroup $(T_t)_{t\geq 0}$ on $L^1(\R^N)$ if and only if $A$ satisfies Theorem \ref{thm:LevyOperatorsmAccretive}\eqref{thm:LevyOperatorsmAccretive1}--\eqref{thm:LevyOperatorsmAccretive3}.
\end{proposition}

We, moreover, have that our operators are \emph{maximal accretive}, i.e.:

\begin{proposition}[{{\cite[Proposition 8.3]{BeCrPa01}}}]\label{prop:ExtensionOfmAccretiveIsTheSame}
If $A$ is $\mathfrak{m}$-accretive, then $\widetilde{A}=A$.
\end{proposition}

\begin{remark}
Hence, a posteriori, $L^1(\R^N)\ni u\mapsto A[u]$ can be identified in a unique way as a limit point in $L^1(\R^N)$. See also Theorem 4.1.40 in \cite{Jac01}.
\end{remark}

Our next task is Theorem \ref{thm:LevyOperatorsmAccretive}\eqref{thm:LevyOperatorsmAccretive4}, which follows by:
\begin{proposition}\label{eq:LinearResolventDirichlet}
Assume \eqref{muas} and $a,b\in\R$ such that $a\leq b$, $a\leq 0$, and $b\geq 0$. For all $f\in L^1(\R^N)$ such that
$$
a\leq f\leq b\qquad\text{a.e.,}
$$
the unique very weak solution $u\in L^1(\R^N)$ of \eqref{eq:LinearResolvent} satisfies
$$
a\leq u\leq b\qquad\text{a.e.}
$$
\end{proposition}

\begin{proof}
By Remark \ref{rem:LinearResolventComparison}, the $T$-contraction holds for $L_\textup{loc}^1$-very weak solutions of \eqref{eq:LinearResolvent}. On one hand, $ f\leq b$ yields $(f-b)^+=0\in L^1(\R^N)$ and then the $T$-contraction gives $u\leq b$. On the other hand,  $ a\leq f$ yields $(a-f)^+=0\in L^1(\R^N)$ and then the $T$-contraction gives $a\leq u$.

Here, we used that $a,b$ are very weak solutions with $a,b$ as right-hand side, and that bounded very weak solutions are unique.
\end{proof}

We have then proven that our operator $A$ satisfies Theorem \ref{thm:LevyOperatorsmAccretive}\eqref{thm:LevyOperatorsmAccretive1}--\eqref{thm:LevyOperatorsmAccretive4}, which is exactly the setting of \cite{CrPi82}.

Finally, let us recall why the above works for our operator $A=(-\Levy^\mu)$. To deduce the $L^1$-contraction---or rather the $T$-contraction---needed to obtain \eqref{eq:LinearResolventAccretive}, we employed a more fundamental result:
\begin{lemma}\label{eq:LinearResolventPositiveOperator}
Assume \eqref{muas}. For all $u\in C_\textup{c}^\infty(\R^N)$,
$$
\int_{\R^N}\widetilde{A}[u]\sgn^+(u)\dd x\geq 0.
$$
\end{lemma}

\begin{remark}
This implies the condition stated as Corollary A.13 in \cite{A-VMaRoT-M10} or in Proposition 4.6.12 in \cite{Jac01} (see also \cite{Sch01} for $p=1$) from which it follows that $A$ is accretive.
\end{remark}

\begin{proof}
Remark \ref{rem:LinearResolventDensityAndRange}(a) gives $\widetilde{A}=A$. By a convex inequality,
$$
A[u]\sgn^+(u)\geq A[(u)^+] \qquad\textup{a.e. in $\R^N$.}
$$
Now, multiply each side by a smooth cut-off function $\mathcal{X}_{R}\in C_\textup{c}^\infty(\R^N)$ satisfying $0\leq \mathcal{X}_R\leq 1$ and $\mathcal{X}_R\to1$ pointwise as $R\to\infty$, integrate, use symmetry, and that $A[\mathcal{X}_R]\to0$ pointwise as $R\to\infty$.
\end{proof}

\begin{corollary}
Assume \eqref{muas}. For all $u\in C_\textup{c}^\infty(\R^N)$,
$$
\int_{\R^N}\widetilde{A}[u]\sgn^+(u-1)\dd x\geq 0.
$$
\end{corollary}

\begin{remark}
Operators satisfying the above are called Dirichlet operators, see Definition 4.6.7 in \cite{Jac01} (and also \cite{Sch01} for $p=1$). Moreover, according to Proposition 4.6.9 in \cite{Jac01} Dirichlet operators imply Theorem \ref{thm:LevyOperatorsmAccretive}\eqref{thm:LevyOperatorsmAccretive4} for both the semigroup and the resolvent of the semigroup generated by the operator, i.e., the semigroup and the resolvent of the semigroup are \emph{sub-Markovian}.
\end{remark}

\begin{proof}
The result can be found as Corollary 3.2 in \cite{Sch01}, which actually provides an equivalence between the two conditions in our setting. Let us include a proof for completeness.

Again, Remark \ref{rem:LinearResolventDensityAndRange}(a) gives $\widetilde{A}=A$.
Now, take $u-\mathcal{X}_R\in C_\textup{c}^\infty(\R^N)$ in Lemma \ref{eq:LinearResolventPositiveOperator}. Since $\sgn^+(u-1)\leq \sgn^+(u-\mathcal{X}_R)$, we have that
\begin{equation*}
\begin{split}
&\int_{\R^N}A[u]\sgn^+(u-\mathcal{X}_R)\dd x\\
&\geq \int_{\R^N}A[\mathcal{X}_R]\sgn^+(u-\mathcal{X}_R)\dd x\geq \int_{\R^N}A[\mathcal{X}_R]\sgn^+(u-1)\dd x.
\end{split}
\end{equation*}
Moreover,
$$
\int_{\R^N}\sgn^+(u-1)\dd x=\int_{\{u>1\}}1\dd x< \int_{\{u>1\}}u\dd x\leq \int_{\R^N}|u|\dd x,
$$
which means that we can, again, use that $A[\mathcal{X}_R]\to0$ pointwise as $R\to\infty$ on the right-hand side. While on the left-hand side we simply use the Lebesgue dominated convergence theorem.
\end{proof}

\begin{remark}\label{rem:ExtensionToAnyMonotoneOperatorAndResolvent}
Lemma \ref{lem:LinearResolventWellposedness} is also true for the general operator (possibly local, nonlocal, or a combination) $\mathfrak{L}^{\sigma,\mu}$ defined by \eqref{def:LevyOperators} with $c=0$, see \cite{DTEnJa17b}. Moreover, if $u$ solves
$$
\veps u+(-\mathfrak{L}^{\sigma,\mu})[u]=f\qquad\text{in $\R^N$ for all $\veps>0$,}
$$
then $u$ solves
$$
(\veps-1) u+(I-\mathfrak{L}^{\sigma,\mu})[u]=f\qquad\text{in $\R^N$ for all $\veps>1$.}
$$
Now, take $u\mapsto\lambda u$ and choose $\veps=1+1/\lambda$ to obtain that $u$ solves
$$
u+\lambda(I-\mathfrak{L}^{\sigma,\mu})[u]=f\qquad\text{in $\R^N$ for all $\lambda>0$.}
$$
Hence, $(-\mathfrak{L}^{\sigma,\mu})$ and $(I-\mathfrak{L}^{\sigma,\mu})$ are also $\mathfrak{m}$-accretive. The latter is exactly \eqref{def:LevyOperators} with $c=1$.

To prove the Dirichlet property, we used that $A[\textup{const}]=0$ in Corollary \ref{eq:LinearResolventDirichlet}. This is of course true for $(-\mathfrak{L}^{\sigma,\mu})$, while for $(I-\mathfrak{L}^{\sigma,\mu})$, we have that $f=\textup{const}$ gives $u=\textup{const}/(1+\lambda)$ as a solution. Arguing as before, however, we still have $f\geq 0$ (resp. $\leq 0$) implies $u\geq 0$ (resp. $\leq 0$). According to Remark \ref{rem:LinearResolventDensityAndRange}, it remains to check $b>0,a<0$ and $a\leq f\leq b$: We readily get $a/(1+\lambda)\leq u\leq b/(1+\lambda)$, i.e., $a\leq u\leq b$ since $\lambda>0$.

We then conclude that both $(-\mathfrak{L}^{\sigma,\mu})$ and $(I-\mathfrak{L}^{\sigma,\mu})$ satisfy Theorem \ref{thm:LevyOperatorsmAccretive}, Corollary \ref{cor:PorousMediummAccretive}, and Proposition \ref{prop:ExtensionOfmAccretiveIsTheSame}. So, indeed the whole class of \emph{symmetric} L\'evy operators with \emph{constant coefficients} are within the above framework.
\end{remark}

\subsection{The setting of very weak solutions}

We will now provide an a priori different approach to the one developed in the previous subsection.

Very weak solutions of
\begin{equation}\label{eq:NonLinearResolvent}
\forall \lambda>0\qquad u+\lambda A[u^m]=f \qquad\text{in $\R^N$}
\end{equation}
can be given as:

\begin{definition}[Very weak solutions]\label{def:NonLinearResolventVeryWeak}
Assume \eqref{muas}. We say that $u\in L_\textup{loc}^1(\R^N)$ is a \emph{very weak solution} of \eqref{eq:NonLinearResolvent} with right-hand side $f\in L_\textup{loc}^1(\R^N)$ if $u^m\in L_\textup{loc}^1(\R^N)$ and
\begin{equation*}
\begin{split}
\int_{\R^N} u\psi\dd x+\lambda\int_{\R^N}u^mA[\psi]\dd x=\int_{\R^N} f\psi\dd x\qquad\text{for all $\psi\in C_\textup{c}^\infty(\R^N)$.}
\end{split}
\end{equation*}
\end{definition}

We collect uniqueness from \cite[Theorem 3.2]{DTEnJa17b}, and a priori estimates from \cite[Remark 5.10]{DTEnJa19}.

\begin{theorem}\label{thm:NonlinearResolventVeryWeakAPriori}
Assume $0\leq f\in (L^1\cap L^\infty)(\R^N)$, \eqref{phias}, \eqref{muas}, and $A=(-\Operator^{\sigma,\mu})$.
\begin{enumerate}[{\rm (a)}]
\item There exists a unique very weak solution $0\leq u\in (L^1\cap L^\infty)(\R^N)$ of \eqref{eq:NonLinearResolvent} with right-hand side $f$.
\item Let $u,v$ be two very weak solutions of \eqref{eq:NonLinearResolvent} with respective right-hand sides $f,g$. Then:
\begin{enumerate}[{\rm (i)}]
\item \textup{(Comparison)} If $f\leq g$, then $u\leq v$.
\item \textup{($L^p$-decay)} $\|u\|_{L^p(\R^N)}\leq \|f\|_{L^p(\R^N)}$ for all $p\in[1,\infty]$.
\end{enumerate}
\end{enumerate}
\end{theorem}

\subsection{Comparison between abstract and very weak solutions}

If $f\in L^1(\R^N)$ only, it is hard to see how to construct very weak solutions of \eqref{eq:NonLinearResolvent}, but as in the abstract setting we could require that $a\leq f\leq b$ for $a,b\in\R$. We in any case have:

\begin{lemma}\label{lem:equivalenceAbstractVeryWeak}
Assume $0\leq f\in (L^1\cap L^\infty)(\R^N)$, \eqref{phias}, and \eqref{muas}. If the operator $A$ has an extension $\widetilde{A}$ to $L^1(\R^N)$, then very weak and a.e.-solutions of \eqref{eq:NonLinearResolvent} coincide.
\end{lemma}

\begin{remark}
Here we discover the advantage of $\mathfrak{m}$-accretive operators: By Proposition \ref{prop:ExtensionOfmAccretiveIsTheSame}, $\widetilde{A}=A$, and thus we obtain an a.e.-equation involving the operator itself! In addition, the abstract setting does not see the difference between  $A=(-\mathfrak{L}^{\sigma,\mu})$ and $A=(I-\mathfrak{L}^{\sigma,\mu})$ since they are both $\mathfrak{m}$-accretive.
\end{remark}

\begin{proof}
By Corollary \ref{cor:PorousMediummAccretive}, we have a.e.-solutions of \eqref{eq:NonLinearResolvent} (with $A\mapsto \widetilde{A}$). Multiplying by $\psi\in C_\textup{c}^\infty(\R^N)$, integrating over $\R^N$, and using the definition of the extension of the operator \eqref{def:DistrOperatorA}, shows that those a.e.-solutions are actually very weak solutions in the sense of Definition \ref{def:NonLinearResolventVeryWeak}. However, we can also start with very weak solutions by Theorem \ref{thm:NonlinearResolventVeryWeakAPriori}, use the $L^1$-extension of the operator \eqref{def:DistrOperatorA}, and see that the equation (with $A\mapsto \widetilde{A}$) actually holds a.e. Hence, the equivalence between the elliptic problems is settled.
\end{proof}

%%%%%%%%%%%%%%%%%%%%%%%%%%%%%%%%%%%%%%%%%%%%%%%%%%%%
%%%%%%%%%%%%%%%%%%%%%NEW SECTION%%%%%%%%%%%%%%%%%%%%%%%
%%%%%%%%%%%%%%%%%%%%%%%%%%%%%%%%%%%%%%%%%%%%%%%%%%%%

\section{The inverse of a densely defined, \texorpdfstring{$\mathfrak{m}$}{m}-accretive, Dirichlet operator}
\label{sec:InverseOfLinearmAccretiveDirichlet}

This section is devoted to showing that a densely defined, $\mathfrak{m}$-accretive, Dirichlet operator in $L^1(\R^N)$ has an inverse such that \eqref{Gas} holds, under (possibly) some additional assumptions on the heat kernel associated with the operator.

Consider a strongly continuous contraction semigroup $(T_t)_{t\geq0}$ in $L^1(\R^N)$ which is moreover sub-Markovian. The discussion before Definition 3.5.17 in \cite{Jac02} gives that the resolvent $(R_\lambda)_{\lambda>0}$ of the semigroup is well-defined for functions in $L^1(\R^N)$, i.e.,
$$
R_\lambda [\psi]:=\int_0^\infty\e^{-\lambda t}T_t[\psi]\dd t<\infty\qquad\text{for all $\psi\in L^1(\R^N)$.}
$$
Moreover, the Green (or potential) operator $G$ associated with $(T_t)_{t\geq0}$ is defined as
$$
G[\psi]:=\lim_{\lambda\to0^+}R_\lambda\psi\qquad\text{for all $0\leq \psi\in L^1(\R^N)$.}
$$

\begin{definition}[Transient]\label{def:SemigroupTransient}
A strongly continuous sub-Markovian contraction semigroup $(T_t)_{t\geq0}$ in $L^1(\R^N)$ which is moreover sub-Markovian is called \emph{transient} if
$$
G[\psi](x)=\int_0^\infty T_t[\psi]\dd t<\infty \qquad\text{for all $0\leq \psi\in L^1(\R^N)$.}
$$
\end{definition}

By Proposition \ref{prop:LumerPhillips}, Theorem \ref{thm:LevyOperatorsmAccretive}, and Remark \ref{rem:ExtensionToAnyMonotoneOperatorAndResolvent}, the operators $(-\mathfrak{L}^{\sigma,\mu})$ and $(I-\mathfrak{L}^{\sigma,\mu})$ generate the respective strongly continuous contraction semigroups $(T_t^{-\Operator^{\sigma,\mu}})_{t\geq0}$ and $(T_t^{I-\Operator^{\sigma,\mu}})_{t\geq0}$ in $L^1(\R^N)$ which are moreover sub-Markovian. Note that by uniqueness of strongly continuous semigroups (cf. Corollary 4.1.35 in \cite{Jac01}), $(T_t^{I-\Operator^{\sigma,\mu}})_{t\geq0}=(\e^{-t}T_t^{-\Operator^{\sigma,\mu}})_{t\geq0}$. An immediate consequence of Example 3.5.30 in \cite{Jac02} is:

\begin{lemma}
The semigroup $(\e^{-t}T_t^{-\Operator^{\sigma,\mu}})_{t\geq0}$ associated with $(I-\mathfrak{L}^{\sigma,\mu})$, which satisfies Proposition \ref{prop:LumerPhillips} and Theorem \ref{thm:LevyOperatorsmAccretive}, is transient.
\end{lemma}

For the operator $(-\mathfrak{L}^{\sigma,\mu})$, we note that the semigroup defined as
$$
T_t[\psi]:=\int_{\R^N}\psi(x-y)\mathbb{H}_t(y)\dd y \qquad\textup{for all $\psi\in L^1(\R^N)$}
$$
is a strongly continuous sub-Markovian contraction semigroup in $L^1(\R^N)$ with $(-\mathfrak{L}^{\sigma,\mu})$ as generator. This can easily be seen since $(-\mathfrak{L}^{\sigma,\mu})$ admits a symmetric and positive heat kernel satisfying $\mathbb{H}_t\in L^1(\R^N)$ due to the fact that the corresponding heat equation enjoys mass conservation, $L^1$-decay, the comparison principle, and has solutions in $C([0,T];L^1(\R^N))$. By Corollary 4.1.35 in \cite{Jac01}, again, this semigroup must coincide with $(T_t^{-\Operator^{\sigma,\mu}})_{t\geq0}$. Moreover:

\begin{lemma}[{{\cite[Theorem 3.5.51]{Jac02}}}]\label{lem:OperatorTransient}
The semigroup $(T_t^{-\Operator^{\sigma,\mu}})_{t\geq0}$ associated with $(-\mathfrak{L}^{\sigma,\mu})$, which satisfies Proposition \ref{prop:LumerPhillips} and Theorem \ref{thm:LevyOperatorsmAccretive}, is transient if and only if, for all compact $K\subset \R^N$,
$$
\int_0^\infty\int_K\mathbb{H}_t(x)\dd x\dd t<\infty.
$$
\end{lemma}

At this point $G$ is a good candidate for $A^{-1}$, but we need additional properties like e.g. $G:L^p\to D(A)$. The rigorous answer can be found in Proposition 3.5.79 in \cite{Jac02}. We simply check that
$$
\lim_{t\to\infty}T_t[\psi]\to0\qquad\text{for all $\psi\in L^1(\R^N)$.}
$$
For convolution semigroups with a symmetric positive kernel $\mathbb{H}_t$, we get
$$
|T_t[\psi](x)|\leq \int_{\R^N}|\psi(x-y)|\mathbb{H}_t(y)\dd y,
$$
and hence, this is really a condition on the kernel by the Lebesgue dominated convergence theorem. Note that we can immediately conclude for $I-\Operator^{\sigma,\mu}$ since $T_t^{I-\Operator^{\sigma,\mu}}[\psi]=\e^{-t}T_t^{-\Operator^{\sigma,\mu}}[\psi]$, $\|T_t^{-\Operator^{\sigma,\mu}}[\psi]\|_{L^1(\R^N)}\leq \|\psi\|_{L^1(\R^N)}$, and hence, we obtain the stronger property
$$
\e^{-t}T_t^{-\Operator^{\sigma,\mu}}[\psi]\to0 \qquad\text{in $L^1(\R^N)$ as $t\to\infty$.}
$$

The above can then be summarized as:

\begin{proposition}[The inverse operator $A^{-1}$]\label{prop:TheInverseOperatorA-1}
\begin{enumerate}[{\rm (a)}]
\item The semigroup $(\e^{-t}T_t^{-\Operator^{\sigma,\mu}})_{t\geq0}$ associated with $(I-\mathfrak{L}^{\sigma,\mu})$, which satisfies Proposition \ref{prop:LumerPhillips} and Theorem \ref{thm:LevyOperatorsmAccretive}, has an inverse operator given, for all $0\leq \psi\in L^1(\R^N)$, by
$$
(I-\mathfrak{L}^{\sigma,\mu})^{-1}[\psi]=\int_0^\infty \e^{-t}T_t^{-\Operator^{\sigma,\mu}}[\psi]\dd t,
$$
and $(I-\mathfrak{L}^{\sigma,\mu})^{-1}[(I-\mathfrak{L}^{\sigma,\mu})[\psi]]=\psi$.
\item Assume that the semigroup $(T_t^{-\Operator^{\sigma,\mu}})_{t\geq0}$ associated with $(-\mathfrak{L}^{\sigma,\mu})$, which satisfies Proposition \ref{prop:LumerPhillips} and Theorem \ref{thm:LevyOperatorsmAccretive}, has a symmetric and positive kernel $\mathbb{H}_t$ satisfying, for all compact $K\subset \R^N$,
$$
\int_0^\infty\int_K\mathbb{H}_t(x)\dd x\dd t<\infty\qquad\text{and}\qquad\lim_{t\to\infty}\mathbb{H}_t(x)\to0\quad\text{for a.e. $x\in \R^N$.}
$$
Then $(-\mathfrak{L}^{\sigma,\mu})$ has an inverse operator given, for all $0\leq \psi\in L^1(\R^N)$, by
$$
(-\mathfrak{L}^{\sigma,\mu})^{-1}[\psi]=\int_0^\infty T_t^{-\Operator^{\sigma,\mu}}[\psi]\dd t=\int_{\R^N}\bigg(\int_0^\infty \mathbb{H}_t(y)\dd t\bigg)\psi(x-y)\dd y
$$
and $(-\mathfrak{L}^{\sigma,\mu})^{-1}[(-\mathfrak{L}^{\sigma,\mu})[\psi]]=\psi$.
\end{enumerate}
\end{proposition}

\begin{remark}
\begin{enumerate}[{\rm (a)}]
\item We immediately see that, in the setting of part (b) above,
$$
\mathbb{G}_{-\mathfrak{L}^{\sigma,\mu}}^{x_0}(x):=\int_0^\infty \mathbb{H}_t^{x_0}(x)\dd t.
$$
Moreover, if we apply the setting of part (b) to part (a), then we also have that
$$
\mathbb{G}_{I-\mathfrak{L}^{\sigma,\mu}}^{x_0}(x):=\int_0^\infty\e^{-t} \mathbb{H}_t^{x_0}(x)\dd t.
$$
Of course, the latter is well-defined for a larger class of kernels than the former.
\item We can then check that all the examples considered in Section \ref{sec:GreenAndHeat} have inverses.
\end{enumerate}
\end{remark}

The positivity assumption of \eqref{Gas} is then very natural since it is related to the fact that the operator ensures the comparison principle.

\begin{corollary}\label{cor:GreenResolventIntegrable}
Under the assumptions of Proposition \ref{prop:TheInverseOperatorA-1},
$$
\mathbb{G}_{-\Operator}^{x_0}(x), \mathbb{G}_{I-\Operator}^{x_0}(x)\geq0 \qquad\text{for a.e. $x\in\R^N$.}
$$
\end{corollary}

\begin{proof}
Since the comparison principle holds for solutions of the heat equation, $\mathbb{H}_{-\Operator}^{x_0}(x,t)\geq0$.
\end{proof}

In fact, we only need to assume that the resolvent is nonnegative since $\e^{-t}\leq 1$:

\begin{corollary}\label{cor:GreenNonnegative}
Under the assumptions of Proposition \ref{prop:TheInverseOperatorA-1},
$$
0\leq \mathbb{G}_{I-\Operator}^{x_0}(x)\leq \mathbb{G}_{-\Operator}^{x_0}(x) \qquad\text{for a.e. $x\in\R^N$.}
$$
\end{corollary}

%%%%%%%%%%%%%%%%%%%%%%%%%%%%%%%%%%%%%%%%%%%%%%%%%%%%
%%%%%%%%%%%%%%%%%%%%%NEW SECTION%%%%%%%%%%%%%%%%%%%%%%%
%%%%%%%%%%%%%%%%%%%%%%%%%%%%%%%%%%%%%%%%%%%%%%%%%%%%

\section{Existence and a priori results for weak dual solutions}
\label{sec:ExistenceAPriori}

Although this is not the main point of the paper, we will illustrate that our assumptions \eqref{G_1}--\eqref{G_3} does not lead to an empty theory. Let us therefore prove Proposition \ref{prop:APriori}. To do so, we rely on the theory of abstract solutions for the corresponding elliptic problem of \eqref{GPME}. Recall Lemma \ref{lem:InverseBounded}, and due to \eqref{Gas}, we also have:

\begin{lemma}
Consider $A:=(-\Operator)$ or $A:=(I-\Operator)$. Then $A[\mathbb{G}_{A}^{y}]=\delta_{y}$ in $\mathcal{D}'(\R^N)$,
$$
A^{-1}[\psi]=\int_{\R^N}\mathbb{G}_{A}^{0}(x-y)\psi(y)\dd y,
$$
and $A[A^{-1}[\psi]]=\psi$ for all $\psi\in C_\textup{c}^\infty(\R^N)$.
\end{lemma}

\begin{proof}[Proof of Proposition \ref{prop:APriori}]
\noindent(a) Consider a uniform grid in time such that $0=t_0<t_1<\cdots<t_J=T$. Let $\mathbb{J}:=\{1,\ldots,J\}$, and denote the time steps by $\Delta t=t_{j}-t_{j-1}$ for all $j\in \mathbb{J}$. The piecewise constant time interpolant $u_{\Delta t}$ is, for $(x,t)\in Q_T$, given as
$$
u_{\Delta t}(x,t):=u_{j}(x)\quad\text{where $t\in(t_{j-1},t_j]$ for all $j\in\mathbb{J}$,}
$$
and $u_{\Delta t}(x,0):=u_0(x)$. Now, each $u_j$ is defined recursively as the solution of the following elliptic equation:
\begin{equation*}
u_j+\Delta tA[u_j^m]=u_{j-1}\qquad\text{in $\R^N$,}
\end{equation*}
which of course is another way of expressing \eqref{eq:NonLinearResolvent}. Since $A$ is densely defined, $\mathfrak{m}$-accretive, and Dirichlet in $L^1(\R^N)$, the above equation has an a.e.-solution (cf. Theorem \ref{thm:LevyOperatorsmAccretive} and Proposition \ref{prop:ExtensionOfmAccretiveIsTheSame}). Then rewriting the equation, multiplying by $A^{-1}[\psi(\cdot,t_{j-1})]$ with $\psi\in C_\textup{c}^\infty(\R^N\times(0,T))$, integrating over $\R^N$, using
$$
\int_{\R^N} A[u_{j}^m]A^{-1}[\psi(\cdot,t_{j-1})]=\int_{\R^N} u_{j}^mA[A^{-1}[\psi(\cdot,t_{j-1})]]=\int_{\R^N} u_{j}^m\psi(\cdot,t_{j-1}),
$$
and summing over $j$, we obtain that
$$
\sum_{j\in \mathbb{J}}\int_{\R^N}\frac{u_j-u_{j-1}}{\Delta t}A^{-1}[\psi(t_{j-1})]\dd x\Delta t+\sum_{j\in \mathbb{J}}\int_{\R^N}u_{j}^m\psi(t_{j-1})\dd x\Delta t=0.
$$
We now perform summation by parts, use the symmetry of $A^{-1}$, and use that $\psi$ has compact support in $(0,T)$ (so that it vanishes for small enough $\Delta t$ for some $n,m\in \mathbb{J}$) to obtain
$$
-\sum_{j=n}^{m-1}\int_{\R^N}A^{-1}[u_{j}]\frac{\psi(t_j)-\psi(t_{j-1})}{\Delta t}\dd x\Delta t+\sum_{j=n}^{m-1}\int_{\R^N}u_{j}^m\psi(t_{j-1})\dd x\Delta t=0.
$$
At this point, we can follow the proof of Proposition 5.2 in \cite{BeBoGrMu21} since, in our case, we have that $A^{-1}[u_{j}]\in C([0,T];L_\textup{loc}^1(\R^N))\cap L^\infty(Q_T)$ and e.g.
$$
|u_j^m-u^m(t_{j})|\leq 2\|u_0\|_{L^\infty(\R^N)}^{m-1}|u_j-u(t_{j})|.
$$
That is, $u\in (L^1\cap L^\infty)(Q_T)\cap C([0,T];L^1(\R^N))$ satisfies
$$
\int_0^T\int_{\R^N}\big(A^{-1}[u]\dell_t\psi-u^m\psi\big)\dd x\dd t=0\qquad\text{for all $\psi\in C_\textup{c}^\infty(Q_T)$.}
$$

Assume $0< \tau_1\leq\tau_2\leq T$, and choose $\psi(x,t)\mapsto \theta_n(t)\psi(x,t)$ where the new $\psi$ is in $C_\textup{c}^1([\tau_1,\tau_2];L_\textup{c}^\infty(\R^N))$ and $\theta_n$ is an approximation of the square pulse with support in $[\tau_1,\tau_2]$. The above expression is still well-defined for this choice, and moreover, since e.g. $A^{-1}[u]\in C([0,T];L_\textup{loc}^1(\R^N))$, we can take the limit as $n\to\infty$. This concludes that $u$ is a weak dual solution according to Definition \ref{def:WeakDualSolution}.

\medskip
\noindent(b) The comparison principle and the $L^p$-decay are immediately inherited from the elliptic problem, see e.g. the proofs of Propositions 5.1 and 5.2 in \cite{BeBoGrMu21}.
\end{proof}

\begin{remark}
We have in fact shown that mild/integral solutions (i.e., the limit points of the time-discretized problem) of \eqref{GPME} are weak dual solutions.
\end{remark}

%%%%%%%%%%%%%%%%%%%%%%%%%%%%%%%%%%%%%%%%%%%%%%%%%%%%
%%%%%%%%%%%%%%%%%%%%%REFERENCES%%%%%%%%%%%%%%%%%%%%%%%
%%%%%%%%%%%%%%%%%%%%%%%%%%%%%%%%%%%%%%%%%%%%%%%%%%%%

\end{document}